\documentclass{amsart}
\usepackage{amsmath,amssymb}
\usepackage{amsthm}
\usepackage{latexsym}
\usepackage[square,sort,comma,numbers]{natbib}

\newcommand{\vX}{\mathtt{X}}
\newcommand{\bvX}{\bar{\mathtt{X}}}
\newcommand{\vS}{\mathsf{S}}
\newcommand{\vD}{\mathsf{D}}

\newcommand{\vP}{\mathsf{S}}
\newcommand{\sfp}{\mathsf{p}}

\newcommand{\vI}{\mathsf{I}}

\newcommand{\vu}{\boldsymbol{u}}
\newcommand{\vv}{\boldsymbol{v}}
\newcommand{\vw}{\boldsymbol{w}}

\newcommand{\myP}{/\hspace{-3pt}/}

\newcommand{\NN}{\mathbb{N}}
\newcommand{\bP}{\mathbb{P}}
\newcommand{\vE}{\mathbb{E}}
\newcommand{\vV}{\mathbb{V}}
\newcommand{\vG}{\mathbb{G}}
\newcommand{\ZZ}{\mathbb{Z}}

\newcommand{\RR}{\mathbb{R}}
\newcommand{\CC}{\mathbb{C}}
\newcommand{\EE}{\mathbb{E}}

\newcommand{\MM}{\mathsf{M}}

\newcommand{\tr}{\mbox{tr}}
\newcommand{\Ric}{\mbox{Ric}}

\newcommand{\ud}{\textup{d}}

\newcommand{\argmin}{\operatornamewithlimits{argmin}}

\theoremstyle{definition}
\newtheorem{theorem}{Theorem}[section]
\newtheorem{definition}[theorem]{Definition}

\newtheorem{lemma}[theorem]{Lemma}
\newtheorem{thm}{Theorem}[section]

\newtheorem{corollary}[thm]{Corollary}

\theoremstyle{remark}
\newtheorem*{remark}{Remark}
\newtheorem*{example}{Example}
\newtheorem{assumption}[thm]{Assumption}

\title[Spectral Convergence of the connection Laplacian]{Spectral Convergence of the connection Laplacian from random samples}

\author[A.~Singer]{A.~Singer}
\address{Department of Mathematics and Program in Applied and Computational Mathematics, Princeton University} 
\email{amits@math.princeton.edu}
\author[H.-T.~Wu]{H.-T.~Wu}
\address{University of Toronto, Department of Mathematics}
\email{hauwu@math.toronto.edu}

\begin{document}

\maketitle

\begin{abstract}
Spectral methods that are based on eigenvectors and eigenvalues of discrete graph Laplacians, such as Diffusion Maps and Laplacian Eigenmaps are often used for manifold learning and non-linear dimensionality reduction. It was previously shown by Belkin and Niyogi \cite{belkin_niyogi:2007} that the eigenvectors and eigenvalues of the graph Laplacian converge to the eigenfunctions and eigenvalues of the Laplace-Beltrami operator of the manifold in the limit of infinitely many data points sampled independently from the uniform distribution over the manifold. Recently, we introduced Vector Diffusion Maps and showed that the connection Laplacian of the tangent bundle of the manifold can be approximated from random samples. In this paper, we present a unified framework for approximating other connection Laplacians over the manifold by considering its principle bundle structure. We prove that the eigenvectors and eigenvalues of these Laplacians converge in the limit of infinitely many independent random samples. We generalize the spectral convergence results to the case where the data points are sampled from a non-uniform distribution, and for manifolds with and without boundary.
\end{abstract}

\section{Introduction}

A recurring problem in fields such as neuroscience, computer graphics and image processing is that of organizing a set of $3$-dim objects by pairwise comparisons. For example, the objects can be $3$-dim brain functional magnetic resonance imaging (fMRI) images \cite{Huettel_Song_McCarthy:2008} that correspond to similar functional activity. In order to separate the actual sources of variability among the images from the nuisance parameters that correspond to different conditions of the acquisition process, the images are initially registered and aligned. Similarly, the shape space analysis problem in computer graphics \cite{Ovsjanikov_Ben-Chen_Solomon_Butscher_Guibas:2012} involves the organization of a collection of shapes. Also in this problem it is desired to factor out nuisance shape deformations, such as rigid transformations.

Once the nuisance parameters have been factored out, methods such as Diffusion Maps (DM) \cite{coifman_lafon:2006} or Laplacian Eigenmaps (LE) \cite{belkin_niyogi:2003} can be used for non-linear dimensionality reduction, classification and clustering. In \cite{singer_wu:2012}, we introduced Vector Diffusion Maps (VDM) as an algorithmic framework for organization of such data sets that simultaneously takes into account the nuisance parameters and the data affinities by a single computation of the eigenvectors and eigenvalues of the graph connection Laplacian (GCL) that encodes both types of information. In \cite{singer_wu:2012}, we also proved pointwise convergence of the GCL to the connection Laplacian of the tangent bundle of the data manifold in the limit of infinitely many sample points. The main contribution of the current paper is a proof for the spectral convergence of the CGL to the connection Laplacian operator over the vector bundle of the data manifold. In passing, we also provide a spectral convergence result for the graph Laplacian (normalized properly)  to the Laplace-Beltrami operator in the case of non-uniform sampling and for manifolds with non-empty boundary, thus broadening the scope of a previous result of Belkin and Niyogi \cite{belkin_niyogi:2007}.

At the center of LE, DM, and VDM is a weighted undirected graph, whose vertices correspond to the data objects and the weights quantify the affinities between them.
A commonly used metric is the Euclidean distance, and the affinity can then be described using a kernel function of the distance. For example, if the data set $\{x_1,x_2,\ldots,x_n\}$ consists of $n$ functions in $L_2(\mathbb{R}^3)$ then the distances are given by 
\begin{equation}
d_E(x_i,x_j):=\|x_i-x_j\|_{L^2(\mathbb{R}^3)},
\end{equation}
and the weights can be defined using the Gaussian kernel with width $\sqrt{h}$ as
\begin{equation}
\label{eq:wij-euclidean}
w_{ij} = e^{-\frac{d_E^2(x_i,x_j)}{2h}}.
\end{equation}

However, the Euclidean distance is sensitive to the nuisance parameters. In order to factor out the nuisance parameters, it is required to use a metric which is invariant to the group of transformations associated with those parameters, denoted by $G$. Let $\mathcal{X}$ be the total space from which data is sampled. The group $G$ acts on $\mathcal{X}$ and instead of measuring distances between elements of $\mathcal{X}$, we want to measure distances between their orbits. The orbit of a point $x\in \mathcal{X}$ is the set of elements of $\mathcal{X}$ to which $x$ can be mapped by the elements of $G$, denoted by $$Gx = \left\{g\circ x \,|\, g\in G \right\}$$ The group action induces an equivalence relation on $\mathcal{X}$ and the orbits are the equivalence classes, such that the equivalence class $[x]$ of $x\in \mathcal{X}$ is $Gx$. The invariant metric is a metric on the orbit space $\mathcal{X}/G$ of equivalent classes.

One possible way of constructing the invariant metric $d_G$ is through optimal alignment, given as
\begin{equation}
d_G([x_i],[x_j]) = \inf_{g_i,g_j\in G} d_E(g_i\circ x_i, g_j\circ x_j).
\end{equation}
If the action of the group is an isometry, then
\begin{equation}
\label{eq:inv-metric}
d_G([x_i],[x_j]) = \inf_{g\in G} d_E(x_i, g\circ x_j).
\end{equation}
For example, if $\mathcal{X}=L^2(\mathbb{R}^3)$ and $G$ is $O(3)$ (the group of $3\times 3$ orthogonal matrices), then the left action 
\begin{equation}
(g\circ f)(x) = f(g^{-1}x)
\end{equation}
is an isometry, and 
\begin{equation}
\label{eq:min-g}
d_G^2([f_i],[f_j]) = \min_{g\in O(3)} \int_{\mathbb{R}^3} |f_i(x) - f_j(g^{-1}x)|^2\,\ud x.
\end{equation}

In this paper we only consider groups that are either orthogonal and unitary, for three reasons. First, this condition guarantees that the GCL is symmetric (or Hermitian). Second, the action is an isometry and the invariant metric (\ref{eq:inv-metric}) is well defined. Third, it is a compact group and the minimizer of (\ref{eq:min-g}) is well defined.

The invariant metric $d_G$ can be used to define weights between data samples, for example, the Gaussian kernel gives
\begin{equation}
\label{eq:wij-invariant}
w_{ij} = e^{-\frac{d_G^2([x_i],[x_j])}{2h}}.
\end{equation}
While LE and DM with weights given in (\ref{eq:wij-euclidean}) correspond to diffusion over the original space $\mathcal{X}$, LE and DM with weights given in (\ref{eq:wij-invariant}) correspond to diffusion over the orbit space $\mathcal{X}/G$. In VDM, the weights (\ref{eq:wij-invariant}) are also accompanied by the optimal transformations
\begin{equation}
g_{ij} = \argmin_{g\in G} d_E(x_i, g\circ x_j).
\end{equation}
VDM corresponds to diffusion over the vector bundle of the orbit space $\mathcal{X}/G$ associated with the group action. The following existing examples demonstrate the usefulness of such a diffusion process in data analysis:
\begin{itemize}
\item {\em Manifold learning:} Suppose we are given a point cloud randomly sampled from a $d$-dim smooth manifold $\MM$ embedded in $\RR^p$. Due to the smoothness of $\MM$, the embedded tangent bundle of $\MM$ can be estimated by local principal component analysis \cite{singer_wu:2012}. All bases of an embedded tangent plane at $x$ form a group isomorphic to $O(d)$. Since the bases of the embedded tangent planes form the frame bundle $O(\MM)$, from this point cloud we obtain a set of samples from the frame bundle which form the total space $\mathcal{X}=O(\MM)$. Since the set of all the bases of an embedded tangent plane is invariant under the action of $O(d)$, for the purpose of learning the manifold $\MM$, we take $O(d)$ as the nuisance group, and hence the orbit space is $\MM=O(\MM)/O(d)$. As shown in \cite{singer_wu:2012}, the generator of the diffusion process corresponding to VDM is the connection Laplacian associated with the tangent bundle. With the eigenvalues and eigenvectors of the connection Laplacian, the point cloud is embedded in an Euclidean space. We refer to the Euclidean distance in the embedded space as the {\it vector diffusion distance (VDD)}, which provides a metric for the point cloud. It is shown in \cite{singer_wu:2012} that VDD approximates the geodesic distance between nearby points on the manifold. Furthermore, by VDM, we extend the earlier spectral embedding theorem \cite{berard_besson_gallot:1994} by constructing a distance in a class of closed Riemannian manifolds with prescribed geometric conditions, which leads to a pre-compactness theorem on the class under consideration \cite{Wu:2013}.
\item {\em Orientability:} Suppose we are given a point cloud randomly sampled from a $d$-dim smooth manifold $\MM$ and we want to learn its orientability. Since the frame bundle encodes whether or not the manifold is orientable, we take the nuisance group as $\ZZ_2$ defined as the determinant of the action $O(d)$ from the previous example. In other words, the orbit of each point on the manifold is $\ZZ_2$, the total space $\mathcal{X}$ is the $\ZZ_2$ bundle on $\MM$ following the orientation, and the orbit space is $\MM$. With the nuisance group $\ZZ_2$, {\it Orientable Diffusion Maps} (ODM) proposed in \cite{singer_wu:2011} can be considered as a variation of VDM in order to estimate the orientability of $\MM$ from a finite collection of random samples.
\item {\em Cryo-EM:} The X-ray transform often serves as a mathematical model to many medical and biological imaging modalities, for example, in cryo-electron microscopy \cite{frank:2006}. In cryo-electron microscopy, the 2-dim projection images of the 3-dim object are noisy and their projection directions are unknown. For the purpose of denoising, it is required to classify the images and average images with similar projection directions, a procedure known as class averaging. When the object of interest has no symmetry, the projection images have a one-to-one correspondence with a manifold diffeomorphic to $SO(3)$. Notice that $SO(3)$ can be viewed as the set of all right-handed bases of all tangent planes to $S^2$, and the set of all right-handed bases of a tangent plane is isomorphic to $SO(2)$. Since the projection directions are parameterized by $S^2$ and the set of images with the same projection direction is invariant under the $SO(2)$ action, we learn the projection direction by taking $SO(2)$ as the nuisance group and $S^2$ as the orbit space. The vector diffusion distance provides a metric for classification of the projection directions in $S^2$, and this metric has been shown to outperform other classification methods \cite{singer_zhao_shkolnisky_hadani:2011,Hadani_Singer:2011a,Zhao_Singer:2014}.
\end{itemize}

The main contribution of this paper is twofold. First, we use the mathematical framework of the {\it principal bundle} \cite{bishop} in order to analyze the relationship between the nuisance group and the orbit space and how their combination can be used to learn the dataset. In this setup, the total space is the principal bundle, the orbit space is the base manifold, and the orbit is the fiber. This principal bundle framework unifies LE, DM, ODM, and VDM by providing a common mathematical language to all of them. Second, for data points that are independently sampled from the uniform distribution over a manifold, in addition to showing pointwise convergence of VDM in the general principal bundle setup, in Theorem \ref{thm:spectralconvergence_heatkernel} we prove that the algorithm converges in the spectral sense, that is, the eigenvalues and the eigenvectors computed by the algorithm converge to the eigenvalues and the eigen-vector-fields of the connection Laplacian of the associated vector bundle. Our pointwise and spectral convergence results also hold for manifolds with boundary and in the case where data points are sampled independently from non-uniform distributions (that satisfy mild technical conditions). We also show spectral convergence of the GCL to the connection Laplacian of the associated tangent bundle in Theorem \ref{thm:spec_conv_pointcloud_heatkernel} when the tangent bundle is estimated from the point cloud. The importance of these spectral convergence results stem from the fact that they provide a theoretical guarantee  in the limit of infinite number of data samples for the above listed problems, namely, estimating vector diffusion distances, determining the orientability of a manifold from a point cloud, and classifying the projection directions of cryo-EM images. In addition, we show that ODM can help reconstruct the orientable double covering of non-orientable manifolds 
by proving a symmetric version of Nash's isometric embedding theorem \cite{Nash1954,Nash1956}.

The rest of the paper is organized as follows. In Section \ref{section:review}, we review VDM and clarify the relationship between the point cloud sampled from the manifold and the bundle structure of the manifold. In Section \ref{section:setup}, we introduce background material and set up the notations. In Section \ref{subsection:tangent_bundle}, we unify LE, DM, VDM, and ODM by taking the principal bundle structure of the manifold into account. In section \ref{proof:PointwiseSpectralConvergence} we prove the first spectral convergence result that assumes knowledge of the bundle structure. The non-empty boundary and nonuniform sampling effects are simultaneously handled. In Section \ref{section:extract_more_info}, we prove the second spectral convergence result when the bundle information is missing and needs to be estimated directly from a finite random point cloud.

\section{The Graph Connection Laplacian and Vector Diffusion Maps}\label{section:review}
Consider an undirected affinity graph $\vG=(\vV,\vE)$, where $\vV=\{x_i\}_{i=1}^{n}$ and fix a $q\in\NN$. Suppose each edge $(i,j)\in \vE$ is assigned a scalar value $w_{ij}>0$ and a group element $g_{ij}\in O(q)$. We call $w_{ij}$ the {\it affinity} between $x_i$ and $x_j$ and $g_{ij}$ the {\it connection group} between {\it vector status} of $x_i$ and $x_j$. We assume that $w_{ij}=w_{ji}$ and $g_{ij}^T = g_{ji}$.
Construct the following $n\times n$ block matrix $\vS_n$ with $q\times q$ entries:
\begin{equation}\label{Svdm}
\vS_n(i,j) = \left\{\begin{array}{ccc}
                  w_{ij}g_{ij} &  & (i,j)\in \vE, \\
                  0_{d\times d} &  & (i,j)\notin \vE.
                \end{array}
 \right.
\end{equation}
Notice that the square matrix $\vS_n$ is symmetric due to the assumption of $w_{ij}$ and $g_{ij}$. 
Define 
$$
d_i = \sum_{(i,j)\in\vE} w_{ij}
$$ 
as the weighted degree of node $i$.
Then define a $n\times n$ diagonal block matrix $\vD_n$ with $q\times q$ entries, where the diagonal blocks are scalar multiples of the identity given by
\begin{equation}\label{Dvdm}
\vD_n(i,i) = d_i\vI_{q},
\end{equation}
where $\vI_q$ is the $q\times q$ identity matrix. The {\it un-normalized GCL} and the {\it normalized GCL} are defined in \cite{singer_wu:2012,bandeira_singer_spielman:2013}
\[
L_n:=\vD_n-\vS_n,\quad \mathcal{L}_n:=\vI_{qn}-\vD_{n}^{-1}\vS_n 
\]
respectively.
Given a $\vv\in \RR^{qn}$, we denote $\vv[l]\in \RR^{q}$ to be the $l$-th component in the vector by saying that $\vv[l]=[\vv((l-1)q+1),\ldots, \vv(lq)]^T\in\RR^{q}$ for all $l=1,\ldots,n$.
The matrix $\vD_n^{-1}\vS_n$ is thus an operator acting on $\boldsymbol{v}\in \RR^{nq}$ by
\begin{align}\label{Cmatrix}
(\vD_n^{-1}\vS_n\boldsymbol{v})[i] = \frac{\sum_{j: (i,j)\in \vE} w_{ij}g_{ij}v[j]}{d_i},
\end{align} 
which suggests the interpretation of $\vD_n^{-1}\vS_n$ as a {\it generalized Markov chain} in the following sense so that the random walker (e.g., diffusive particle) is characterized by a generalized status vector. Indeed, a particle at $i$ is endowed with a $q$-dim vector status, and at each time step it hops from $i$ to $j$ with probability $w_{ij}/d_i$. In the absence of the group, these statuses are separately viewed as $q$ functions defined on $\mathbb{G}$. Notice that the graph Laplacian arises as a special case for  $q=1$ and $g_{ij}=1$.
However, when $q>1$ and $g_{ij}$ are not identity matrices, in general the coordinates of the status vectors do not decouple into $q$ independent processes due to the non-trivial effect of the group elements $g_{ij}$.
Thus, if a particle with status $v[i]\in\RR^{q}$ moves along a path of length $t$ from $j_0$ to $j_t$ containing vertices $j_0,{j_1},\ldots,{j_{t-1}},j_t$ so that $({j_l},{j_{l+1}})\in \vE$ for $0=1,\ldots,t-1$, in the end it becomes
\begin{equation*}
g_{j,j_{t-1}}\cdots g_{j_2,j_1}g_{j_1,i}v[i].
\end{equation*}
That is, when the particle arrives $j$, its vector status is influenced by a series of rotations along the path from $i$ to $j$.
In case there are more than two paths from $i$ to $j$ and the rotational groups on paths vary dramatically, we may get cancelation while adding transformations of different paths. Intuitively, ``the closer two points are'' or ``the less variance of the translational group on the paths is'', the more consistent the vector statuses are between $i$ and $j$. We can thus define a new affinity between $i$ and $j$ by the consistency between the vector statuses. Notice that the matrix $(\vD_n^{-1}\vS_n)^{2t}(i,j)$, where $t>0$, contains the average of the rotational information over all paths of length $2t$ from $i$ to $j$. Thus, the squared Hilbert-Schmidt norm, $\|(\vD_n^{-1}\vS_n)^{2t}(i,j)\|^2_{HS}$, can be viewed as a measure of not only the number of paths of length $2t$ from $i$ to $j$ but also the amount of consistency of the vector statuses that propagated along different paths connecting $i$ and $j$. This motivates to define the affinity  between $i$ and $j$ as $\|(\vD_n^{-1}\vS_n)^{2t}(i,j)\|^2_{HS}$.

To understand this affinity, we consider the symmetric matrix $\widetilde{\vS}_n = \vD_n^{-1/2}\vS_n\vD_n^{-1/2}$ which is similar to $\vD_n^{-1}\vS_n$. Since $\widetilde{\vS}_n$ is symmetric, it has a complete set of eigenvectors $\vv_1,\vv_2,\ldots,\vv_{nq}$ and eigenvalues $\lambda_1, \lambda_2, \ldots, \lambda_{nq}$, where the eigenvalues are the same as those of $\vD_n^{-1}\vS_n$. Order the eigenvalues so that $\lambda_1  \geq  \lambda_2  \geq \ldots \geq  \lambda_{nq} $. A direct calculation of the HS norm of $\widetilde{\vS}_n^{2t}(i,j)$ leads to:
\begin{equation}\label{HS-norm}
\|\widetilde{\vS}_n^{2t}(i,j)\|^2_{HS} = \sum_{l,r=1}^{nq} (\lambda_l \lambda_r)^{2t} \langle \vv_l[i], \vv_r[i] \rangle \langle \vv_l[j], \vv_r[j] \rangle.
\end{equation}
The {\it  vector diffusion map} (VDM) $V_t$ is defined as the following map from $\vG$ to $\RR^{(nq)^2}$:
\begin{equation*}
V_t: i \mapsto \left( (\lambda_l \lambda_r)^{t} \langle \vv_l[i], \vv_r[i] \rangle\right)_{l,r=1}^{nq}.
\end{equation*}
With this map, $\|\widetilde{\vS}_n^{2t}(i,j)\|^2_{HS}$ becomes an inner product for a finite dimensional Hilbert space, that is,
\begin{equation*}
\|\widetilde{\vS}_n^{2t}(i,j)\|^2_{HS} = \langle V_t(i), V_t(j) \rangle.
\end{equation*}
The {\it vector diffusion distance} (VDD) between nodes $i$ and $j$ is defined as  
\[
d^{\textup{(VDD)}}_{t}(i,j):=\| V_t(i)- V_t(j)\|^2.
\] 
Furthermore, $|\lambda_l|\leq 1$ due to the following identity:
\begin{equation}\label{fact:MayBeNegative}
\vv^T (\vI_n \pm \widetilde{\vS}_n) \vv = \sum_{(i,j)\in \EE} \left\|\frac{\vv[i]}{\sqrt{d_i}} \pm \frac{w_{ij}g_{ij}\vv[j]}{\sqrt{d_j}} \right\|^2 \geq 0,
\end{equation}
for any $\vv\in \mathbb{R}^{nq}$. By the above we cannot guarantee that the eigenvalues of $\widetilde{\vS}_n$ are non-negative, and that is the main reason we 
define $V_t$ through $\|\widetilde{\vS}_n^{2t}(i,j)\|^2_{HS}$ rather than $\|\widetilde{\vS}_n^{t}(i,j)\|^2_{HS}$. On the other hand, we know that the unnormalized GCL is positive semi-definite because 
\[
\vv^T(\vD_n-\vS_n)\vv=\sum_{(i,j)\in \vE}w_{ij}\|g_{ij}\vv[j]-\vv[i]\|^2\geq 0.
\] 

We now come back to $\vD^{-1}_n\vS_n$. The eigenvector of $\vD^{-1}_n\vS_n$ associated with eigenvalue $\lambda_l$ is $\vw_l=\vD_n^{-1/2}\vv_l$. 
This motivates the definition of another VDM from $\vG$ to $\RR^{(nq)^2}$ as
\begin{equation*}
V'_t: i \mapsto \left( (\lambda_l \lambda_r)^{t} \langle \vw_l[i], \vw_r[i] \rangle\right)_{l,r=1}^{nq},
\end{equation*}
so that $V'_t(i) = \frac{V_t(i)}{d_i}$. In other words, $V'_t$ maps the data set in a Hilbert space upon proper normalization by the vertex degrees. The associated VDD is thus defined as $\| V'_t(i)- V'_t(j) \|^2$.
For further discussion of the motivation about VDM, VDD and other normalizations, please refer to \cite{singer_wu:2012}.

\section{Notations, Background and Assumptions}\label{section:setup}
In this section, we collect all notations and background facts about differential geometry needed in throughout the paper. 

\subsection{Notations and Background of Differential Geometry}
We refer the readers who are not familiar with the principal bundle structure to Appendix \ref{section:appendix:principal_bundle} for a quick introduction and \cite{bishop,Berline_Getzler_Vergne:2004} for a general treatment. 

Denote $\MM$ to be a $d$-dim compact smooth manifold. If the boundary $\partial\MM$ is non-empty, it is smooth. Denote $\iota:\MM\hookrightarrow\RR^p$ to be a smooth embedding of $\MM$ into $\RR^p$ and equip $\MM$ with the metric $g$ induced from the canonical metric on $\RR^p$ via $\iota$. With the metric $g$ we have an induced measure, denoted as $\ud V$, on $\MM$. Denote
$$
\MM_{t}:=\{x\in\MM:~\min_{y\in\partial\MM}d(x,y)\leq t\},
$$
where $t\geq 0$ and $d(x,y)$ is the geodesic distance between $x$ and $y$.

Denote $P(\MM, G)$ to be the principal bundle with a connection $1$-form $\omega$, where $G$ is a Lie group right acting on $P(\MM,G)$ by $\circ$. Denote $\pi:P(\MM,G)\to \MM$ to be the canonical projection. We call $\MM$ the base space of the principal $G$ bundle and $G$ the structure group or the fiber of the principal bundle.
From the view point of orbit space, $P(\MM,G)$ is the total space, $G$ is the group acting on $P(\MM,G)$, and $\MM$ is the orbit space of $P(\MM,G)$ under the action of $G$. In other words, when our interest is the parametrization of the orbit space, $G$ becomes the nuisance group. 

Denote $\rho$ to be a representation of $G$ into $O(q)$, where $q>0$\footnote{We may also consider representing $G$ into $U(q)$ if we take the fiber to be $\CC^{q}$. However, to simplify the discussion, we focus ourselves on $O(q)$ and the real vector space.}. 
When there is no danger of confusion, we use the same symbol $g$ to denote the Riemannian metric on $\MM$ and an element of $G$. Denote $\mathcal{E}(P(\MM,G),\rho,\RR^{q})$, $q\geq 1$, to be the associated vector bundle with the fiber diffeomorphic to $\RR^{q}$. By definition, $\mathcal{E}(P(\MM,G),\rho,\RR^{q})$ is the quotient space $P(\MM,G)\times\RR^{q}/\sim$, where the equivalence relationship $\sim$ is defined by the group action on $P(\MM,G)\times \RR^{q}$ by $g:(u,v)\to (g\circ u ,\rho(g)^{-1}v)$, where $g\in G$, $u\in P(\MM,G)$ and $v\in \RR^{q}$. When there is no danger of confusion, we use $\mathcal{E}$ to simplify the notation. Denote $\pi_{\mathcal{E}}$ to be the associated canonical projection and $E_x$ to be the fiber of $\mathcal{E}$ on $x\in\MM$; that is, $E_x:=\pi_{\mathcal{E}}^{-1}(x)$. Given a fiber metric $g^{\mathcal{E}}$ in $\mathcal{E}$, which always exists since $\MM$ is compact, we consider the metric connection under which the parallel displacement of fiber of $\mathcal{E}$ is isometric related to $g^{\mathcal{E}}$. The metric connection on $\mathcal{E}$ determined from $\omega$ is denoted as $\nabla^{\mathcal{E}}$. Note that by definition, each $u\in P(\MM,G)$ turns out to be a linear mapping from $\RR^{q}$ to $E_x$ preserving the inner product structure, where $x=\pi(u)$, and satisfies 
\[
(g\circ u)v=u(\rho(g)v)\in E_x, 
\]
where $u\in P(\MM,G)$, $g\in G$ and $v\in \RR^{q}$. We interpret the linear mapping $u$ as finding the point $u(v)\in E_x$ possessing the coordinate $v\in \RR^{q}$. 

\begin{example}
An important example is the frame bundle of the Riemannian manifold $(\MM,g)$, denoted as $O(\MM)=P(\MM,O(d))$, and the tangent bundle $T\MM$, which is the associated vector bundle of the frame bundle $O(\MM)$ if we take $\rho=\textit{id}$ and $q=d$. The relationship among the principal bundle and its associated vector bundle can be better understood by considering the practical meaning of the relationship between the frame bundle and its associated tangent bundle. It is actually the change of coordinate (or change of variable linearly). In fact, if we view a point $u\in O(\MM)$ as the basis of the fiber $T_x\MM$, where $x=\pi(u)$, then the coordinate of a point on the tangent plane $T_x\MM$ changes, that is, $v\to g^{-1}v$, according to the changes of the basis, that is, $u\to g\circ u$, where $g\in O(d)$.
\end{example}

Denote $\Gamma(\mathcal{E})$ to be the set of sections, $C^k(\mathcal{E})$ to be the set of $k$-th differentiable sections, where $k\geq 0$. Also denote $C(\mathcal{E}):=C^0(\mathcal{E})$ to be the set of continuous sections. Denote $L^p(\mathcal{E})$, $1\leq p< \infty$ to be the set of $L^p$ integrable sections, that is, $X\in L^p(\mathcal{E})$ iff $\int |g^{\mathcal{E}} (X  ,X )|^{p/2}\ud V <\infty$. Denote $\|X\|_{L^\infty}$ to be the $L^\infty$ norm of $X$.

The covariant derivative $\nabla^{\mathcal{E}}$ of $X\in C^1(\mathcal{E})$ in the direction $v$ at $x$ is defined as
\begin{equation}\label{definition:parallel:frameX}
\nabla^{\mathcal{E}}_{\dot{c}(0)} X(x)=\lim_{h\to 0}\frac{1}{h}[u(0) u(h)^{-1}(X(c(h)))-X(c(0))],
\end{equation}
where $c:[0,1]\to \MM$ is the curve on $\MM$ so that $c(0)=x$, $\dot{c}(0)=v$ and $u(h)$ is the horizontal lift of $c(h)$ to $P(\MM,G)$ so that $\pi(u(0))=x$. Let $\myP^{x}_{y}$ denote the parallel displacement from $y$ to $x$. When $y$ is in the cut locus of $x$, we set $\myP^x_yX(y)=0$ ; when $h$ is small enough, $\myP^{c(0)}_{c(h)}=u(0) u(h)^{-1}$ by definition. For a smooth section $X$, denote $X^{(l)}$, $l\in\NN$, to be the $l$-th order covariant derivatives of $X$. 

\begin{example}
We can better understand this definition in the frame bundle $O(\MM)$ and its associated tangent bundle. Take $X\in C^1(T\MM)$. The practical meaning of (\ref{definition:parallel:frameX}) is the following: find the coordinate of $X(c(h))$ by $u(h)^{-1}(X(c(h)))$, then view this coordinate to be associated with $T_x\MM$, and map it back to the fiber $T_x\MM$ by the basis $u(0)$. In this way we can compare two different ``abstract fibers'' by comparing their coordinates.  
\end{example}

Denote $\nabla^2$ the connection Laplacian over $\MM$ with respect to $\mathcal{E}$. Denote by $\mathcal{R}$, $\Ric$, and $s$ the Riemanian curvature tensor, the Ricci curvature, and the scalar curvature of $(\MM,g)$, respectively. The second fundamental form of the embedding $\iota$ is denoted by $\textup{II}$.
Denote $\tau$ to be the largest positive number having the property: the open normal bundle about $\MM$ of radius $r$ is embedded in $\RR^p$ for every $r<\tau$ \cite{NSW}. Note that $1/\tau$ can be interpreted as the condition number of the manifold. Since $\MM$ is compact, $\tau>0$ holds automatically . 
Denote $\text{inj}(\MM)$ to be the injectivity radius of $\MM$.

\subsection{Notations and Background of Numerical Finite Samples}
When the range of a random vector $Y$ is supported on a $d$ dimensional manifold $\MM$ embedded in $\RR^p$ via $\iota$, where $d<p$, the notion of probability density function (p.d.f.) may not be defined. It is possible to discuss more general setups, but we restrict ourselves here to the following definition for the sake of the asymptotic analysis \cite{cheng_wu:2012}. Let the random vector $Y:(\Omega,\mathcal{F},\ud P) \rightarrow \RR^p$ be a measurable function defined on the probability space $\Omega$. Let $\tilde{\mathcal{B}}$ be the Borel sigma algebra of $\iota(\MM)$. Denote by $\ud\tilde{P}_Y$ the probability measure of $Y$, defined on $\tilde{\mathcal{B}}$, induced from the probability measure $\ud P$. Assume that $\ud\tilde{P}_{Y}$ is absolutely continuous with respect to the volume measure on $\iota(\MM)$, that is, $\ud \tilde{P}_{Y}(x)=\sfp(\iota^{-1}(x))\iota_*\ud V(x)$.
\begin{definition}
We call $\sfp:\MM\to \RR_+$ the p.d.f. of the $p$-dimensional random vector $Y$ when its range is supported on a $d$ dimensional manifold $\iota(\MM)$, where $d<p$. When $\sfp$ is constant, we say the sampling is {\it uniform}; otherwise {\it non-uniform}.
\end{definition} 
From now on we assume $\sfp\in C^4(\MM)$. With this definition, we can thus define the expectation. For example, if $f:\iota(\MM)\rightarrow \RR$ is an integrable function, we have
\begin{align}
 \EE f(Y)=&\int_{\Omega}f(Y(\omega))\ud P(\omega)=\int_{\iota(\MM)} f(x)\ud \tilde{P}_{Y}(x)\nonumber\\ 
=& \int_{\iota(\MM)} f(x)\sfp(\iota^{-1}(x)) \iota_*\ud V(x)
=\int_{\MM} f(\iota(x)) \sfp(x) \ud V(\iota(x)),\nonumber
\end{align}
where the second equality follows from the fact that $\tilde{P}_Y$ is the induced probability measure, and the last one comes from the change of variable. To simplify the notation, hereafter we will not distinguish between $x$ and $\iota(x)$ and $\MM$ and $\iota(\MM)$ when there is no danger of ambiguity.

Suppose the data points $\mathcal{X}:=\{x_1,x_2,\ldots,x_n\}\subset\RR^p$ are identically and independently (i.i.d.) sampled from $Y$. For each $x_i$ we pick $u_i\in P(\MM,G)$ so that $\pi(u_i)=x_i$. 
To simplify the notation, we denote $u_i:= u_{x_i}$ when $x_i\in\mathcal{X}$ and $\myP^i_{j}:=\myP^{x_i}_{x_j}$ when $x_i,x_j\in\mathcal{X}$.
Denote the $nq$ dimensional Euclidean vector spaces $V_{\mathcal{X}}:=\oplus_{x_i\in\mathcal{X}}\RR^{q}$ and $E_{\mathcal{X}}:=\oplus_{x_i\in\mathcal{X}}E_{x_i}$, which represents the discretized vector bundle. Note that $V_{\mathcal{X}}$ is isomorphic to $E_{\mathcal{X}}$ since $E_{x_i}$ is isomorphic to $\RR^{q}$. 
Given a $\vw\in E_{\mathcal{X}}$, we denote $\vw=[\vw[1],\ldots,\vw[n]]$ and $\vw[l]\in E_{x_l}$ to be the $l$-th component in the direct sum for all $l=1,\ldots,n$. 

We need a map to realize the isomorphism between $V_{\mathcal{X}}$ and $E_{\mathcal{X}}$.
Define operators $B_{\mathcal{X}}: V_{\mathcal{X}}\to E_{\mathcal{X}}$ and $B^T_{\mathcal{X}}: E_{\mathcal{X}}\to V_{\mathcal{X}} $  by
\begin{align}
&B_{\mathcal{X}}\vv:=[ u_1\vv[1] ,\ldots,  u_n\vv[n] ] \in E_{\mathcal{X}},\label{definition:B_X}\\
&B_{\mathcal{X}}^T\vw:=[u^{-1}_1\vw[1],\ldots u^{-1}_n\vw[n]]\in V_{\mathcal{X}},\nonumber
\end{align}
where $\vw\in E_{\mathcal{X}}$ and $\vv\in V_{\mathcal{X}}$. Note that $B^T_{\mathcal{X}}B_{\mathcal{X}}\vv=\vv$ for all $\vv\in V_{\mathcal{X}}$.
And we define $\delta_{\mathcal{X}}:X\in C(\mathcal{E})\to E_{\mathcal{X}}$ by
\begin{align}
&\delta_{\mathcal{X}}X:=[X({x_1}),\ldots X({x_n})]\in E_{\mathcal{X}}.\nonumber
\end{align}
Here $\delta_{\mathcal{X}}$ is interpreted as the operator finitely sampling the section $X$ and $B_{\mathcal{X}}$ the discretization of the action of a section from $\MM\to P(\MM,G)$ on $\RR^{q}$. Note that under the tangent bundle setup, the operator $B^T_{\mathcal{X}}$ can be understood as finding the coordinates of $\vw[i]$ associated with $u_i$; $B_{\mathcal{X}}$ can be understood as recovering the point on $E_{x_i}$ from the coordinate $\vv[i]$ with related to  $u_i$. We can thus define
\begin{align}
\vX :=B^T_{\mathcal{X}}\delta_{\mathcal{X}}X\in V_{\mathcal{X}},\label{definition:vX}
\end{align}
which is the coordinate of the discretized section $X$ associated with the samples on the principal bundle if we are considering the tangent bundle setup.

We follow the standard notation defined in \cite{vanderVaart_Wellner:1996}. 
\begin{definition}
Take a probability space $(\Omega,\mathcal{F},P)$. For a pair of measurable functions $l:\Omega\to \RR$ and $u:\Omega\to \RR$, a bracket $[l,u]$ is the set of all measurable functions $f:\Omega\to \RR$ with $l\leq f\leq u$. An $\epsilon$-bracket in $L_1(P)$, where $\epsilon>0$, is a bracket $[l,u]$ with $\int |u(y)-l(y)| \ud P(y)\leq \epsilon$. Given a class of measurable function $\mathfrak{F}$, the bracketing number $N_{[]}(\epsilon,\mathfrak{F},L_1(P))$ is the minimum number of $\epsilon$-brackets needed to cover $\mathfrak{F}$.
\end{definition}
Define the empirical measure from the i.i.d. samples $\mathcal{X}$:
$$
\mathbb{P}_n:=\frac{1}{n}\sum_{i=1}^n\delta_{x_i}.
$$
For a given measurable vector-valued function $F:\MM\to \RR^m$ for $m\in\NN$, define
$$
\mathbb{P}_nF:=\frac{1}{n}\sum_{i=1}^n F(x_i)\quad\mbox{and}\quad\bP F:=\int_\MM F(x)\sfp(x)\ud V(x).
$$ 
\begin{definition}
Take a sequence of i.i.d. samples $\mathcal{X}:=\{x_1,\ldots,x_n\}\subset\MM$ according to the p.d.f. $p$. We call a class $\mathfrak{F}$ of measurable functions a Glivenko-Cantelli class if 
\begin{enumerate}
\item $\bP f$ exists for all $f\in \mathfrak{F}$
\item $\sup_{f\in \mathfrak{F}}| \mathbb{P}_nf-\bP f|\to 0$
almost surely when $n\to\infty$.
\end{enumerate} 
\end{definition}

Next, we introduce the following notations regarding the kernel used throughout the paper. Fix $h>0$. Given a non-negative continuous kernel function $K:[0,\infty)\to\RR_+$ decaying fast enough characterizing the affinity between two sampling points $x\in\MM$ and $y\in\MM$, we denote
\begin{align}
K_h(x,y)&:=\,K\Big(\frac{\|x-y\|_{\RR^p}}{\sqrt{h}}\Big)\in C(\MM\times\MM)\label{notation:Kh_def}.
\end{align}
where $x,y\in\MM$. For $0\leq \alpha\leq 1$, we define the following functions
\begin{equation}\label{definition:continuousK}\begin{split}
p_h(x) &:=\,\int K_{h}(x,y)\sfp(y)\ud V(y)\in C(\MM),\quad K_{h,\alpha}(x,y) :=\,\frac{K_{h}(x,y)}{p^\alpha_h(x) p^\alpha_h(y)}\in C(\MM\times\MM),\\
d_{h,\alpha}(x) &:=\,\int K_{h,\alpha}(x,y)\sfp(y)\ud V(y)\in C(\MM),\quad  
M_{h,\alpha}(x,y):=\,\frac{K_{h,\alpha}(x,y)}{d_{h,\alpha}(x)}\in C(\MM\times\MM),  
\end{split}\end{equation}
where $p_h(x)$ is an estimation of the p.d.f. at $x$  by the approximation of identify. 
Here, the practical meaning of $K_{h,\alpha}(x,y)$ is a new kernel function at $(x,y)$ adjusted by the estimated p.d.f. at $x$ and $y$; that is, the kernel is normalized to reduce the influence of the non-uniform p.d.f.. In practice, when we have only finite samples, we approximate the above terms by the following estimators :
\begin{equation}\label{definition:finiteK}\begin{split}
\widehat{p}_{h,n}(x)&:=\,\frac{1}{n}\sum_{k=1}^n K_{h}(x,x_k)\in C(\MM),\quad \widehat{K}_{h,\alpha,n}(x,y) :=\frac{K_{h}(x,y)}{\widehat{p}^\alpha_{h,n}(x) \widehat{p}^\alpha_{h,n}(y)}\in C(\MM\times\MM), \\
\widehat{d}_{h,\alpha,n}(x)&:=\,\frac{1}{n}\sum_{k=1}^n \widehat{K}_{h,\alpha,n}(x,x_k)\in C(\MM),\quad \widehat{M}_{h,\alpha,n}(x,y) :=\,\frac{\widehat{K}_{h,\alpha,n}(x,y)}{\widehat{d}_{h,\alpha,n}(x)}\in C(\MM\times\MM). 
\end{split}\end{equation}
Note that $\widehat{d}_{h,\alpha,n}$ is always positive if $K$ is positive.

\section{Unifying VDM, ODM, LE and DM from the Principal Bundle viewpoint}\label{subsection:tangent_bundle}
Before unifying these algorithms, we state some of the known results relevant to VDM, ODM, LE and DM. Most of the results that have been obtained are of two types: either they provide the topological information about the data which is global in nature, or they concern the geometric information which aims to recover the local information of the data. Fix the undirected affinity graph $\vG=(\vV,\vE)$. When it is built from a point cloud randomly sampled from a Riemannian manifold $\iota:\MM\hookrightarrow\RR^p$ with the induced metric $g$ from the canonical metric of the ambient space, the main ingredient of LE and DM is the Laplace-Beltrami operator $\Delta_g$ of $(\MM,g)$ \cite{coifman_lafon:2006}. It is well known that the Laplace-Beltrami operator $\Delta_g$ provides some topology and geometry information about $\MM$ \cite{gilkey:1974}. For example, the dimension of the null space of $\Delta_g$ is the number of connected components of $\MM$; the spectral embedding of $\MM$ into the Hilbert space \cite{berard_besson_gallot:1994} preserves the geometric information of $\MM$. We can actually study LE and DM in the principal bundle framework. Indeed, $\Delta_g$ is associated with the trivial bundle $\mathcal{E}(P(\MM,\{e\}),\rho,\RR)$, where $\rho$ is the trivial representation of $\{e\}$ on $\RR$. If we consider a non-trivial bundle, we obtain different Laplacian operators, which provide different geometric/topological information \cite{gilkey:1974}. For example, the core of VDM in \cite{singer_wu:2012} is the connection Laplacian associated with the tangent bundle $T\MM$, which provides not only the geodesic distance among nearby points (local information) but also the $1$-Betti number mixed with the Ricci curvature of the manifold. In addition, the notion of synchronization of vector fields on $\vG$ accompanied with translation group can be analyzed by the graph connection Laplacian \cite{bandeira_singer_spielman:2013}.

\subsection{Principal Bundle Setup}
As the reader may have noticed, the appearance of VDM is similar to that of LE, DM and ODM. This is not a coincidence if we take the notion of principal bundle and its connection into account. Based on this observation, we are able unify VDM, ODM, LE and DM in this section. 

We make the following assumptions about the manifold setup.
\begin{assumption}\label{Assumption:A}
\begin{enumerate}
\item[(A1)] The manifold $\MM$ is $d$-dim, smooth and smoothly embedded in $\RR^p$ via $\iota$ with the metric $g$ induced from the canonical metric of $\RR^p$. If the manifold is not closed, we assume that the boundary is smooth.
\item[(A2)] Fix a principal bundle $P(\MM,G)$ with a connection $1$-form $\omega$. Denote $\rho$ to be the representation of $G$ into $O(q)$, where $q>0$ depending on the application.\footnote{We restrict ourselves to the orthogonal representation in order to obtain a symmetric matrix in the VDM algorithm. Indeed, if the translation of the vector status from $x_i$ to $x_j$ satisfies $u_j^{-1}u_i$, where $u_i,u_j\in P(\MM,G)$ and $\pi(u_i)=x_i$ and $\pi(u_j)=x_j$, the translation from $x_j$ back to $x_i$ should satisfy $u_i^{-1}u_j$, which is the inverse of $u_j^{-1}u_i$. To have a symmetric matrix in the end, we thus need $u_j^{-1}u_i=(u_i^{-1}u_j)^T$, which is satisfied only when $G$ is represented into the orthogonal group. We refer the reader to Appendix \ref{section:appendix:principal_bundle} for further details based on the notion of connection.}   
Denote $\mathcal{E}:=\mathcal{E}(P(\MM,G),\rho,\RR^{q})$ to be the associated vector bundle with a fiber metric $g^{\mathcal{E}}$ and the metric connection $\nabla^{\mathcal{E}}$.\footnote{In general, $\rho$ can be the representation of $G$ into $O(q)$ which acts on the tensor space $T^r_s(\RR^{q})$ of type $(r,s)$ or others. But we consider $\RR^{q}=T^1_0(\RR^{q})$ to simplify the discussion.}
\end{enumerate}
\end{assumption} 
The following two special principal bundles and their associated vector bundles are directly related to ODM, LE and DM. The principal bundle for ODM is the non-trivial orientation bundle associated with the tangent bundle of a manifold $\MM$, denoted as $P(\MM,\ZZ_2)$, where $\ZZ_2=\{-1,1\}$. The construction of $P(\MM,\ZZ_2)$ is shown in Example \ref{Appendix:Example:Z2bundleConstruction}. Since $\ZZ_2$ is a discrete group, we take the connection as an assignment of the horizontal subspace of $TP(\MM,\ZZ_2)$ as the simply the tangent space of $P(\MM,\ZZ_2)$; that is, $TP(\MM,\ZZ_2)$. Its associated vector bundle is $\mathcal{E}^{\textup{ODM}}=\mathcal{E}(P(\MM,\ZZ_2),\rho,\RR)$, where $\rho$ is the representation of $\ZZ_2$ so that $\rho$ satisfies $\rho(g)x=gx$ for all $g\in \ZZ_2$ and $x\in\RR$. Note that $\ZZ_2\cong O(1)$. 
The principal bundle for LE and DM is $P(\MM,\{e\})$, where $\{e\}$ is the identify group. Its construction can be found in Example \ref{Appendix:Example:trivialbundleConstruction} and we focus on the trivial connection. Its associated vector bundle is $\mathcal{E}^{\textup{DM}}=\mathcal{E}(P(\MM,\{e\}),\rho,\RR)$, where the representation $\rho$ satisfies $\rho(e)x=x$ and $x\in\RR$. In other words, $\mathcal{E}^{\textup{DM}}$ is the trivial line bundle on $\MM$. Note that $\{e\}\cong SO(1)$.

Under the manifold setup assumption, we sample data from a random vector $Y$ satisfying:  
\begin{assumption}\label{Assumption:B}
\begin{enumerate}
\item[(B1)] The random vector $Y$ has the range $\iota(\MM)$ satisfying Assumption \ref{Assumption:A}. The probability density function $\sfp\in C^4(\MM)$ of $Y$ is uniformly bounded from below and above; that is, $0<p_m\leq \sfp(x)\leq p_M<\infty$ for all $x\in\MM$.
\item[(B2)] The sample points $\mathcal{X} = \{x_i\}_{i=1}^n \subset M$ are sampled independently from $Y$. 
\item[(B3)]  For each $x_i\in \mathcal{X}$, pick $u_i\in P(\MM,G)$ such that $\pi(u_i)=x_i$. Denote $\mathcal{G}=\{u_i:\RR^{q}\to E_{x_i}\}_{i=1}^{n}$. 
\end{enumerate}
\end{assumption} 
The kernel and bandwidth used in the following sections satisfy:
\begin{assumption}\label{Assumption:K}
\begin{enumerate}
\item[(K1)] The kernel function $K\in C^2(\RR_+)$ is a positive function satisfying that $K$ and $K'$ decay exponentially fast. 
Denote $\mu^{(k)}_{r,l}:=\int_{\mathbb{R}^d}\|x\|^l \partial^k (K^r)(\|x\|)\ud x<\infty$, where $k=0,1,2$, $l=0,1,2,3$, $r=1,2$ and $\partial^k$ denotes the $k$-th order derivative. We assume $\mu^{(0)}_{1,0}=1$.
\item[(K2)] The bandwidth of the kernel, $h$, satisfies $0<\sqrt{h}< \min\{\tau,\text{inj}(\MM)\}$.
\end{enumerate}
\end{assumption}

\subsection{Unifying VDM, ODM, LE and DM under the manifold setup}

Suppose Assumption \ref{Assumption:A} is satisfied and we are given $\mathcal{X}$ and $\mathcal{G}$ satisfying Assumption \ref{Assumption:B}.
The affinity graph $\vG=(\vV,\vE)$ is constructed in the following way. Take $\vV=\mathcal{X}$ and $\vE=\{(x_i,x_j)|\, x_i,x_j\in\mathcal{X} \}$. Under this construction $\vG$ is undirected and complete. The affinity between $x_i$ and $x_j$ is defined by 
$$
w_{ij}:=\widehat{K}_{h,\alpha,n}(x_i,x_j),
$$ 
where $0\leq \alpha\leq 1$, $K$ is the kernel function satisfying Assumption \ref{Assumption:K} and $\widehat{K}_{h,\alpha,n}(x_i,x_j)$ is defined in (\ref{definition:finiteK}); that is, we define an affinity function $w:\vE\to \RR_+$. The connection group $g_{ij}$ between $x_i$ and $x_j$ is constructed from $\mathcal{G}$ by
\begin{equation}\label{VDM:gij}
g_{ij} := u_i^{-1}\myP_{j}^iu_j,
\end{equation}
which form a group-valued function $g:\vE\to O(q)$. We call $(\vG,w,g)$ a {\it connection graph}.
With the connection graph, the GCL can be implemented. Define the following $n\times n$ block matrix $\vP_{h,\alpha,n}$ with $q\times q$ block entries:
\begin{equation}\label{def:Sbundle}
\vP_{h,\alpha,n}(i,j) = \left\{\begin{array}{lcl}
                  w_{ij}g_{ij} &  & (i,j)\in E, \\
                  0_{q\times q} &  & (i,j)\notin E.
                \end{array}
 \right.
\end{equation}
Notice that the square matrix $\vP_{h,\alpha,n}$ is symmetric since $w_{ij}=w_{ji}$ and $g_{ij}=g_{ji}^T$. Then define a $n\times n$ diagonal block matrix $\vD_n$ with $q\times q$ entries, where the diagonal blocks are scalar multiples of the identity matrices given by
\begin{equation}\label{def:Dbundle_all}
\vD_{h,\alpha,n}(i,i) = \sum_{j:(i,j)\in E} w_{ij}\vI_{q}=\widehat{d}_{h,\alpha,n}(x)\vI_{q}.
\end{equation}
Take $\boldsymbol{v}\in \RR^{nq}$. The matrix $\vD_{h,\alpha,n}^{-1}\vP_{h,\alpha,n}$ is thus an operator acting on $\boldsymbol{v}$ by
\begin{align*} 
(\vD_{h,\alpha,n}^{-1}\vP_{h,\alpha,n}\boldsymbol{v})[i] = \frac{\sum_{j=1 }^n \widehat{K}_{h,\alpha,n}(x_i,x_j)g_{ij}\vv[j]}{\sum_{j=1 }^n \widehat{K}_{h,\alpha,n}(x_i,x_j)}=\frac{1}{n}\sum_{j=1}^n \widehat{M}_{h,\alpha,n}(x_i,x_j)g_{ij}\vv[j],
\end{align*}
where $\widehat{M}_{h,\alpha,n}(x_i,x_j)$ is defined in (\ref{definition:finiteK}).

Recall the notation $\vX:=B^T_{\mathcal{X}}\delta_{\mathcal{X}}X$ defined in (\ref{definition:vX}). Then, consider the following quantity:
\begin{align}
\Big(\frac{\vD_{h,\alpha,n}^{-1}\vP_{h,\alpha,n}-\vI_n}{h}\vX\Big)[i] =& \frac{1}{n}\sum_{j=1}^n \widehat{M}_{h,\alpha,n}(x_i,x_j)\frac{1}{h}(g_{ij}\vX[j]-\vX[i])\nonumber\\
= &\frac{1}{n}\sum_{j=1}^n \widehat{M}_{h,\alpha,n}(x_i,x_j)\frac{1}{h}(u_i^{-1}\myP_j^iX(x_j)-u_i^{-1}X(x_i)).\label{VDM:comparison}
\end{align}
Note that geometrically $g_{ij}$ is closely related to the parallel transport  (\ref{definition:parallel:frameX}) from $x_j$ to $x_i$. 
Indeed, rewrite the definition of the covariant derivative in (\ref{definition:parallel:frameX}) by
\begin{equation*}
\nabla_{\dot{c}(0)} X(x_i)=\lim_{h\to 0}\frac{1}{h}[u(0)u(h)^{-1}X(c(h))-X(c(0))].
\end{equation*}
where $c:[0,1]\to \MM$ is the geodesic on $\MM$ so that $c(0)=x_i$ and $c(h)=x_j$ and $u(h)$ is the horizontal lift of $c$ to $P(\MM,G)$ so that $\pi(u(0))=x_i$.
Next rewrite
\begin{equation}\label{VDM:cov_deri2}
u(0)^{-1}\nabla_{\dot{c}(0)} X =\lim_{h\to 0}\frac{1}{h}\big\{u(h)^{-1}X(c(h))-u(0)^{-1}X(c(0))\big\},
\end{equation}
where the right hand side is exactly the term appearing in (\ref{VDM:comparison}) by the definition of parallel transport since $u(h)^{-1}=u(0)^{-1}\myP_{c(h)}^{c(0)}$. As will be shown explicitly in the next section, the GCL reveals the information about the manifold by accumulating the local information via taking the covariant derivative into account.

Now we unify ODM, LE and DM. 
For ODM, we consider the orientation principal bundle $P(\MM,\ZZ_2)$ and its associated vector bundle $\mathcal{E}^{\textup{ODM}}$. In this case, $\mathcal{G}$ is $\{u_i^{\textup{ODM}}\}_{i=1}^n$, $u^{\textup{ODM}}_i\in P(\MM,\ZZ_2)$ and $u^{\textup{ODM}}_i:\RR\to E_i$, where $E_i$ is the fiber of $\mathcal{E}^{\textup{ODM}}$ at $x_i\in \MM$. Note that the fiber of $\mathcal{E}^{\textup{ODM}}$ is isomorphic to $\RR$. The connection group $g^{\textup{ODM}}_{ij}$ between $x_i$ and $x_j$ is constructed by
\begin{equation*}
g^{\textup{ODM}}_{ij} = {u^{\textup{ODM}}_i}^{-1} u^{\textup{ODM}}_j.
\end{equation*}
In practice, $u^{\textup{ODM}}_j$ comes from the orientation of the sample from the frame bundle. Indeed, given $x_i$ and $u_i\in O(\MM)$ so that $\pi(u_i)=x_i$, $g^{\textup{ODM}}_{ij}$ is defined to be the orientation of $u_i^{-1}\myP_{j}^iu_j$; that is, the determinant of $u_i^{-1}\myP_{j}^iu_j\in O(d)$. Define a $n\times n$ matrix with scalar entries $\vP^{\textup{ODM}}_{h,\alpha,n}$, where
\begin{equation*}
\vP^{\textup{ODM}}_{h,\alpha,n}(i,j) = \left\{\begin{array}{lcl}
                  w_{ij}g^{\textup{ODM}}_{ij}  &  & (i,j)\in \vE, \\
                  0 &  & (i,j)\notin \vE 
                \end{array}
 \right.
\end{equation*}
and a $n\times n$ diagonal matrix $\vD^{\textup{ODM}}_{h,\alpha,n}$, where
\begin{equation*}
\vD^{\textup{ODM}}_{h,\alpha,n}(i,i) = d_i. 
\end{equation*}
It has been shown in \cite[Section 2.3]{singer_wu:2011} that the orientability information of $\MM$ can be obtained from analyzing ${\vD^{\textup{ODM}}_{h,1,n}}^{-1}\vP^{\textup{ODM}}_{h,1,n}$. When the manifold is orientable, we get the orientable diffusion maps (ODM) by taking the higher eigenvectors of ${\vD^{\textup{ODM}}_{h,1,n}}^{-1}\vP^{\textup{ODM}}_{h,1,n}$ into account; when the manifold is non-orientable, we can recover the orientable double covering of the manifold by the {\it modified diffusion maps} \cite[Section 3]{singer_wu:2011}. In \cite[Section 3]{singer_wu:2011}, it is conjectured that any smooth, closed non-orientable manifold $(\MM,g)$ has an orientable double covering embedded symmetrically inside $\RR^p$ for some $p\in\NN$. To make the unification self-contained, we will show in Appendix \ref{appendix_proof_symmetric_embedding} that this conjecture is true by modifying the proof of the Nash embedding theorem  \cite{Nash1954,Nash1956}. This fact provides us a better visualization of reconstructing the orientable double covering by the modified diffusion maps.

For LE and DM, we consider the trivial principal bundle $P(\MM,\{e\})$ and its associated trivial line bundle $\mathcal{E}^{\textup{DM}}$. 
In this case, $q=1$. Define a $n\times n$ matrix with scalar entries $\vP^{\textup{DM}}_{h,\alpha,n}$:
\begin{equation*}
\vP^{\textup{DM}}_{h,\alpha,n}(i,j) = \left\{\begin{array}{lcl}
                  w_{ij} &  & (i,j)\in \vE, \\
                  0 &  & (i,j)\notin \vE.
                \end{array}
 \right.
\end{equation*}
and a $n\times n$ diagonal matrix $\vD^{\textup{DM}}_{h,\alpha,n}$: 
\begin{equation*}
\vD^{\textup{DM}}_{h,\alpha,n}(i,i) = d_i. 
\end{equation*}
Note that this is equivalent to ignoring the connection group in each edge in GCL. Indeed, when we study DM, we do not need the notion of connection group. This actually comes from the fact that functions defined on the manifold are actually sections of the trivial line bundle of $\MM$ -- since the fiber $\RR$ and $\MM$ are decoupled, we can directly take the algebraic relationship of $\RR$ into consideration, so that it is not necessary to mention the bundle structure. With the well-known normalized graph Laplacian, $\vI_n-{\vD^{\textup{DM}}_{h,0,n}}^{-1}\vP^{\textup{DM}}_{h,0,n}$, we can apply DM or LE for dimension reduction, spectral clustering, reparametrization, etc.

To sum up, we are able to unify the VDM, ODM, LE and DM by considering the principal bundle structure. In the next sections, we focus on the pointwise and spectral convergence of the corresponding operators. 

\section{Pointwise and Spectral Convergence of GCL}\label{proof:PointwiseSpectralConvergence}
With the above setup, we now do the asymptotic analysis under Assumption \ref{Assumption:A}, Assumption \ref{Assumption:B} and Assumption \ref{Assumption:K}. 
Throughout the proof, since $\sfp$, $\MM$ and $\iota$ are fixed, and $\sfp\in C^4$, $\MM$, $\partial\MM$ and $\iota$ are smooth and $\MM$ is compact, we know that $\|\sfp^{(l)}\|_{L^\infty}$, $l=0,1,2,3,4$, the volume of $\partial\MM$, the curvature of $\MM$ and $\partial\MM$ and the second fundamental form of the embedding $\iota$, as well as their first few covariant derivatives are bounded independent of $h$ and $n$. Thus, we would ignore them in the error term. However, when the error term depends on a given section (or function), it will be precisely stated.

The pointwise convergence of the normalized GL can be found in \cite{belkin_niyogi:2005,coifman_lafon:2006,hein_audibert_luxburg:2005,Gine_Koltchinskii:2006}, and the spectral convergence of the normalized GL when the sampling is uniform and the boundary is empty can be found in \cite{belkin_niyogi:2007}. Here we take care of simultaneously the boundary, the nonuniform sampling and the bundle structure. Note that the asymptotical analysis of the normalized GL is a special case of the analysis in this paper since it is unified to the current framework based on the trivial principal bundle $P(\MM,\{e\})$ and its associated trivial line bundle $\mathcal{E}^{\textup{DM}}$. From a high level, except taking the possibly non-trivial bundle structure into account, the analysis is standard. 

\subsection{Pointwise Convergence}\label{section:pointwise_conv}

\begin{definition} \label{step1:defn1}
Define operators $T_{h,\alpha}:C(\mathcal{E})\to C(\mathcal{E})$ and $\widehat{T}_{h,\alpha,n}:C(\mathcal{E})\to C(\mathcal{E})$ as
\begin{align}
T_{h,\alpha}X(y) =\,\int_\MM M_{h,\alpha}\left(y,x\right)\myP^{y}_{x}X(x)\sfp(x)\ud V(x),\quad
\widehat{T}_{h,\alpha,n}X(y) =\,\frac{1}{n}\sum_{j=1}^n \widehat{M}_{h,\alpha,n}\left(y,x_j\right)\myP^{y}_{x_j}X(x_j),\nonumber
\end{align}
where $X\in C(\mathcal{E})$, $0\leq \alpha\leq 1$ and $M_{h,\alpha}$ and $\widehat{M}_{h,\alpha,n}$ are defined in (\ref{definition:continuousK}) and (\ref{definition:finiteK}) respectively.
\end{definition}
First, we have the following theorem stating that the integral operator $T_{h,\alpha}$ is an approximation of identity which allows us to obtain the connection Laplacian:

\begin{theorem}\label{thm:pointwise_conv:approx_of_identity}
Suppose Assumption \ref{Assumption:A} and Assumption \ref{Assumption:K} hold. Take $0<\gamma<1/2$. When $0\leq \alpha\leq 1$, for all $x_i\notin \MM_{h^\gamma}$ and $X\in C^4(\mathcal{E})$ we have
\begin{align*}
\begin{split}
(T_{h,\alpha}X-X) (x)=h\frac{\mu^{(0)}_{1,2}}{2d}\left(\nabla^2X(x)+\frac{2\nabla X(x)\cdot\nabla(\sfp^{1-\alpha})(x)}{\sfp^{1-\alpha}(x)}\right)+O(h^2),
\end{split}
\end{align*}
where $O(h^2)$ depends on $\|X^{(\ell)}\|_{L^\infty}$, where $\ell=0,1\ldots,4$;
when $x_i\in\MM_{h^{\gamma}}$, we have
\begin{equation}\label{bdryrslt3}
(T_{h,\alpha}X -X)(x) =\sqrt{h}\frac{m_{h,1}}{m_{h,0}}\myP^{x}_{x_0}\nabla_{\partial_d}X(x_0)+O(h^{2\gamma}),
\end{equation}
where $O(h^{2\gamma})$ depends on $\|X^{(\ell)}\|_{L^\infty}$, where $\ell=0,1,2$, $x_0=\argmin_{y\in\partial\MM}d(x_i,y)$, $m_{h,1}$ and $m_{h,0}$ are constants defined in (\ref{meps0}), and $\partial_d$ is the normal direction to the boundary at $x_0$. 
\end{theorem}

Second, we show that when $n\to \infty$, asymptotically the matrix $\vD_{h,\alpha,n}^{-1}\vP_{h,\alpha,n}-\vI$ behaves like the integral operator $T_{h,\alpha}-1 $. The main component in this asymptotical analysis in the stochastic fluctuation analysis of the GCL.  
As is shown in Theorem \ref{thm:pointwise_conv:approx_of_identity}, the term we have interest in, the connection Laplacian (or Laplace-Beltrami operator when we consider GL), is of order $O(h)$. Therefore the stochastic fluctuation incurred by the finite sampling points should be controlled to be much smaller than $O(h)$ otherwise we are not able to recover the object of interest. 

\begin{theorem}\label{thm:pointwise_conv_to_integral}
Suppose Assumption \ref{Assumption:A}, Assumption \ref{Assumption:B} and Assumption \ref{Assumption:K} hold and $X\in C(\mathcal{E})$. Take $0< \alpha\leq 1$. Suppose we focus on the situation that the stochastic fluctuation of $(\vD_{h,\alpha,n}^{-1}\vP_{h,\alpha,n}\vX-\vX)[i]$ is $o(h)$ for all $i$. Then, with probability higher than $1-O(1/n^2)$, for all $i=1,\ldots,n$,   
\begin{align}
(\vD_{h,\alpha,n}^{-1}\vP_{h,\alpha,n}\vX-\vX)[i] = u_i^{-1}(T_{h,\alpha}X -X)(x_i)+ O\left(\frac{\sqrt{\log(n)}}{n^{1/2}h^{d/4}}\right),\nonumber
\end{align}
where $\vX$ is defined in (\ref{definition:vX}). 

Take $\alpha=0$ and $1/4<\gamma<1/2$. Suppose we focus on the situation that the stochastic fluctuation of $(\vD_{h,0,n}^{-1}\vP_{h,0,n}\vX-\vX)[i]$ is $o(h)$ for all $i$. Then with probability higher than $1-O(1/n^2)$, for all $x_i\notin\MM_{h^\gamma}$ we have
\begin{align}
(\vD_{h,0,n}^{-1}\vP_{h,0,n}\vX-\vX)[i] = u_i^{-1}(T_{h,0}X -X)(x_i)+ O\left(\frac{\sqrt{\log(n)}}{n^{1/2}h^{d/4-1/2}}\right);\nonumber
\end{align}
if we focus on the situation that the stochastic fluctuation of $(\vD_{h,0,n}^{-1}\vP_{h,0,n}\vX-\vX)[i]$ is $o(h^{1/2})$ for all $i$, with probability higher than $1-O(1/n^2)$, for all $x_i\in\MM_{h^\gamma}$:
\begin{align}
(\vD_{h,0,n}^{-1}\vP_{h,0,n}\vX-\vX)[i] = u_i^{-1}(T_{h,0}X -X)(x_i)+ O\left(\frac{\sqrt{\log(n)}}{n^{1/2}h^{d/4-1/4}}\right).\nonumber
\end{align} 

\end{theorem}
Here $\sqrt{\log(n)}$ in the error term shows up due to the union bound for all $i=1,\ldots,n$ and the probability bound we are seeking. When $\alpha\neq 0$, we need to estimate the p.d.f. from finite sampling points. This p.d.f. estimation slows down the convergence rate. The proofs of Theorem \ref{thm:pointwise_conv:approx_of_identity} and Theorem \ref{thm:pointwise_conv_to_integral} are postponed to the Appendix. These Theorems lead to the following pointwise convergence of the GCL. Here, the error term in Theorem \ref{thm:pointwise_conv_to_integral} is the stochastic fluctuation (variance) when the number of samples is finite, and the error term in Theorem \ref{thm:pointwise_conv:approx_of_identity} is the bias term introduced by the kernel approximation.
\begin{corollary}\label{thm:pointwise_conv}
Suppose Assumption \ref{Assumption:A}, Assumption \ref{Assumption:B} and Assumption \ref{Assumption:K} hold. Take $0<\gamma<1/2$ and $X\in C^4(\mathcal{E})$. When $0< \alpha\leq 1$, if we focus on the situation that the stochastic fluctuation of $(\vD_{h,\alpha,n}^{-1}\vP_{h,\alpha,n}\vX-\vX)[i]$ is $o(h)$ for all $i$, then with probability higher than $1-O(1/n^2)$, the following holds for all $x_i\notin \MM_{h^\gamma}$:
\begin{align}
 h^{-1}(\vD_{h,\alpha,n}^{-1}\vP_{h,\alpha,n}\vX-\vX)[i] = &\frac{\mu^{(0)}_{1,2}}{2d}u_i^{-1}\left\{\nabla^2X(x_i)+\frac{2\nabla X(x_i)\cdot\nabla(\sfp^{1-\alpha})(x_i)}{\sfp^{1-\alpha}(x_i)}\right\} + O(h) + O\left(\frac{\sqrt{\log(n)}}{n^{1/2}h^{d/4+1}}\right),\nonumber
\end{align}
where $\nabla X(x_i)\cdot\nabla(\sfp^{1-\alpha})(x_i):=\sum_{l=1}^d\nabla_{\partial_l}X\nabla_{\partial_l}(\sfp^{1-\alpha})$ and $\{\partial_l\}_{l=1}^d$ is an normal coordinate around $x_i$; if we focus on the situation that the stochastic fluctuation of $(\vD_{h,\alpha,n}^{-1}\vP_{h,\alpha,n}\vX-\vX)[i]$ is $o(h^{1/2})$ for all $i$, then with probability higher than $1-O(1/n^2)$, the following holds for all $x_i\in \MM_{h^\gamma}$:
\begin{align*}
(\vD_{h,\alpha,n}^{-1}\vP_{h,\alpha,n}\vX)[i]= &u_i^{-1}\left(X(x)+\sqrt{h}\frac{m_{h,1}}{m_{h,0}}\myP^{x_i}_{x_0}\nabla_{\partial_d}X(x_0)\right) +O(h^{2\gamma}) + O\left(\frac{\sqrt{\log(n)}}{n^{1/2}h^{d/4}}\right),
\end{align*}
where $x_0=\argmin_{y\in\partial\MM}d(x_i,y)$, $m_{h,1}$ and $m_{h,0}$ are constants defined in (\ref{meps0}), and $\partial_d$ is the normal direction to the boundary at $x_0$.

Take $\alpha=0$. If we focus on the situation that the stochastic fluctuation of $(\vD_{h,0,n}^{-1}\vP_{h,0,n}\vX-\vX)[i]$ is $o(h)$ for all $i$, then with probability higher than $1-O(1/n^2)$, the following holds for all $x_i\notin \MM_{h^\gamma}$:
\begin{align}
 h^{-1}(\vD_{h,0,n}^{-1}\vP_{h,0,n}\vX-\vX)[i] = &\frac{\mu^{(0)}_{1,2}}{2d}u_i^{-1}\left\{\nabla^2X(x_i)+\frac{2\nabla X(x_i)\cdot\nabla \sfp (x_i)}{\sfp(x_i)}\right\} + O(h) + O\left(\frac{\sqrt{\log(n)}}{n^{1/2}h^{d/4+1/2}}\right);\nonumber
\end{align}
if we focus on the situation that the stochastic fluctuation of $(\vD_{h,0,n}^{-1}\vP_{h,0,n}\vX-\vX)[i]$ is $o(h^{1/2})$ for all $i$, we have with probability higher than $1-O(1/n^2)$, the following holds for all $x_i\in\MM_{h^{\gamma}}$: 
\begin{align*}
(\vD_{h,\alpha,n}^{-1}\vP_{h,\alpha,n}\vX)[i]= &u_i^{-1}\left(X(x)+\sqrt{h}\frac{m_{h,1}}{m_{h,0}}\myP^{x_i}_{x_0}\nabla_{\partial_d}X(x_0)\right) +O(h^{2\gamma}) + O\left(\frac{\sqrt{\log(n)}}{n^{1/2}h^{d/4-1/4}}\right),
\end{align*}

\end{corollary}
 
\begin{remark}
Several existing results regarding normalized GL and the estimation of Laplace-Beltrami operator are unified in Theorem \ref{thm:pointwise_conv:approx_of_identity}, Theorem \ref{thm:pointwise_conv_to_integral} and Corollary \ref{thm:pointwise_conv}. Indeed, when the principal bundle structure is trivial, that is, when we work with the normalized GL, and when $\alpha=0$, the p.d.f. is uniform and the boundary does not exist, the results in \cite{belkin_niyogi:2005,Gine_Koltchinskii:2006} are recovered; when $\alpha=0$, the p.d.f. is non-uniform and the boundary is not empty, we recover results in \cite{singer:2006,coifman_lafon:2006}; when $\alpha\neq 0$ and the boundary is empty, we recover results in \cite{hein_audibert_luxburg:2005}.
\end{remark}

\begin{remark}
We now discuss how GCL converges from the discrete setup to the continuous setup and the how to choose the optimal bandwidth under the assumption that $\partial\MM=\emptyset$. Similar arguments hold when $\partial\MM\neq \emptyset$. Take $\alpha=0$. Clearly, if $h$ depends on $n$ so that $h_n\to 0$ and $\frac{n^{1/2}h_n^{d/4+1/2}}{\sqrt{\log(n)}}\to \infty$, then when $n\to \infty$, asymptotically $h_n^{-1}(\vD_{h_n,0,n}^{-1}\vP_{h_n,0,n}\vX-\vX)[i]$ converges to $\frac{\mu^{(0)}_{1,2}}{2d}u_i^{-1}\left\{\nabla^2X(x_i)+\frac{2\nabla X(x_i)\cdot\nabla\sfp (x_i)}{\sfp(x_i)}\right\}$ a.s. by the Borel-Cantelli Lemma. When $n$ is finite, by balancing the variance and squared bias we get
$$
h^2 = O\left(\frac{\log(n)}{n h^{d/2+1}}\right),
$$ 
so the optimal kernel bandwidth depending on $n$ which we may choose for the practical purpose is
$$
h_n = O\left(\left(\frac{\log(n)}{n}\right)^{1/(d/2+3)}\right).
$$
Take $\alpha\neq 0$. Similarly, when $h$ depends on $n$ so that $h_n\to 0$ and $\frac{n^{1/2}h_n^{d/4+1}}{\sqrt{\log(n)}}\to \infty$, asymptotically $h_n^{-1}(\vD_{h_n,\alpha,n}^{-1}\vP_{h_n,\alpha,n}\vX-\vX)[i]$ converges to $\frac{\mu^{(0)}_{1,2}}{2d}u_i^{-1}\left\{\nabla^2X(x_i)+\frac{2\nabla X(x_i)\cdot\nabla(\sfp^{1-\alpha})(x_i)}{\sfp^{1-\alpha}(x_i)}\right\}$ a.s.. In this case, the optimal kernel bandwidth is 
$$
h_n = O\left(\left(\frac{\log(n)}{n}\right)^{1/(d/2+4)}\right).
$$
\end{remark}

\begin{remark}
In Theorem \ref{thm:pointwise_conv:approx_of_identity} and Corollary \ref{thm:pointwise_conv}, the regularity of $X$ and the p.d.f. $\sfp$ are assumed to be $C^4$. These conditions can be relaxed to $C^3$ and the proof remains almost same except that the bias term in Corollary \ref{thm:pointwise_conv} becomes $h^{1/2}$.
\end{remark}

\begin{remark}
Note that near the boundary, the error term $m_{h,1}/m_{h,0}$ is of order $\sqrt{h}$, which asymptotically dominates $h$.
A consequence of Corollary \ref{thm:pointwise_conv} and the above discussion about the error terms is that the eigenvectors of $\vD_{h,1,n}^{-1}\vP_{h,1,n}-\vI_n$ are discrete approximations of the eigen-vector-fields of the connection Laplacian operator with homogeneous Neumann boundary condition that satisfy
\begin{equation}
\left\{\begin{array}{ll}
\nabla^2 X(x) = -\lambda X(x), & \mbox{for } x \in \MM, \\
\nabla_{\partial_d}X(x) = 0, & \mbox{for } x \in \partial \MM.
\end{array}\right.
\end{equation}
Also note that the above results are pointwise in nature. The spectral convergence will be discussed in the coming section.
\end{remark}

\subsection{Spectral Convergence}\label{section:spectral_conv}

As informative as the pointwise convergence results in Corollary \ref{thm:pointwise_conv} are, they   are not strong enough to guarantee the spectral convergence of our numerical algorithm, in particular those depending on the spectral structure of the underlying manifold. In this section, we explore this problem and provide the spectral convergence theorem. 

Note that in general $0$ might not be the eigenvalue of the connection Laplaclain $\nabla^2$. For example, when the manifold is $S^2$, the smallest eigenvalue of the connection Laplaclain associated with the tangent bundle is strictly positive due to the vanishing theorem \cite[p. 126]{Berline_Getzler_Vergne:2004}. When $0$ is an eigenvalue, we denote the spectrum of $\nabla^2$ by $\{-\lambda_l\}_{l=0}^\infty$, where $0=\lambda_0<\lambda_1\leq \ldots$, and the corresponding eigenspaces are denoted by $E_l:=\{X\in L^2(\mathcal{E}):~\nabla^2X=-\lambda_l X\}$, $l=0,1,\ldots$; otherwise we denote the spectrum by $\{-\lambda_l\}_{l=1}^\infty$, where $0 <\lambda_1\leq \ldots$, and the eigenspaces by $E_l$. It is well known \cite{gilkey:1974} that $\dim(E_l)<\infty$, the eigen-vector-fields are smooth and form a basis for $L^2(\mathcal{E})$, that is,
$ L^2(\mathcal{E})=\overline{\oplus_{l\in\NN\cup\{0\}} E_l}$, where the over line means completion according to the measure associated with $g$. To simplify the statement and proof, we assume that $\lambda_l$ for each $l$ are simple and $X_l$ is the normalized basis of $E_l$.\footnote{When any of the eigenvalues is not simple, the statement and proof are complicated by introducing the notion of eigen-projection \cite{Chatelin:2011}.} 

The first theorem states the spectral convergence of $(\vD^{-1}_{h,1,n}\vP_{h,1,n})^{t/h}$ to $e^{t\nabla^2}$. Note that in the statement of the theorem, we use $\widehat{T}_{h,1,n}$ instead of $\vD^{-1}_{h,1,n}\vP_{h,1,n}$. As we will see in the proof, they are actually equivalent under proper transformation.
\begin{theorem}\label{thm:spectralconvergence_heatkernel}
Suppose Assumption \ref{Assumption:A}, Assumption \ref{Assumption:B} and Assumption \ref{Assumption:K} hold, and $2/5<\gamma<1/2$. Fix $t>0$. Denote $\mu_{t,i,h,n}$ to be the $i$-th eigenvalue of
$\widehat{T}^{t/h}_{h,1,n}$
with the associated eigenvector $X_{t,i,h,n}$. Also denote $\mu_{t,i}>0$ to be the $i$-th eigenvalue of the heat kernel of the connection Laplacian $e^{t\nabla^2}$ with the associated eigen-vector field $X_{t,i}$. We assume that $\mu_{t,i}$ are simple and both $\mu_{t,i,h,n}$ and $\mu_{t,i}$ decrease as $i$ increase, respecting the multiplicity. Fix $i\in\NN$. Then there exists a sequence $h_n\to 0$ such that $\lim_{n\to \infty}\mu_{t,i,h_n,n}=\mu_{t,i}$ and $\lim_{n\to \infty}\|X_{t,i,h_n,n}-X_{t,i}\|_{L^2(\mathcal{E})}=0$ in probability.
\end{theorem}
\begin{remark}
Recall that for a finite integer $n$, as is discussed in (\ref{fact:MayBeNegative}), $\mu_{t,i,h,n}$ may be negative while $\mu_{t,i}$ is always non-negative. We mention that the existence of $\gamma$ is for the sake of dealing with the boundary, whose effect is shown in (\ref{bdryrslt3}). When the boundary is empty, we can ignore the $\gamma$ assumption. 
\end{remark}
The second theorem states the spectral convergence of $h^{-1}(\vD^{-1}_{h,1,n}\vP_{h,1,n}-\vI_{qn})$ to $\nabla^2$.
\begin{theorem}\label{thm:spectralconvergence_laplacian}
Suppose Assumption \ref{Assumption:A}, Assumption \ref{Assumption:B} and Assumption \ref{Assumption:K} hold, and $2/5<\gamma<1/2$. Denote $-\lambda_{i,h,n}$ to be the $i$-th eigenvalue of
$h^{-1}(\widehat{T}_{h,1,n}-1)$
with the associated eigenvector $X_{i,h,n}$. Also denote $-\lambda_i$, where $\lambda_i>0$, to be the $i$-th eigenvalue of the connection Laplacian $\nabla^2$ with the associated eigen-vector field $X_i$. We assume that $\lambda_{i}$ are simple and both $\lambda_{i,h,n}$ and $\lambda_i$ increase as $i$ increase, respecting the multiplicity. Fix $i\in\NN$. Then there exists a sequence $h_n\to 0$ such that $\lim_{n\to \infty}\lambda_{i,h_n,n}=\lambda_i$ and $\lim_{n\to \infty}\|X_{i,h_n,n}-X_i\|_{L^2(\mathcal{E})}=0$ in probability.
\end{theorem} 
Note that the statement and proof hold for the special cases associated with DM and ODM. We prepare some bounds for the proof.

\begin{lemma}\label{lemma:boundsForAllKindsOfK} 
Take $0\leq\alpha\leq 1$ and $h>0$. Assume Assumption \ref{Assumption:A}, Assumption \ref{Assumption:B} and Assumption \ref{Assumption:K} hold. Then the following uniform bounds hold 
\begin{equation}\label{bound:Mhat}
\begin{split}
&\delta\leq p_h(x) \leq \|K\|_{L^\infty},\quad\delta\leq  \widehat{p}_{h,n}(x) \leq \|K\|_{L^\infty} \\
&\frac{\delta}{\|K\|_{L^\infty}^{2\alpha}}\leq K_{h,\alpha}(x,y)  \leq \frac{\|K\|_{L^\infty}}{\delta^{2\alpha}},\quad\frac{\delta}{\|K\|_{L^\infty}^{2\alpha}}\leq  \widehat{K}_{h,\alpha,n}(x,y) \leq \frac{\|K\|_{L^\infty}}{\delta^{2\alpha}} \\
&\frac{\delta}{\|K\|_{L^\infty}^{2\alpha}}\leq d_{h,\alpha}(x) \leq \frac{\|K\|_{L^\infty}}{\delta^{2\alpha}},\quad \frac{\delta}{\|K\|_{L^\infty}^{2\alpha}}\leq \widehat{d}_{h,\alpha,n}(x) \leq \frac{\|K\|_{L^\infty}}{\delta^{2\alpha}} \\
&\frac{\delta^{1+2\alpha}}{\|K\|^{1+2\alpha}_{L^\infty}}\leq M_{h,\alpha}(x,y) \leq \frac{\|K\|_{L^\infty}^{1+2\alpha}}{\delta^{1+2\alpha}},\quad\frac{\delta^{1+2\alpha}}{\|K\|^{1+2\alpha}_{L^\infty}}\leq \widehat{M}_{h,\alpha,n}(x,y) \leq \frac{\|K\|_{L^\infty}^{1+2\alpha}}{\delta^{1+2\alpha}},
\end{split}
\end{equation}
where $\delta:=\inf_{t\in[0,D/\sqrt{h}]} K(t)$ and $D=:\max_{x,y\in\MM}\|x-y\|_{\RR^p}$.
\end{lemma}
\begin{proof}
By the assumption that the manifold $\MM$ is compact, there exists $D>0$ so that $\|x-y\|_{\RR^p}\leq D$ for all $x,y\in\MM$. Under the assumption that the kernel function $K$ is positive in Assumption \ref{Assumption:K}, for a fixed $h>0$, for all $n\in\NN$ and $x,y\in\MM$ we have
\[
K_h(x,y)\geq \delta:=\inf_{t\in[0,D/\sqrt{h}]} K(t).
\]
Then, for all $x,y\in\MM$, the bounds in (\ref{bound:Mhat}) hold immediately. 
\end{proof}
To prove Theorem \ref{thm:spectralconvergence_heatkernel} and Theorem \ref{thm:spectralconvergence_laplacian}, we need the following Lemma to take care of the pointwise convergence of a series of vector fields in the uniform norm on $\MM$ with the help of the notion of Glivenko-Cantelli class:
\begin{lemma}\label{class:glivenko-cantelli}
Take $0\leq \alpha\leq 1$ and fix $h>0$. Suppose Assumption \ref{Assumption:A}, Assumption \ref{Assumption:B} and Assumption \ref{Assumption:K} are satisfied. Denote two functional classes
\begin{align*}
 \mathcal{K}_h :=\{K_h(x,\cdot);\,x\in \MM\},\quad
 \mathcal{K}_{h,\alpha}:=\{K_{h,\alpha}(x,\cdot);\,x\in \MM\}.
\end{align*}
Then the above classes are Glivenko-Cantelli classes. Take $X\in C(\mathcal{E})$ and a measurable section $q_0:\MM\to P(\MM,G)$, and denote
\begin{align*}
X\circ\mathcal{M}_{h,\alpha}:=\Big\{M_{h,\alpha}(x,\cdot)q_0(x)^{T}\myP^{x}_{\cdot}X(\cdot);\,x\in \MM\Big\}.
\end{align*}
Then the above classes satisfy
\begin{align}\label{GC5:rslt}
\sup_{W\in X\circ\mathcal{M}_{h,\alpha}}\| \bP_nW-\bP W\|_{\RR^{q}}\to 0
\end{align}
a.s. when $n\to \infty$.
\end{lemma}
Note that $W_x \in X\circ\mathcal{M}_{h,\alpha}$ is a $\RR^q$-valued function defined on $\MM$. Also recall that when $y$ is in the cut locus of $x$, we set $\myP^x_yW_x(y)=0$. The above notations are chosen to be compatible with the matrix notation used in the VDM algorithm.  
\begin{proof}
We prove (\ref{GC5:rslt}). The proof for $\mathcal{K}_h$ and $\mathcal{K}_{h,\alpha}$ can be found in \cite[Proposition 11]{VonLuxburg2008}.  
Take $W_x\in X\circ\mathcal{M}_{h,\alpha}$. Since $X\in C(\mathcal{E})$, $\MM$ is compact, $\nabla^{\mathcal{E}}$ is metric and $ q(x) :\RR^{q}\to E_x$ preserving the inner product, we know
$$
\|W_x\|_{L^\infty}\leq  \frac{\|K\|_{L^\infty}^{1+2\alpha}}{\delta^{1+2\alpha}}\| q(x)^{-1}\myP^{x}_{y}X(y)\|_{L^\infty}=\frac{\|K\|_{L^\infty}^{1+2\alpha}}{\delta^{1+2\alpha}} \|X\|_{L^\infty},
$$
where the first inequality holds by the bound in Lemma \ref{lemma:boundsForAllKindsOfK}. Under Assumption \ref{Assumption:A}, $g_x$ is isometric pointwisely, so $X\circ\mathcal{M}_{h,\alpha}$ is uniformly bounded. 

We now tackle the vector-valued function $W_x$ component by component. Rewrite a vector-valued function $W_x$ as $W_x=(W_{x,1},\ldots,W_{x,q})$. Consider
$$
 \mathcal{M}^{(j)}_{h,\alpha}:=\Big\{M_{h,\alpha}(x,\cdot)W_{x,j}(\cdot),\,x\in \MM\Big\},
$$
where $j=1,\ldots q$. Fix $\epsilon>0$. Since $\MM$ is compact and $W_{x}$ is uniformly bounded over $x$, we can choose finite $\epsilon$-brackets $[l_{j,i},u_{j,i}]$, where $i=1,\ldots,N(j,\epsilon)$. 
so that its union contains $ \mathcal{M}^{(j)}_{h,\alpha}$ and $\bP |u_{j,i}-l_{j,i}|<\epsilon$ for all $i=1,\ldots,N(j,\epsilon)$. Then, for every $f\in \mathcal{M}^{(j)}_{h,\alpha}$, there is an $\epsilon$-bracket $[l_{j,l},u_{j,l}]$ in $L^1(P)$ such that $l_{j,l}\leq f\leq u_{j,l}$, and hence
\begin{align}
|\bP_nf-\bP f|&\leq |\bP_nf-\bP u_{j,l}|+\bP |u_{j,l}(y)-f(y)|\leq |\bP_nu_{j,l}-\bP u_{j,l}|+\bP |u_{j,l} -f |\nonumber\\
&\leq |\bP_nu_{j,l}-\bP u_{j,l}|+\bP |u_{j,l} -l_{j,l} | \leq |\bP_nu_{j,l}-\bP u_{j,l}|+\epsilon.\nonumber
\end{align}
Hence we have
\begin{align}
\sup_{f\in  \mathcal{M}^{(j)}_{h,\alpha}}|\bP_nf-\bP f|\leq\max_{l=1,\ldots,N(j,\epsilon)}|\bP_nu_{j,l}-\bP u_{j,l}|+\epsilon,\nonumber
\end{align}
where the right hand side converges a.s. to $\epsilon$ when $n\to\infty$ by the strong law of large numbers and the fact that $N(j,\epsilon)$ is finite.
As a result, we have
\begin{align}
& | \bP_nW_x-\bP W_x|\leq \sum_{l=1}^{q}\sup_{f\in  \mathcal{M}^{(j)}_{h,\alpha}}|\bP_nf-\bP f| \leq \sum_{j=1}^{q}\max_{l=1,\ldots,N(j,\epsilon)}|\bP_nu_{j,l}-\bP u_{j,l}|+q\epsilon \nonumber,
\end{align}
so that $\limsup_{W_x\in X\circ\mathcal{M}_{h,\alpha}}| \bP_nW_x-\bP W_x|$ is bounded by $q\epsilon$ a.s. as $n\to\infty$. Since $q$ is fixed and $\epsilon$ is arbitrary, we conclude the proof.
\end{proof}

With these Lemmas, we now prove Theorem \ref{thm:spectralconvergence_heatkernel} and Theorem \ref{thm:spectralconvergence_laplacian}. The proof consists of three steps. First, we relate the normalized GCL to an integral operator $\widehat{T}_{h,\alpha,n}$. Second, for a given fixed bandwidth $h>0$, we show a.s. spectral convergence of $\widehat{T}_{h,\alpha,n}$ to $T_{h,\alpha}$ when $n\to \infty$. Third, the spectral convergence of $T_{h,1}$ to $\nabla^2X$ in $L^2(\mathcal{E})$ as $h\to 0$ is proved. Finally, we put all ingredients together and finish the proof. Essentially the proof follows \cite{coifman_lafon:2006,belkin_niyogi:2007,VonLuxburg2008}, while we take care simultaneously the non-empty boundary, the nonuniform sampling and the non-trivial bundle structure. Note that when we work with the trivial principal bundle, that is, we work with the normalized GL, $\alpha=0$, the p.d.f. is uniform and the boundary is empty, then we recover the result in \cite{belkin_niyogi:2007}.
\newline 
\begin{proof}[Theorem \ref{thm:spectralconvergence_heatkernel} and Theorem \ref{thm:spectralconvergence_laplacian}]

\underline{\textbf{Step 1: Relationship between $\vD_{h,\alpha,n}^{-1}\vP_{h,\alpha,n}$ and $\widehat{T}_{h,\alpha,n}$.}}

We immediately have that
\begin{equation}\label{proof:step1:expansionBDTX}
(B^T_{\mathcal{X}}\delta_{\mathcal{X}}\widehat{T}_{h,\alpha,n}X)[i]=\frac{1}{n}\sum_{j=1}^n \widehat{M}_{h,\alpha,n}\left(x_i,x_j\right)u_i^{-1}\myP^{i}_{j}X(x_j)=(\vD_{h,\alpha,n}^{-1}\vP_{h,\alpha,n}\vX)[i],
\end{equation}
which leads to the relationship between the eigen-structure of $h^{-1}(\widehat{T}_{h,\alpha,n}-1)$ and $h^{-1}(\vD_{h,\alpha,n}^{-1}\vP_{h,\alpha,n}-\vI)$. Suppose $X$ is an eigen-section of $h^{-1}(\widehat{T}_{h,\alpha,n}-1)$ with eigenvalue $\lambda$. We claim that $\vX=B^T_{\mathcal{X}}\delta_{\mathcal{X}}X$ is an eigenvector of $\vD_{h,\alpha,n}^{-1}\vP_{h,\alpha,n}$ with eigenvalue $\lambda$. Indeed, for all $i=1,\ldots,n$,
\begin{align}
&h^{-1}[(\vD_{h,\alpha,n}^{-1}\vP_{h,\alpha,n}-\vI)\vX][i]=\,\frac{1}{hn}\sum_{j=1}^n \widehat{M}_{h,\alpha,n}\left(x_i,x_j\right)u^{-1}_i[\myP^{i}_{j}X(x_j)-X(x_i)]\nonumber\\
=&\,u^{-1}_i\frac{1}{hn}\sum_{j=1}^n \widehat{M}_{h,\alpha,n}\left(x_i,x_j\right)[\myP^{i}_{i}X(x_j)-X(x_i)]=\,u^{-1}_ih^{-1}(\widehat{T}_{h,\alpha,n}-1)X(x_i) = \lambda u^{-1}_iX(x_i)=\lambda\vX[i]\nonumber.
\end{align} 
On the other hand, given an eigenvector $\vv$ of $h^{-1}(\vD_{h,\alpha,n}^{-1}\vP_{h,\alpha,n}-\vI_{nd})$ with eigenvalue $\lambda$, that is,
\begin{equation}
(\vD_{h,\alpha,n}^{-1}\vP_{h,\alpha,n}\vv)[i] =(1+h\lambda) \vv[i].
\end{equation}
When $0\geq h\lambda>-1$, we show that there is an eigen-vector field of $h^{-1}(\widehat{T}_{h,\alpha,n}-1)$ with eigenvalue $\lambda$. In order to show this fact, we note that if $X$ is an eigen-vector field of $h^{-1}(\widehat{T}_{h,\alpha,n}-1)$ with eigenvalue $\lambda$ so that $0\geq h\lambda> -1$, it should satisfy
\begin{align}
X(x_i)=&\frac{\widehat{T}_{h,\alpha,n}X(x_i)}{1+h\lambda}=\frac{\frac{1}{n}\sum_{j=1}^n \widehat{M}_{h,\alpha,n}\left(x_i,x_j\right)\myP^{i}_{j}X(x_j)}{ 1+h\lambda }\nonumber\\
=&\,\frac{\frac{1}{n}\sum_{j=1}^n \widehat{M}_{h,\alpha,n}\left(x_i,x_j\right)\myP^{i}_{j}u_ju^{-1}_jX(x_j)}{1+h\lambda}=\frac{\frac{1}{n}\sum_{j=1}^n \widehat{M}_{h,\alpha,n}\left(x_i,x_j\right)\myP^{i}_{j}u_j\vX [j]}{1+h\lambda}.\label{step1:motivation1}
\end{align}
The relationship in (\ref{step1:motivation1}) leads us to consider the vector field
$$
X_{\vv}(x):=\frac{\frac{1}{n}\sum_{j=1}^n \widehat{M}_{h,\alpha,n}\left(x,x_j\right)\myP^{x}_{x_j}u_j\vv[j]}{1+h\lambda}
$$
to be the related eigen-vector field of $h^{-1}(\widehat{T}_{h,\alpha,n}-1)$ associated with $\vv$. To show this, we directly calculate:
\begin{align}
\widehat{T}_{h,\alpha,n}X_{\vv}(y)=&\, \frac{1}{n}\sum_{j=1}^n \widehat{M}_{h,\alpha,n}\left(y,x_j\right)\myP^{y}_{x_j}X_{\vv}(x_j)=\frac{1}{n}\sum_{j=1}^n \widehat{M}_{h,\alpha,n}\left(y,x_j\right)\myP^{y}_{x_j}\Big(\frac{\frac{1}{n}\sum_{k=1}^n \widehat{M}_{h,\alpha,n}\left(x_j,x_k\right)\myP^{x_j}_{x_k}u_k\vv[k]}{1+h\lambda}\Big)\nonumber\\
=&\,\frac{1}{1+h\lambda}\frac{1}{n}\sum_{j=1}^n \widehat{M}_{h,\alpha,n}\left(y,x_j\right)\myP^{y}_{x_j}(1+h\lambda) u_j\vv[j] = (1+h\lambda)X_{\vv}(y)\nonumber,
\end{align}
where the third equality comes from the expansion (\ref{proof:step1:expansionBDTX}) and the last equality comes from the definition of $X_{\vv}$.
Thus we conclude that $X_{\vv}$ is the eigen-vector field of $h^{-1}(\widehat{T}_{h,\alpha,n}-1)$ with eigenvalue $\lambda$ since $0\geq h\lambda>-1$. 

The above one to one relationship between eigenvalues and eigenfunctions of $h^{-1}(\widehat{T}_{h,\alpha,n}-1)$ and $h^{-1}(\vD_{h,\alpha,n}^{-1}\vP_{h,\alpha,n}-\vI)$ when $0\geq h\lambda>-1$ allows us to analyze the spectral convergence of $h^{-1}(\vD_{h,\alpha,n}^{-1}\vP_{h,\alpha,n}-\vI)$ by analyzing the spectral convergence of $h^{-1}(\widehat{T}_{h,\alpha,n}-1)$ to $h^{-1}(T_{h,\alpha}-1)$. A similar argument shows that when the eigenvalue of $\vD_{h,\alpha,n}^{-1}\vP_{h,\alpha,n}$ is between $(0,1]$, the eigen-structures of $\vD_{h,\alpha,n}^{-1}\vP_{h,\alpha,n}$ and $\widehat{T}_{h,\alpha,n}$ are again related. Note that in general the eigenvalues of $\vD_{h,\alpha,n}^{-1}\vP_{h,\alpha,n}$ might not negative when $n$ is finite, as is shown in (\ref{fact:MayBeNegative}).
\newline\newline
\underline{\textbf{Step 2:} compact convergence of $\widehat{T}_{h,\alpha,n}$ to $T_{h,\alpha}$ a.s. when $n\to \infty$ and $h$ is fixed.}

{\allowdisplaybreaks

Recall the definition of compact convergence of a series of operators \cite[p122]{Chatelin:2011} in $C(\mathcal{E})$ with the $L^\infty$ norm. We say that a sequence of operators $\mathtt{T}_{n}:C(\mathcal{E})\to C(\mathcal{E})$ compactly converges to $\mathtt{T}:C(\mathcal{E})\to C(\mathcal{E})$ if and only if
\begin{enumerate}
\item[(C1)] $\mathtt{T}_{n}$ converges to $\mathtt{T}$ pointwisely, that is, for all $X\in C(\mathcal{E})$, we have $\|\mathtt{T}_{n}X-\mathtt{T}X\|_{L^\infty(\mathcal{E})}\to 0$;
\item[(C2)] for any uniformly bounded sequence $\{X_l:\|X_l\|_{L^\infty}\leq 1\}_{l=1}^\infty\subset C(\mathcal{E})$, the sequence $\{(\mathtt{T}_{n}-\mathtt{T})X_l\}_{l=1}^\infty$ is relatively compact.
\end{enumerate}
Now we show (C1) -- the pointwise convergence of $\widehat{T}_{h,\alpha,n}$ to $T_{h,\alpha}$ a.e. when $h$ is fixed and $n\to\infty$. By a simple bound we have
\begin{align}
&\|\widehat{T}_{h,\alpha,n}X-T_{h,\alpha}X\|_{L^\infty(\mathcal{E})} = \sup_{y\in\MM} |\bP_{n}\widehat{M}_{h,\alpha,n}(y,\cdot)\myP^{y}_{\cdot}X(\cdot)-\bP M_{h,\alpha}(y,\cdot)\myP^{y}_{\cdot}X(\cdot)|\nonumber\\
\leq&\,\sup_{y\in\MM} |\bP_{n}\widehat{M}_{h,\alpha,n}(y,\cdot)\myP^{y}_{\cdot}X(\cdot)-\bP_n \widehat{M}^{(d_{h,\alpha})}_{h,\alpha,n}(y,\cdot)\myP^{y}_{\cdot}X(\cdot)|\label{proof:C1}\\
&+\,\sup_{y\in\MM} |\bP_{n}\widehat{M}^{(d_{h,\alpha})}_{h,\alpha,n}(y,\cdot)\myP^{y}_{\cdot}X(\cdot)-\bP_n M_{h,\alpha}(y,\cdot)\myP^{y}_{\cdot}X(\cdot)|\label{proof:C2}\\
&+\,\sup_{y\in\MM} |\bP_n M_{h,\alpha}(y,\cdot)\myP^{y}_{\cdot}X(\cdot)-\bP M_{h,\alpha}(y,\cdot)\myP^{y}_{\cdot}X(\cdot)|,\label{proof:C3}
\end{align}
where $\widehat{M}^{(d_{h,\alpha})}_{h,\alpha,n}(x,y) :=\,\frac{K_{h,\alpha}(x,y)}{\widehat{d}_{h,\alpha,n}(x)}\in C(\MM\times\MM)$. 

Rewrite (\ref{proof:C3}) as $\sup_{W\in X\circ \mathcal{M}_{h,\alpha}}\|\bP_{n}W-\bP W\|_{\RR^m}$. Since $u_i$ preserves the inner product structure, by Lemma \ref{class:glivenko-cantelli}, (\ref{proof:C3}) converges to 0 a.s. when $n\to \infty$. Next, by a direct calculation and the bound in Lemma \ref{lemma:boundsForAllKindsOfK}, we have
\begin{align}
&\sup_{y\in\MM} |\bP_{n}\widehat{M}_{h,\alpha,n}(y,\cdot)\myP^{y}_{\cdot}X(\cdot)-\bP_n \widehat{M}^{(d_{h,\alpha})}_{h,\alpha,n}(y,\cdot)\myP^{y}_{\cdot}X(\cdot)| \leq \|X\|_{L^\infty}\sup_{x,y\in\MM} \left|\frac{\widehat{K}_{h,\alpha,n}(x,y)-K_{h,\alpha}(x,y)}{\widehat{d}_{h,\alpha,n}(x)}\right|\nonumber\\
\leq&\,\|X\|_{L^\infty}\frac{\|K\|_{L^\infty}^{2\alpha}}{\delta}\sup_{x,y\in\MM}|\widehat{K}_{h,\alpha,n}(x,y)-K_{h,\alpha}(x,y)| \nonumber\\
\leq&\, \|X\|_{L^\infty}\frac{\|K\|_{L^\infty}^{2\alpha+1}}{\delta}\sup_{x,y\in\MM}\left|\frac{1}{\widehat{p}^\alpha_{h,n}(x)\widehat{p}^\alpha_{h,n}(y)}-\frac{1}{p^\alpha_{h}(x) p^\alpha_{h}(y)}\right|\nonumber\\
\leq&\,\frac{2\|X\|_{L^\infty}\|K\|_{L^\infty}^{2\alpha+1}}{\delta^{\alpha+1}}
\sup_{y\in\MM} \left|\frac{1}{\widehat{p}^\alpha_{h,n}(y)}-\frac{1}{p^\alpha_{h}(y)}\right|\leq \frac{2\|X\|_{L^\infty}\|K\|_{L^\infty}^{2\alpha+1}}{\delta^{3\alpha+1}}
\sup_{y\in\MM} \left| \widehat{p}^\alpha_{h,n}(y) - p^\alpha_{h}(y) \right|\nonumber \\
\leq&\, \frac{2\alpha\|X\|_{L^\infty}\|K\|_{L^\infty}^{2\alpha+1}}{\delta^{2\alpha+2}} \sup_{f\in \mathcal{K}_h}\|(\bP_n f) -(\bP f) \| \nonumber,
\end{align}
where the last inequality holds due to the fact that when $A,B\geq c>0$, $|A^\alpha-B^\alpha|\leq \frac{\alpha}{c^{1-\alpha}}|A-B|$ and $\widehat{p}_{h,n}(y), p_{h}(y)>\delta$ by Lemma \ref{lemma:boundsForAllKindsOfK}. Note that since $h$ is fixed, $\delta$ is fixed. Thus, the term (\ref{proof:C1}) converges to $0$ a.s. as $n\to \infty$ by Lemma \ref{class:glivenko-cantelli}.
The convergence of (\ref{proof:C2}) follows the same line:
\begin{align}
&\sup_{y\in\MM}|\bP_n \widehat{M}^{(d_{h,\alpha})}_{h,\alpha,n}(y,\cdot)\myP^{y}_{\cdot}X(\cdot)-\bP_n M_{h,\alpha}(y,\cdot)\myP^{y}_{\cdot}X(\cdot)| 
\leq \|X\|_{L^\infty}\|K\|_{L^\infty}\sup_{x\in\MM}\left| \frac{1}{\widehat{d}_{h,\alpha,n}(x)}-\frac{1}{d_{h,\alpha}(x)}\right|\nonumber\\
\leq &  \|X\|_{L^\infty}\|K\|^{3-2\alpha}_{L^\infty}\sup_{x\in\MM}| \widehat{d}_{h,\alpha,n}(x)-d_{h,\alpha}(x)|,\nonumber
\end{align}
where the last term is bounded by
\begin{align}
&\sup_{x\in\MM}| \widehat{d}_{h,\alpha,n}(x)-d_{h,\alpha}(x)| \leq \sup_{x\in\MM}|  \widehat{d}_{h,\alpha,n}(x)-\widehat{d}^{(p_h)}_{h,\alpha,n}(x)|+\sup_{x\in\MM}| \widehat{d}^{(p_h)}_{h,\alpha,n}(x)-d_{h,\alpha}(x)|\nonumber\\
\leq&\,\frac{\|K\|_{L^\infty}}{\delta^{3\alpha}}\sup_{x\in\MM}| \widehat{p}^\alpha_{h,n}(x)-p^\alpha_{h}(x)|+\|K\|_{L^\infty}\sup_{f\in \mathcal{K}_{h,\alpha}}\| \bP_n f-\bP f\|, 
\end{align}
where $\widehat{d}^{(p_h)}_{h,\alpha,n}(x) :=\,\frac{1}{n}\sum_{k=1}^n K_{h,\alpha}(x,x_k)\in C(\MM)$, which again converges to $0$ a.s. as $n\to \infty$ by Lemma \ref{class:glivenko-cantelli}. We thus conclude the pointwise convergence of $\widehat{T}_{h,\alpha,n}$ to $T_{h,\alpha}$ a.e. as $n\to\infty$.

Next we check the condition (C2). 
Since $T_{h,\alpha}$ is compact, the problem is reduced to show that $\widehat{T}_{h,\alpha,n}X_n$ is pre-compact for any given sequence of vector fields $\{X_1,X_2,\ldots\}\subset C(\mathcal{E})$ so that $\|X_l\|_{L^\infty}\leq1$ for all $l\in\NN$. We count on the Arzela-Ascoli theorem \cite[IV.6.7]{Dunford_Schwartz:1958} to finish the proof. By Lemma \ref{lemma:boundsForAllKindsOfK}, a direct calculation leads to
$$
\sup_{n\geq1}\|\widehat{T}_{h,\alpha,n}X_n\|_{L^\infty}=\sup_{n\geq1,y\in\MM}\left|\frac{1}{n}\sum_{j=1}^n\widehat{M}_{h,\alpha,n}(y,x_i)\myP^{y}_{x_i}X_n(x_i)\right|\leq\frac{\|K\|^{2\alpha+1}_{L^\infty}}{\delta^{2\alpha+1}},
$$
which guarantees the uniform boundedness. Next we show the equi-continuity of $\widehat{T}_{h,\alpha,n}X_n$. For a given pair of close points $x\in \MM$ and $y\in \MM$, a direct calculation leads to
\begin{align}
&|\widehat{T}_{h,\alpha,n}X_n(y)-\myP^{y}_{x}\widehat{T}_{h,\alpha,n}X_n(x)| = \left| \bP_{n}\widehat{M}_{h,\alpha,n}(y,\cdot)\myP^{y}_{\cdot}X_n(\cdot)-\myP^{y}_{x}\bP_{n}\widehat{M}_{h,\alpha,n}(x,\cdot)\myP^{x}_{\cdot}X_n(\cdot)\right|\nonumber \\
\leq&\,\|X_n\|_{L^\infty}\sup_{z\in\MM}\left| \bP_{n}\widehat{M}_{h,\alpha,n}(y,z)-\bP_{n}\widehat{M}_{h,\alpha,n}(x,z)\right| \leq \sup_{z\in\MM}\left|\widehat{M}_{h,\alpha,n}(y,z)-\widehat{M}_{h,\alpha,n}(x,z)\right|\nonumber\\
\leq&\, \frac{\|K\|^{4\alpha}_{L^\infty}}{\delta^2}\sup_{z\in\MM}\left|\widehat{d}_{h,\alpha,n}(y)\widehat{K}_{h,\alpha,n}(x,z)-\widehat{d}_{h,\alpha,n}(x)\widehat{K}_{h,\alpha,n}(y,z)\right|\nonumber\\
\leq&\,\frac{\|K\|^{4\alpha+1}_{L^\infty}}{\delta^{2+2\alpha}}\Big( \sup_{z\in\MM}\left|\widehat{K}_{h,\alpha,n}(x,z)-\widehat{K}_{h,\alpha,n}(y,z)\right| +  \big|\widehat{d}_{h,\alpha,n}(y)-\widehat{d}_{h,\alpha,n}(x)\big|\Big)\nonumber\\
\leq&\,\frac{\|K\|^{4\alpha+1}_{L^\infty}}{\delta^{2+2\alpha}}\Big(\sup_{z\in\MM}\left| \widehat{K}_{h,\alpha,n}(x,z)-\widehat{K}_{h,\alpha,n}(y,x)\right|+|\widehat{d}_{h,\alpha,n}(y)-\widehat{d}^{(p_n)}_{h,\alpha,n}(y)|\nonumber\\
&\qquad\qquad\qquad+|\widehat{d}^{(p_n)}_{h,\alpha,n}(y)-\widehat{d}^{(p_n)}_{h,\alpha,n}(x)|+|\widehat{d}^{(p_n)}_{h,\alpha,n}(x)-\widehat{d}_{h,\alpha,n}(x)|\Big)\nonumber\\
\leq&\,\frac{\|K\|^{4\alpha+1}_{L^\infty}}{\delta^{2+2\alpha}}\Big(\sup_{z\in\MM}\left|\widehat{K}_{h,\alpha,n}(x,z)-\widehat{K}_{h,\alpha,n}(y,z)\right| +|\widehat{d}^{(p_n)}_{h,\alpha,n}(y)-\widehat{d}^{(p_n)}_{h,\alpha,n}(x)|+2\|\widehat{d}^{(p_h)}_{h,\alpha,n}-\widehat{d}_{h,\alpha,n}\|_{L^\infty}\Big)\nonumber,
\end{align}
where the last term is further controlled by $|\widehat{d}^{(p_n)}_{h,\alpha,n}(y)-\widehat{d}^{(p_n)}_{h,\alpha,n}(x)|\leq \sup_{z\in\MM}|K_{h,\alpha}(y,z)-K_{h,\alpha}(x,z)|$,
\begin{align}
&\,\sup_{z\in\MM}\left|\widehat{K}_{h,\alpha,n}(x,z)-\widehat{K}_{h,\alpha,n}(y,z)\right|\nonumber\\
\leq\,& \sup_{z\in\MM}\frac{1}{\widehat{p}^\alpha_{h,n}(x)\widehat{p}^\alpha_{h,n}(y)\widehat{p}^\alpha_{h,n}(z)}\left|\widehat{p}^\alpha_{h,n}(y)K_h(x,z)-\widehat{p}^\alpha_{h,n}(x)K_h(y,z)\right|\nonumber\\
\leq\,& \frac{1}{\delta^{3\alpha}}\left(\sup_{z\in\MM}\widehat{p}^\alpha_{h,n}(y)|K_h(x,z)-K_h(y,z)|+\sup_{z\in\MM} K_h(y,z)|\widehat{p}^\alpha_{h,n}(y)-\widehat{p}^\alpha_{h,n}(x)|\right)\nonumber\\
\leq\,&\frac{\|K\|_{L^\infty}}{\delta^{3\alpha}}\left(\sup_{z\in\MM} |K_h(x,z)-K_h(y,z)|+\sup_{z\in\MM}  \frac{\alpha}{\delta^{1-\alpha}}|\widehat{p}_{h,n}(y)-\widehat{p}_{h,n}(x)|\right)\nonumber\\
\leq\,&\frac{\|K\|_{L^\infty}}{\delta^{3\alpha}}\left(\sup_{z\in\MM} |K_h(x,z)-K_h(y,z)|+\frac{\alpha}{\delta^{1-\alpha}}\sup_{z\in\MM}  |K_{h}(y,z)-K_{h}(x,z)|\right)\nonumber, 
\end{align} 
and similarly
\begin{align*}
& \|\widehat{d}^{(p_h)}_{h,\alpha,n}-\widehat{d}_{h,\alpha,n}\|_{L^\infty}=\sup_{z\in\MM}\left|\frac{1}{n}\sum_{k=1}^n\left(\frac{K_h(z,x_k)}{\widehat{p}^\alpha_{h,n}(z)\widehat{p}^\alpha_{h,n}(x_k)}-\frac{K_h(z,x_k)}{p^\alpha_h(z)p^\alpha_h(x_k)}\right)\right|\\
\leq\,&\frac{\|K\|_{L^\infty}}{\delta^\alpha}\sup_{z\in\MM} \left|\frac{1}{\widehat{p}^\alpha_{h,n}(z)\widehat{p}^\alpha_{h,n}(x_k)}-\frac{1}{p^\alpha_h(z)p^\alpha_h(x_k)}\right|\leq \frac{\|K\|_{L^\infty}}{\delta^{3\alpha}} \sup_{z\in\MM}\left| \widehat{p}^\alpha_{h,n}(z) - p^\alpha_h(z) \right|\\
\leq\,& \frac{\|K\|_{L^\infty}}{\delta^{3\alpha}} \frac{\alpha}{\delta^{1-\alpha}} \left\| \widehat{p}_{h,n}- p_h \right\|_{L^\infty}.
\end{align*}
As a result, we have the following bound
\begin{align*}
|\widehat{T}_{h,\alpha,n}X_n(y)-\myP^{y}_{x}\widehat{T}_{h,\alpha,n}X_n(x)| 
\leq \frac{\|K\|^{4\alpha+2}_{L^\infty}}{\delta^{2+5\alpha}}\left( \big(1+\frac{\alpha}{\delta^{1-\alpha}}\big) \sup_{z\in\MM} |K_h(x,z)-K_h(y,z)|\right.\\
\qquad \left.+\frac{\delta^{3\alpha}}{\|K\|_{L^\infty}}\sup_{z\in\MM}|K_{h,\alpha}(y,z)-K_{h,\alpha}(x,z)|+ \frac{\alpha}{\delta^{1-\alpha}} \left\| \widehat{p}_{h,n}- p_h \right\|_{L^\infty}\right).
\end{align*}
Thus, when $y\to x$, $\sup_{z\in\MM} |K_h(x,z)-K_h(y,z)|$ and $\sup_{z\in\MM}|K_{h,\alpha}(y,z)-K_{h,\alpha}(y,z)|$ both converge to $0$ since $K_{h}$ and $K_{h,\alpha}$ are both continuous. Also, $\left\| \widehat{p}_{h,n}- p_h \right\|_{L^\infty}$ converges to $0$ a.s. as $n\to\infty$ by the Glivenko-Cantali property; that is, for a given small $\epsilon>0$, we can find $N>0$ so that $\left\| \widehat{p}_{h,n}- p_h \right\|_{L^\infty}\leq \epsilon$ a.s. for all $n\geq N$. Thus, by the Arzela-Ascoli theorem, we have the compact convergence of $\widehat{T}_{h,\alpha,n}$ to $T_{h,\alpha}$ a.s. when $n\to \infty$.

} 

Since the compact convergence implies the spectral convergence (see \cite{Chatelin:2011} or Proposition 6 in \cite{VonLuxburg2008}), we get the spectral convergence of $\widehat{T}_{h,\alpha,n}$ to $T_{h,\alpha}$ almost surely when $n\to \infty$.
\newline\newline
\underline{\textbf{Step 3: Spectral convergence of $T^{t/h}_{h,1}$ to $e^{t\nabla^2}$ and $h^{-1}(T_{h,1}-1)$ to $\nabla^2$ in $L^2(\mathcal{E})$ as $h\to 0$.}}

First we consider the case when $\partial\MM=\emptyset$. We assume $\frac{\mu_2}{d}=1$ to simplify the notation. Denote $-\lambda_i$, where $\lambda_i>0$, to be the $i$-th eigenvalue of the connection Laplacian $\nabla^2$ with the associated eigen-vector field $X_i$. Order $\lambda_i$ so that it increases as $i$ increases.  Fix $l_0\geq 0$. For all $x\in\MM$, by Proposition \ref{thm:pointwise_conv_to_integral} we have uniformly
\[
\frac{T_{\epsilon,1}X_l(x)-X_l(x)}{h}=\nabla^2X_l(x)+O(h),
\]
where $O(h)$ depends on $\|X^{(k)}_l\|_{L^\infty(\mathcal{E})}$, where $k=0,1,2,3$. By the Sobolev embedding theorem \cite[Theorem 9.2]{Palais1968}, for all $l\leq l_0$ we have
\[
\| X^{(3)}_l\|_{L^\infty(\mathcal{E})} \lesssim \|X_l\|_{H^{d/2+4}(\mathcal{E})} \lesssim \left(1+\| (\nabla^2)^{d/4+2}X_l\|_{L^2(\mathcal{E})}\right) =1+ \lambda_l^{d/4+2}\leq 1+\lambda_{l_0}^{d/4+2},
\]
where we choose $d/2+4$ for convenience. Thus, in the $L^2$ sense, for all $l\leq l_0$ 
\begin{equation}\label{Ieps00}
\left\|\frac{T_{h,1}X_l-X_l}{h}-\nabla^2X_l\right\|_{L^2(\mathcal{E})} =  O\left(\left(1+\lambda_{l_0}^{d/4+2}\right)h\right).
\end{equation}
In addition to $h<\frac{1}{2\lambda_{l_0}}$, if we choose $h$ so that $h\left(1+\lambda_{l_0}^{d/4+2}\right)\leq h^{1/2}$ and $h\left(1+\lambda_{l_0+1}^{d/4+2}\right)>h^{1/2}$, which is equivalent to $\lambda_{l_0}\leq h^{-2/(d+8)}<\lambda_{l_0+1}$, we reach the fact that
$$
 \left\| \frac{T_{h,1}-1}{h} -\nabla^2 \right\|_{L^2(\mathcal{E})} =O\left( h^{1/2}\right)
$$
on $\overline{\oplus_{k\leq l_0} E_k}$. Thus, on the finite dimensional subspace $\overline{\oplus_{k\leq l_0} E_k}$, as $h\to 0$, $h^{-1}(T_{h,1}-1)X$ spectrally converges to $\nabla^2X$ in the boundary-free case.

Next we show $T_{h,1}^{t/h}$ converges to $e^{t\nabla^2}$ for $t>0$ as $h\to 0$. Note that we have from (\ref{Ieps00})
\begin{equation}\label{Ieps1}
\|T_{h,1}-I-h\nabla^2\|=O\left(\lambda_{l_0}^{d/4+2}h^2\right)
\end{equation}
on $\overline{\oplus_{k\leq l_0} E_k}$. 
When $h<\frac{1}{2\lambda_{l_0}}$, $I+h\nabla^2$ is invertible on $\overline{\oplus_{k\leq l} E_k}$ with norm $\frac{1}{2}\leq\|I+h\nabla^2\|<1$. So, by the binomial expansion, for all $l\leq l_0$ we have
$$
(I+h\nabla^2)^{t/h}X_l=(1-t\lambda_l+t^2\lambda_l^2/2-th\lambda_l^2/2+\ldots)X_l.
$$
On the other hand, we have
$$
e^{t\nabla^2}X_l=(1-t\lambda_l+t^2\lambda_l^2/2+\ldots)X_l.
$$
Therefore, when $ h<\frac{1}{2\lambda_{l_0}}$ we have on $\overline{\oplus_{k\leq l_0} E_k}$
\begin{align}
e^{t\nabla^2}=(I+h\nabla^2)^{t/h}+O(\lambda_{l_0}^2th)\label{Ieps2}.
\end{align}
Now we put the above together. Take $ h<\frac{1}{2\lambda_{l_0}}$ small enough so that
\begin{align}
\|T_{h,1}-I-h\nabla^2\|=O\left(\lambda_{l_0}^{d/4+2}h^2\right)\leq 1/2.\label{spec_conv:assumption_for_h}
\end{align}
Then we have for $X\in E_{l_0}$ 
\begin{align*}
\left\|(T^{t/h}_{h,1}-e^{t\nabla^2})X\right\|&=
\left\|\left(\left(I+h\nabla^2+O\left(\lambda_{l_0}^{d/4+2}h^2\right)\right)^{t/h}-\left(I+h\nabla^2\right)^{t/h}-O\left(\lambda_{l_0}^2th\right)\right)X\right\|\\
&= (1-h\lambda_{l_0})^{t/h}\left|\left[1+O\left(\left(1-h\lambda_{l_0}\right)^{-1}\lambda_{l_0}^{d/4+2}h^2\right) \right]^{t/h}-1-O\left(\left(1-h\lambda_{l_0}\right)^{-1}\lambda_{l_0}^2th\right) \right|\\
&\leq \left|\left[1+O\left(\left(1+h\lambda_{l_0}\right)^{-1}\lambda_{l_0}^{d/4+2}h^2\right) \right]^{t/h}-1-O\left(\left(1-h\lambda_{l_0}\right)^{-1}\lambda_{l_0}^2th\right) \right|\\
&=  O\left(\lambda_{l_0}^{d/4+2}th\right) ,
\end{align*}
where the first equality comes from (\ref{Ieps1}) and (\ref{Ieps2}), the second inequality comes from the fact that $h<\frac{1}{2\lambda_{l_0}}$, and the last inequality comes from the binomial expansion, that is, $\left[1+O\left(\left(1+h\lambda_{l_0}\right)^{-1}\lambda_{l_0}^{d/4+2}h^2\right) \right]^{t/h}\approx 1+O\left(\left(1+h\lambda_{l_0}\right)^{-1}\lambda_{l_0}^{d/4+2}th\right) $ when $h$ and $O\left(\left(1+h\lambda_{l_0}\right)^{-1}\lambda_{l_0}^{d/4+2}h^2\right)$ are small enough, and the fact that $\lambda_{l_0}^{d/4+2}h >\lambda_{l_0}^{2}h$ when $l_0$ is large enough, i.e., when $\lambda_{l_0}>1$. Thus, over $\overline{\oplus_{k\leq l_0} E_k}$, we have $\|T^{t/h}_{h,1}-e^{t\nabla^2}\|=O\left(\lambda_{l_0}^{d/4+2}th\right) $.
Furthermore, in addition to $h<\frac{1}{2\lambda_{l_0}}$, if we choose $h$ so that $h\lambda_{l_0}^{d/4+2}\leq h^{1/2}$ and $h\lambda_{l_0+1}^{d/4+2}>h^{1/2}$, which is equivalent to $\lambda_{l_0}\leq h^{-2/(d+8)}<\lambda_{l_0+1}$, we reach the fact that
$$
\|T^{t/h}_{h,1}-e^{t\nabla^2}\|=O\left(th^{1/2}\right)
$$
on $\mathcal{H}_h:=\overline{\oplus_{l: \lambda_l<h^{-2/(d+8)}} E_l}$. Thus, we conclude the spectral convergence of $T^{t/h}_{h,1}$ while the boundary is empty.

When $\partial\MM\neq\emptyset$, the proof is essentially the same, except that we have to take the boundary effect (\ref{bdryrslt3}) and the Neuman's condition into account. Here we give the proof quickly and only show the main different parts. We adapt the same notations as those used in the boundary-free case. By Proposition \ref{thm:pointwise_conv_to_integral} and the Sobolev embedding theorem, for $x\in\MM\backslash \MM_{h^\gamma}$, we have 
\[
 T_{\epsilon,1}X_l(x)-X_l(x) =h\nabla^2X_l(x)+O\left(\lambda_{l_0}^{d/4+2}h^2\right);
\]
for $x\in\MM_{h^\gamma}$, by the Neuman's condition, we have
\[
T_{h,1} X_l(x)=\myP^{x}_{x_0}X(x_0)+O\left(\lambda_{l_0}^{d/4+3/2}h^{2\gamma}\right).
\]
Thus, we have on $\overline{\oplus_{k\leq l_0} E_k}$
$$
\left\|\frac{T_{h,1}-I}{h}- \nabla^2\right\|=O\left(\lambda_{l_0}^{d/4+2}h^{5\gamma/2-1}\right),
$$ 
where $5\gamma/2-1>0$ since $\gamma>2/5$.
Thus, by the same argument, if we choose $h$ much smaller so that $h^{5\gamma/2-1}\lambda_{l_0}^{d/4+2}\leq h^{4\gamma/5-1/2}$ and $h^{5\gamma/2-1}\lambda_{l_0}^{d/4+2}>h^{5\gamma/4-1/2}$, which is equivalent to $\lambda_{l_0}\leq h^{-(5\gamma/4-1/2)/(d/4+2)}<\lambda_{l_0+1}$, we reach the conclusion when the boundary is not empty.
\newline\newline
\underline{\textbf{Final step: Putting everything together.}}

We now finish the proof of Theorem \ref{thm:spectralconvergence_heatkernel} here. Fix $i$ and denote $\mu_{t,i,h}$ to be the $i$-th eigenvalue of $T_{h,1}$ with the associated eigenvector $Y_{t,i,h}$.
By Step 1, we know that all the eigenvalues inside $(-1/h,0]$ of $\widehat{T}_{h,1,n}$ and $\vD_{h,1,n}^{-1}\vP_{h,1,n}$ are the same and their eigenvectors are related. By Step 2, since we have the spectral convergence of $\widehat{T}_{h,1,n}$ to $T_{h,1}$ almost surely as $n\to\infty$, for each $j\in\NN$ large enough, we have by the definition of convergence in probability that for $h_j=1/j$, we can find $n_j\in\NN$ so that
$$
P\{ \|Y_{t,i,h_j}-Y_{t,i,h_j,n_j} \|_{L^2(\mathcal{E})}\geq 1/j\}\leq 1/j. 
$$   
Take $n_j$ as an increasing sequence. By step 3(A) (or Step 3(B) if $\partial\MM\neq \emptyset$), for each $j\in\NN$, there exists $j'>0$ so that
$$
\|Y_{t,i}-Y_{t,i,h_{j'}}\|_{L^2(\mathcal{E})}< 1/2j' .
$$
Arrange $j'$ as an increasing sequence. Similar statements hold for $\mu_{t,i,h}$. Thus, for all $j\in \NN$ large enough, there exists $j'\in\NN$ and hence $n_{j'}\in\NN$ so that 
\[
P\{ \|Y_{t,i}-Y_{t,i,h_{j'},n_{j'}} \|_{L^2(\mathcal{E})}\geq 1/j\}\leq P\{ \|Y_{t,i,h_{j'}}-Y_{t,i,h_{j'},n_{j'}} \|_{L^2(\mathcal{E})}\geq 1/2j'\}1/2j'.
\]
Therefore we conclude the convergence in probability. Since the proof for Theorem \ref{thm:spectralconvergence_laplacian} is the same, we skip it.
\end{proof}

\section{Extract more Topological/Geometric Information from a Point Cloud}\label{section:extract_more_info}

In Section \ref{section:review}, we understand VDM under the assumption that we have an access to the principal bundle structure of the manifold. However, in practice the knowledge of the bundle structure is not always available and we may only have access to the point cloud sampled from the manifold. Is it possible to obtain any principal bundle under this situation? The answer is yes if we restrict ourselves to a special principal bundle, the frame bundle.

We summarize the proposed reconstruction algorithm considered in \cite{singer_wu:2012} below. Take a point cloud $\mathcal{X}=\{x_i\}_{i=1}^n$ sampled from $\MM$ under Assumption \ref{Assumption:A} (A1), Assumption \ref{Assumption:B} (B1) and Assumption \ref{Assumption:B} (B2). The algorithm consists of the following three steps:
\begin{enumerate}
\item[(Step a)] Reconstruct the frame bundle from $\mathcal{X}$.
It is possible since locally a manifold can be well approximated by an affine space up to second order \cite{little_jung_maggioni:2009,Wang_Jiang_Wu_Zhou:2011,singer_wu:2012,Gong_Zhao_Medioni:2012,Kaslovsky_Meyer:2014,Wang_Slavakis_Lerman:2014,Arias-Castro_Lerman_Zhang:2014}. Thus, the embedded tangent bundle is estimated by local principal component analysis (PCA) with the kernel bandwidth $h_{\textup{pca}}>0$. Indeed, the top $d$ eigenvectors, $v_{x,k}\in\RR^p$, $k=1,\ldots,d$, of the covariance matrix of the dataset near $x\in\MM$, $\mathcal{N}_x:=\{x_j\in\mathcal{X};\,\|x-x_j\|_{\RR^p}\leq \sqrt{h_{\textup{pca}}}\}$, are chosen to form the estimated basis of the embedded tangent plane $\iota_*T_{x}\MM$. Denote $O_x$ to be a $p\times d$ matrix, whose $k$-th column is $v_{x,k}$. Note that $x$ may or may not be in $\mathcal{X}$.
Here $O_x$ can be viewed as an estimation of a point $u_x$ of the frame bundle such that $\pi(u_x)=x$. When $x=x_i\in \mathcal{X}$, we use $O_i$ to denote $O_{x_i}$. See \cite{singer_wu:2012} for details.
\item[(Step b)] Estimate the parallel transport between tangent planes by aligning $O_x$ and $O_y$ by
\begin{equation*}
O_{x,y} = \argmin_{O \in O(d)} \|O - O_x^T O_y\|_{HS}\in O(d),
\end{equation*}
where $\| \cdot \|_{HS}$ is the Hilbert-Schmidt norm. It is proved that $O_{x,y}$ is an approximation of the parallel transport from $y$ to $x$ when $x$ and $y$ are close enough in the following sense \cite[(B.6)]{singer_wu:2012}:
\begin{align*} 
O_{x,y}\overline{X}_y\approx O_x^T\iota_*\myP^{x}_{y}X(y),
\end{align*}
where $X\in C(T\MM)$ and $\overline{X}_y=O_y^T\iota_*X(y)\in \RR^d$ is the coordinate of $X(y)$ with related to the estimated basis. Note that $x$ and $y$ may or may not be in $\mathcal{X}$. When $x=x_i\in\mathcal{X}$ and $y=x_j\in\mathcal{X}$, we use $O_{ij}$ to denote $O_{x_i,x_j}$ and $\overline{X}_j$ to denote $\overline{X}_{x_j}$;
\item[(Step c)] Build GCL mentioned in Section \ref{section:review} based on the connection graph  from $\mathcal{X}$ and $\{O_{ij}\}$. We build up a block matrix $\vP^{\textup{O}}_{h,\alpha,n}$ with $d\times d$ entries, where $h>h_{\text{pca}}$:
\begin{equation*}
\vP^{\textup{O}}_{h,\alpha,n}(i,j) = \left\{\begin{array}{lcl}
                  \widehat{K}_{h,\alpha,n}(x_i,x_j)O_{ij} &  & (i,j)\in E, \\
                  0_{d\times d} &  & (i,j)\notin E,
                \end{array}
 \right.
\end{equation*}
where $0\leq \alpha\leq 1$ and the kernel $K$ satisfies Assumption \ref{Assumption:K}, and a $n\times n$ diagonal block matrix $\vD_{h,\alpha,n}$ with $d\times d$ entries defined in (\ref{def:Dbundle_all}). Denote operators $O^T_{\mathcal{X}}:T\MM_{\mathcal{X}}\to V_{\mathcal{X}}$, $O_{\mathcal{X}}:V_{\mathcal{X}}\to T\MM_{\mathcal{X}}$
\begin{align*}
&O_{\mathcal{X}}\vv:=[\iota^T_*O_1\vv[1],\ldots \iota^T_*O_n\vv[n]]\in T\MM_{\mathcal{X}},\nonumber\\ 
&O_{\mathcal{X}}^T\vw:=[(O^T_1\iota_*\vw[1])^T,\ldots, (O^T_n\iota_*\vw[n])^T]^T\in V_{\mathcal{X}},\nonumber 
\end{align*}
where $\vw\in T\MM_{\mathcal{X}}$ and $\vv\in V_{\mathcal{X}}$. Here $V_{\mathcal{X}}$ means the coordinates of a set of embedded tangent vectors on $\MM$ with related to the estimated basis of the embedded tangent plane. The pointwise convergence of GCL has been shown in \cite[Theorem 5.3]{singer_wu:2012}; that is, a.s. we have
\begin{align*}
&\lim_{h\rightarrow0}\lim_{n\rightarrow\infty}\frac{1}{h}(\vD_{h,\alpha,n}^{-1}\vP^{\textup{O}}_{h,\alpha,n}\bvX-\bvX)[i] = \frac{\mu^{(0)}_{1,2}}{2d}O_i^T\iota_*\left\{\nabla^2X(x_i)+\frac{2\nabla X(x_i)\cdot\nabla(\sfp^{1-\alpha})(x_i)}{\sfp^{1-\alpha}(x_i)}\right\},
\end{align*}
where $X\in C^4(T\MM)$ and $\bvX=O_{\mathcal{X}}^T\delta_{\mathcal{X}}X$. This means that by taking $\alpha=1$, we reconstruct the connection Laplacian associated with the tangent bundle $T\MM$.
\end{enumerate}
Note that the errors introduced in (a) and (b) may accumulate and influence spectral convergence of the GCL. 
In this section we study the spectral convergence under this setup which answers our question in the beginning and affirms that we are able to extract further geometric/topological information simply from the point cloud. 

\begin{definition}
Define operators $\widetilde{T}^{\textup{O}}_{h,\alpha,n}:C(T\MM)\to C(T\MM)$ as
\begin{align}
\widetilde{T}^{\textup{O}}_{h,\alpha,n}X(y)&=\,\iota_*^TO_y\frac{1}{n}\sum_{j=1}^n \widehat{M}_{h,\alpha,n}\left(y,x_j\right)O_{y,x_j}O_j^T\iota_*X(x_j).\nonumber
\end{align}
\end{definition}

The main result of this section is the following spectral convergence theorems stating the spectral convergence of $( \vD_{h,1,n}^{-1}\vP^{\textup{O}}_{h,1,n})^{t/h}$ to $e^{t\nabla^2}$ and $h^{-1}( \vD_{h,1,n}^{-1}\vP^{\textup{O}}_{h,1,n}-\vI_{dn})$ to $\nabla^2$. Note that except the estimated parallel transport, the statements of Theorem \ref{thm:spec_conv_pointcloud_heatkernel} and Theorem \ref{thm:spec_conv_pointcloud_laplacian} are the same as those of Theorem \ref{thm:spectralconvergence_heatkernel} and Theorem \ref{thm:spectralconvergence_laplacian}.

\begin{theorem}\label{thm:spec_conv_pointcloud_heatkernel}
Assume Assumption \ref{Assumption:A} (A1), Assumption \ref{Assumption:B} (B1), Assumption \ref{Assumption:B} (B2) and Assumption \ref{Assumption:K} hold. Estimate the parallel transport and construct the GCL by Step a, Step b and Step c. 
Fix $t>0$. Denote $\widetilde{\mu}_{t,i,h,n}$ to be the $i$-th eigenvalue of
$(\widetilde{T}^{\textup{O}}_{h,1,n})^{t/h}$
with the associated eigenvector $\widetilde{Y}_{t,i,h,n}$. Also denote $\mu_{t,i}>0$ to be the $i$-th eigenvalue of the heat kernel of the connection Laplacian $e^{t\nabla^2}$ with the associated eigen-vector field $Y_{t,i}$. We assume that both $\mu_{t,i,h,n}$ and $\mu_{t,i}$ decrease as $i$ increase, respecting the multiplicity. Fix $i\in\NN$. Then there exists a sequence $h_n\to 0$ such that $\lim_{n\to \infty}\widetilde{\mu}_{t,i,h_n,n}=\mu_{t,i}$ and $\lim_{n\to \infty}\|\widetilde{Y}_{t,i,h_n,n}-Y_{t,i}\|_{L^2(T\MM)}=0$ in probability.
\end{theorem}

\begin{theorem}\label{thm:spec_conv_pointcloud_laplacian}
Assume Assumption \ref{Assumption:A} (A1), Assumption \ref{Assumption:B} (B1), Assumption \ref{Assumption:B} (B2) and Assumption \ref{Assumption:K} hold. Estimate the parallel transport and construct the GCL by Step a, Step b and Step c. 
Denote $-\widetilde{\lambda}_{i,h,n}$ to be the $i$-th eigenvalue of
$h^{-1}(\widehat{T}^{\textup{O}}_{h,1,n}-1)$
with the associated eigenvector $\widetilde{X}_{i,h,n}$. Also denote $-\lambda_i$, where $\lambda_i>0$, to be the $i$-th eigenvalue of the connection Laplacian $\nabla^2$ with the associated eigen-vector field $X_i$. We assume that both $\lambda_{i,h,n}$ and $\lambda_i$ increase as $i$ increase, respecting the multiplicity. Fix $i\in\NN$. Then there exists a sequence $h_n\to 0$ such that $\lim_{n\to \infty}\widetilde{\lambda}_{i,h_n,n}=\lambda_i$ and $\lim_{n\to \infty}\|\widetilde{X}_{i,h_n,n}-X_i\|_{L^2(T\MM)}=0$ in probability.
\end{theorem}

The proofs of Theorem \ref{thm:spec_conv_pointcloud_heatkernel} and Theorem \ref{thm:spec_conv_pointcloud_laplacian} are essentially the same as those of Theorem \ref{thm:spectralconvergence_heatkernel} and Theorem \ref{thm:spectralconvergence_laplacian} except the fact that we lack the knowledge of the parallel transport. Indeed, in (\ref{VDM:gij}) the parallel transport is assumed to be accessible to the data analyst while in this section we only have access to the point cloud. Thus, the key ingredient of the proofs of Theorem \ref{thm:spec_conv_pointcloud_heatkernel} and Theorem \ref{thm:spec_conv_pointcloud_laplacian} is controlling the error terms coming from the estimation of the tangent plane and the parallel transport, while the other comments and details are the same as those in Section \ref{proof:PointwiseSpectralConvergence}. To better appreciate the role of these two estimations, we assume here that we have access to the embedding $\iota$ and the knowledge of the embedded tangent bundle. Precisely, suppose we have access to the embedded tangent plane, which is an affine space inside $\RR^p$, but the embedding $\iota$ and the parallel transport are not accessible to us. 
Denote the basis of the embedded tangent plane $\iota_*T_{x_i}\MM$ to be a $p\times d$ matrix $Q_i$.
By \cite[(B.68)]{singer_wu:2012}, we can approximate the parallel transport from $x_i$ to $x_j$ from $Q_i$ and $Q_j$ with a tolerable error; that is,
\begin{equation*}
\myP^{i}_{j}X(x_j)\approx \iota_*^TQ_iQ_{ij}Q^T_j\iota_*X(x_j),
\end{equation*}
where $Q_{ij} := \argmin_{O \in O(d)} \|O - Q_i^T Q_j\|_{HS}$.
Notice that even we know bases of these embedded tangent planes, the optimization step to obtain $Q_{ij}$ is still needed since in general $Q_i^TQ_j$ is not orthogonal due to the curvature. With the above discussion, we know that if the embedded tangent bundle information is further missing and we have to estimate it from the point cloud, another resource of error comes to play. Indeed, denote the estimated embedded tangent plane by a $p\times d$ matrix $O_x$. In \cite{singer_wu:2012}, it has been shown that
\begin{align*}
O^T_{x}\iota_*X(x)\approx Q^T_i\iota_*X(x).
\end{align*}
This approximation is possible due to the following two facts. First, by definition locally a manifold is isomorphic to the Euclidean space up to a second order error depending on the curvature. Second, the embedding $\iota$ is smooth so locally the manifold is distorted up to the Jacobian of $\iota$. We point out that in \cite{singer_wu:2012} we focus on the pointwise convergence so the error terms in \cite{singer_wu:2012} were simplified by the big O notations. \newline
\begin{proof}[Proof of Theorem \ref{thm:spec_conv_pointcloud_heatkernel} and Theorem \ref{thm:spec_conv_pointcloud_laplacian}] 

\underline{\textbf{Step 1: Estimate the frame bundle and connection.}}

Here we give an outline of the proof and indicate how the error terms look like. We refer the reader to \cite{singer_wu:2012} for the other details. Recall the following results in \cite[Theorem B.1]{singer_wu:2012} under a proper choice of the kernel bandwidth $h_{\text{pca}}\ll h$:
\begin{align*}
O^T_i\iota_*X(x_i)=Q^T_i\iota_*X(x_i)+O_p(h^{3/2}),
\end{align*}
where $x_i\in\mathcal{X}$ and the $O(h^{3/2})$ term contains both the bias error and variance originating from the finite sample. It is cleat that the constant solely depends on the curvatures of the manifold and their covariant derivatives. Indeed, for a fixed $i$, the covariance matrix $\Xi_i$ built up in the local PCA step is
\begin{equation*}
\Xi_i = \frac{1}{n-1}\sum_{j\neq i}^{n} F_{i,j}\chi_{\|\iota(x_i)-\iota(x_j)\|\leq \sqrt{h_{\text{pca}}}},
\end{equation*}
where $F_{i,j}$ are i.i.d random matrix of size $p\times p$
\begin{equation*}
F_{i,j}=K\left(\frac{\|\iota(x_i)-\iota(x_j)\|_{\mathbb{R}^p}}{\sqrt{h_{\text{pca}}}}\right)(\iota(x_j)-\iota(x_i))(\iota(x_j)-\iota(x_i))^T,
\end{equation*}
so that its $(k,l)$-th entry
\begin{equation*}
F_{i,j}(k,l)=K\left(\frac{\|\iota(x_i)-\iota(x_j)\|_{\mathbb{R}^p}}{\sqrt{h_{\text{pca}}}}\right)\langle \iota(x_j)-\iota(x_i),v_k\rangle\langle \iota(x_j)-\iota(x_i),v_l\rangle,
\end{equation*}
where $v_l$ is the unit column vector with the $l$-th entry $1$ and $0<h_{\text{pca}}<h$.
Since $F_{i,j}$ are i.i.d. in $j$, we use $F_i$ to denote the random matrix whose expectation is
\begin{align*}
\EE F_{i}(k,l)=\int_{B_{\sqrt{h_{\text{pca}}}}(x_i)} K_{h_{\text{pca}}}(x_i,y)\langle\iota(y)-\iota(x_i),v_k\rangle\langle \iota(y)-\iota(x_i),v_l\rangle \sfp(y)\ud V(y),
\end{align*}
By Berstein's inequality, it has been shown in \cite{singer_wu:2012} that 
\begin{align*}
\mbox{Pr}\left\{\left|\Xi_i(k,l)-\EE F_i(k,l)\right|>\alpha\right\} \leq\exp\left\{-\frac{(n-1)\alpha^2}{O(h_{\text{pca}}^{d/2+2})+O(h_{\text{pca}})\alpha}\right\}.
\end{align*}
when $k,l=1,\ldots,d$;
\begin{align*}
\mbox{Pr}\left\{\left|\Xi_i(k,l)-\EE F_i(k,l)\right|>\alpha\right\}\leq \exp\left\{-\frac{(n-1)\alpha^2}{O(h_{\text{pca}}^{d/2+4})+O(h_{\text{pca}}^2)\alpha}\right\},
\end{align*}
when $k,l=d+1,\ldots,p$;
\begin{align*}\begin{split}
\mbox{Pr}\left\{\left|\Xi_i(k,l)-\EE F_i(k,l)\right|>\alpha\right\}\leq \exp\left\{-\frac{(n-1)\alpha^2}{O(h_{\text{pca}}^{d/2+3})+O(h_{\text{pca}}^{3/2})\alpha}\right\},
\end{split}\end{align*}
for the other cases. Then, denote $\Omega_{n,\alpha_1,\alpha_2,\alpha_3}$ to be the event space that for all $i=1,\ldots,n$, $|\Xi_i(k,l)-\EE F_i(k,l)|\leq\alpha_1$ for all $k,l=1,\ldots,d$, $|\Xi_i(k,l)-\EE F_i(k,l)|\leq\alpha_2$ for all $k,l=d+1,\ldots,p$, $|\Xi_i(k,l)-\EE F_i(k,l)|\leq\alpha_3$ for all $k=1,\ldots,d$, $l=d+1,\ldots,p$ and $l=1,\ldots,d$, $k=d+1,\ldots,p$. By a direct calculation we know that the probability of $\Omega_{n,\alpha_1,\alpha_2,\alpha_3}$ is lower bounded by
\begin{align}
1-n\Big(&d^2\exp\left\{-\frac{(n-1)\alpha_1^2}{O(h_{\text{pca}}^{d/2+2})+O(h_{\text{pca}})\alpha_1}\right\}+(p-d)^2\exp\left\{-\frac{(n-1)\alpha_2^2}{O(h_{\text{pca}}^{d/2+3})+O(h_{\text{pca}}^2)\alpha_2}\right\}\nonumber\\
&+p(p-d)\exp\left\{-\frac{(n-1)\alpha_3^2}{O(h_{\text{pca}}^{d/2+3})+O(h_{\text{pca}}^{3/2})\alpha_3}\right\}\Big).\nonumber
\end{align} 
Choose $\alpha_1=O\left(\frac{\log(n)h_{\text{pca}}^{d/4+1}}{n^{1/2}}\right)$, $\alpha_2=O\left(\frac{\log(n)h_{\text{pca}}^{d/4+2}}{n^{1/2}}\right)$ and $\alpha_3=O\left(\frac{\log(n)h_{\text{pca}}^{d/4+3/2}}{n^{1/2}}\right)$. Then, the probability of $\Omega_{n,\alpha_1,\alpha_2,\alpha_3}$ is higher than $1-O(1/n^2)$. As a result, when conditional on $\Omega_{n,\alpha_1,\alpha_2,\alpha_3}$ and a proper chosen $h_{\text{pca}}$, that is, $h_{\text{pca}}=O(n^{-2/(d+2)})$, we have  
\begin{align*}
O^T_{i}\iota_*X(x)=Q^T_i\iota_*X(x_i)+h_{\text{pca}}^{3/2}b_{1}\iota_*X(x_i), 
\end{align*}
where $b_{1}:\RR^p\to\RR^d$ is a bounded operator. Thus, conditional on $\Omega_{n,\alpha_1,\alpha_2,\alpha_3}$, we have \cite[(B.76)]{singer_wu:2012}:
\begin{align*}
O_i^TO_i=Q_i^TQ_i+h^{3/2}b_{2},  
\end{align*}
and hence \cite[Theorem B.2]{singer_wu:2012}
\begin{align}
\iota_*^TO_iO_{ij}B^T_iX(x_j)=\myP^{i}_{j}X(x_j)+h^{3/2}b_{3}X(x_j)\label{step0:Oij_vs_Pij},
\end{align} 
where $b_{2}:\RR^d\to \RR^d$ and $b_{3}:T_{x_j}\MM\to T_{x_i}\MM$ are bounded operators. Note that since  $h_{\text{pca}}\ll h$, the error introduced by local PCA step is absorbed in $h^{3/2}$.
We emphasize that both $O_i$ and $O_{ij}$ are random in nature, and they are dependent to some extent. When conditional on $\Omega_{n,\alpha_1,\alpha_2,\alpha_3}$, the randomness is bounded and we are able to proceed.

Define operators $Q^T_{\mathcal{X}}:T\MM_{\mathcal{X}}\to V_{\mathcal{X}}$ and $Q_{\mathcal{X}}:V_{\mathcal{X}}\to T\MM_{\mathcal{X}}$ by
\begin{align}
&Q_{\mathcal{X}}\vv:=[\iota^T_*Q^T_1\vv[1],\ldots \iota^T_*Q^T_n\vv[n]]\in T\MM_{\mathcal{X}},\nonumber\\ 
&Q_{\mathcal{X}}^T\vw:=[(Q_1\iota_*\vw[1])^T,\ldots, (Q_n\iota_*\vw[n])^T]^T\in V_{\mathcal{X}},\nonumber
\end{align}
where $\vw\in T\MM_{\mathcal{X}}$ and $\vv\in V_{\mathcal{X}}$. 

Note that $Q_{\mathcal{X}}^T\vD_{h,\alpha,n}^{-1}\vP_{h,\alpha,n}\vX$ is exactly the same as $B_{\mathcal{X}}^T\vD_{h,\alpha,n}^{-1}\vP_{h,\alpha,n}\vX$, so its behavior has been understood in Theorem \ref{thm:spectralconvergence_heatkernel}. Therefore, if we can control the difference between $Q_{\mathcal{X}}^T\vD_{h,\alpha,n}^{-1}\vP_{h,\alpha,n}\vX$ and $O_{\mathcal{X}}^T \vD_{h,\alpha,n}^{-1}\vP^{\textup{O}}_{h,\alpha,n}\vX$, where $\vP_{h,\alpha,n}$ is defined in (\ref{def:Sbundle}) when the frame bundle information can be fully accessed, by some modification of the proof of Theorem \ref{thm:spectralconvergence_heatkernel}, we can conclude the Theorem. By Lemma \ref{lemma:LargeDeviation}, we know that conditional on the event space $\Omega_{\sfp}$, which has probability higher than $1-O(1/n^2)$, we have $\widehat{p}_{h,n}>p_m/4 $. Thus, while conditional on $\Omega_{n,\alpha_1,\alpha_2,\alpha_3}\cap \Omega_{\sfp}$, by (\ref{step0:Oij_vs_Pij}), for all $i=1,\ldots,n$,
\begin{align}
&\left|Q_{\mathcal{X}}^T\vD_{h,\alpha,n}^{-1}\vP_{h,\alpha,n}\vX[i]-O_{\mathcal{X}}^T \vD_{h,\alpha,n}^{-1}\vP^{\textup{O}}_{h,\alpha,n}\vX[i]\right| = \left|\frac{1}{n}\sum_{j=1}^n\widehat{M}_{h,\alpha,n}(x_i,x_j)\big(\myP^{i}_{j}-\iota_*^TO_iO_{ij}B_j^T\big)X(x_j)\right|\nonumber\\
=&\,h^{3/2}\left|\frac{1}{n}\sum_{j=1}^n\widehat{M}_{h,\alpha,n}(x_i,x_j)b_{3}X(x_j)\right| =O(h^{3/2}), 
\end{align}
where the last inequality holds due to the fact that $\widehat{p}_{h,n}>p_m/4$ and $O(h^{3/2})$ depends on $\|X\|_{L^\infty}$. As a result, when conditional on $\Omega_{n,\alpha_1,\alpha_2,\alpha_3}\cap \Omega_{\sfp}$, the error introduced by the frame bundle estimation is of order high enough so that the object of interest, the connection Laplacian, is not influenced if we focus on a proper subspace of $L^2(\mathcal{E})$ depending on $h$. 
\newline\newline
\underline{\textbf{Step 2: Spectral convergence}}

Based on the analysis on Step 1, when conditional on $\Omega_{n,\alpha_1,\alpha_2,\alpha_3}\cap \Omega_{\sfp}$, we can directly study $Q_{\mathcal{X}}^T\vD_{h,\alpha,n}^{-1}\vP_{h,\alpha,n}\vX$ with the price of a negligible error. Clearly all steps in the proof of Theorem \ref{thm:spectralconvergence_heatkernel} hold for $Q_{\mathcal{X}}^T\vD_{h,1,n}^{-1}\vP_{h,1,n}$. As a result, conditional on $\Omega_{n,\alpha_1,\alpha_2,\alpha_3}\cap \Omega_{\sfp}$, by the perturbation theory, the eigenvectors of $O_{\mathcal{X}}^T\vD_{h,1,n}^{-1}\vP^{\textup{O}}_{h,1,n}$ is deviated from the eigenvectors of $Q_{\mathcal{X}}^T\vD_{h,1,n}^{-1}\vP_{h,1,n}$ by an error of order $h^{3/2}$, and we have finished the proof.

\end{proof}

\section*{Acknowledgment}
A. Singer was partially supported by Award Number R01GM090200 from the
NIGMS, by Award Number FA9550-12-1-0317 and FA9550-13-1-0076 from AFOSR,
and by Award Number LTR DTD 06-05-2012 from the Simons Foundation. H.-T. Wu acknowledges support by AFOSR grant FA9550-09-1-0551, NSF grant CCF-0939370 and FRG grant DSM-1160319. H.-T. Wu thanks Afonso Bandeira for reading the first version of this manuscript.

\appendix

\section{An Introduction to Principal Bundle}\label{section:appendix:principal_bundle}

In this appendix, we collect a few relevant and self-contained facts about the mathematical framework {\it principal bundle} which are used in the main text. We refer the readers to, for example \cite{bishop,Berline_Getzler_Vergne:2004}, for more general definitions which are not used in this paper.

We start from discussing the notion of {\it group action}, {\it orbit} and {\it orbit space}. Consider a set $Y$ and a group $G$ with the identity element $e$. The left group action of $G$ on $Y$ is a map from $G\times Y$ onto $Y$
\begin{align}
G\times Y\to Y,\quad (g,x)\mapsto g\circ x
\end{align}
so that $(gh)\circ x=g\circ (h\circ x)$ is satisfied for all $g,h\in G$ and $x\in Y$ and $e\circ x=x$ for all $x$. The right group action can be defined in the same way. Note that we can construct a right action by composing with the inverse group operation, so in some scenarios it is sufficient to discuss only left actions. There are several types of group action. We call an action {\it transitive} if for any $x,y\in Y$, there exists a $g\in G$ so that $g\circ x=y$. In other words, under the group action we can jump between any pair of two points on $Y$, or $Y=G\circ x$ for any $x\in Y$. We call an action {\it effective} is for any $g,h\in G$, there exists $x$ so that $g\circ x\neq h\circ x$. In other words, different group elements induce different permutations of $Y$. We call an action {\it free} if $g\circ x = x$ implies $g=e$ for all $g$. In other words, there is no fixed points under the $G$ action, and hence the name free. If $Y$ is a topological space, we call an action {\it totally discontinuous} if for every $x\in Y$, there is an open neighborhood $U$ such that $(g\circ U)\cap U=\emptyset$ for all $g\in G$, $g\neq e$.

The {\it orbit} of a point $x\in Y$ is the set
\[
Gx:=\{g\circ x;\, g\in G\}.
\]
The group action induces an equivalence relation. We say $x\sim y$ if and only if there exists $g\in G$ so that $g\circ x = y$ for all pairs of $x,y\in Y$. Clearly the set of orbits form a partition of $Y$, and we denote the set of all orbits as $Y/\sim$ or $Y/G$. We can thus define a projection map $\pi$ by
\[
Y\to Y/G,\quad x\mapsto Gx.
\]
We call $Y$ the {\it total space} or the {\it left $G$-space}, $G$ the {\it structure group}, $Y/G$ the {\it quotient space}, the {\it base space} or the {\it orbit space} of $Y$ under the action of $G$ and $\pi$ the {\it canonical projection}.

We define a principal bundle as a special $G$-space which satisfies more structure. Note that the definitions given here are not the most general ones but are enough for our purpose.
\begin{definition}[Fiber bundle] Let $\mathcal{F}$ and $\MM$ be two smooth manifolds and $\pi$ a smooth map from $\mathcal{F}$ to $\MM$. We say that $\mathcal{F}$ is a {\it fiber bundle} with fiber $F$ over $\MM$ if there is an open covering of $\MM$, denoted as $\{U_i\}$, and diffeomorphisms $\{\psi_i:\pi^{-1}(U_i)\to U_i\times F\}$ so that $\pi:\pi^{-1}(U_i)\to U_i$ is the composition of $\psi_i$ with projection onto $U_i$.
\end{definition}
By definition, $\pi^{-1}(x)$ is diffeomorphic to $F$ for all $x\in \MM$. We call $\mathcal{F}$ the total space of the fiber bundle, $\MM$ is the base space, $\pi$ the canonical projection, and $F$ the fiber of $\mathcal{F}$. With the above algebraic setup, in a nutshell, the principal bundle is a special fiber bundle accompanied by a group action.
\begin{definition}[Principal bundle] Let $\MM$ be a smooth manifold and $G$ a Lie group. A {\it principal bundle over $\MM$ with structure group $G$} is a fiber bundle $P(\MM,G)$ with fiber diffeomorphic to $G$, a smooth right action of $G$, denoted as $\circ$, on the fibers and a canonical projection $\pi:P\to\MM$ so that
\begin{enumerate}
\item $\pi$ is smooth and $\pi(g\circ p)=\pi(p)$ for all $p\in P$ and $g\in G$;
\item $G$ acts freely and transitively;
\item the diffeomorphism $\psi_i:\pi^{-1}(U_i)\to U_i\times G$ satisfies $\psi_i(p)=(\pi(p),\phi_i(p))\in U_i\times G$ such that $\phi_i:\pi^{-1}(U_i)\to G$ satisfying $\phi_i(pg)=\phi_i(p)g$ for all $p\in \pi^{-1}(U_i)$ and $g\in G$. 
\end{enumerate}
\end{definition}
Note that 
$\MM=P(\MM,G)/G$, where the equivalence relation is induced by $G$. From the view point of orbit space, $P(\MM,G)$ is the total space, $G$ is the structure group, and $\MM$ is the orbit space of $P(\MM,G)$ under the action of $G$. 
Intuitively, $P(\MM,G)$ is composed of a bunch of sets diffeomorphic to $G$, all of which are pulled together under some rules.\footnote{These rules are referred to as {\it transition functions}.} We give some examples here:

\begin{example}\label{Appendix:Example:trivialbundleConstruction}
Consider $P(\MM,G)=\MM\times G$ so that $G$ acts by $g\circ (x,h)=(x,hg)$ for all $(x,h)\in \MM\times G$ and $g\in G$. We call such principal bundle {\it trivial}. In particular, when $G=\{e\}$, the trivial group, $P(\MM,\{e\})$ is the principal bundle, which we choose to unify the graph Laplacian and diffusion map. 
\end{example}

\begin{example}
A particular important example of the principal bundle is the {\it frame bundle}, denoted as $GL(\MM)$, which is the principal $GL(d,\RR)$-bundle with the base manifold a $d$-dim smooth manifold $\MM$. We construct $GL(\MM)$ for the purpose of completeness. Denote $B_x$ to be the set of bases of the tangent space $T_x\MM $, that is, $B_x\cong GL(d,\RR)$ and $u_x\in B_x$ is a basis of $T_x\MM$. Let $GL(\MM)$ be the set consisting of all bases at all points of $\MM$, that is, $GL(\MM):=\{u_x;\, u_x\in B_x,\,x\in\MM\}$. Let $\pi:GL(\MM)\to \MM$ by $u_x\mapsto x$ for all $u_x\in B_x$ and $x\in \MM$. Define the right $GL(d,\RR)$ action on $GL(\MM)$ by $g\circ u_x=v_x$, where $g=[g_{ij}]_{i,j=1}^d\in GL(d,\RR)$, $u_x=(X_1,\ldots,X_d)\in B_x$ and $v_x=(Y_1,\ldots,Y_d)\in B_x$ with $Y_i=\sum_{j=1}^dg_{ij}X_j$. By a direct calculation, $GL(d,\RR)$ acts on $GL(\MM)$ from the right freely and transitively, and $\pi(g\circ u_x)=\pi(u_x)$. In a coordinate neighborhood $U$, $\pi^{-1}(U)$ is 1-1 corresponding with $U\times GL(d,\RR)$, which induces a differentiable structure on $GL(\MM)$. Thus $GL(\MM)$ is a principal $GL(d,\RR)$-bundle.
\end{example}

\begin{example}\label{Appendix:Example:Z2bundleConstruction}
Another important example is the {\it orientation principal bundle}, which we choose to unify the orientable diffusion map. The construction is essentially the same as that of the frame bundle. First, let $P(\MM,O(1))$ be the set of all orientations at all points of $\MM$ and let $\pi$ be the canonical projection from $P(\MM,O(1))$ to $\MM$, where $O(1)\cong \ZZ_2\cong \{1,-1\}$. In other words, $P(\MM,O(1)):=\{u_x;\, u_x\in \{1,-1\},\,x\in\MM\}$, where $\ZZ_2$ stands for the possible orientation of each point $x$. The $O(1)\cong \{1,-1\}$ group acts on $P(\MM,O(1))$ simply by $u\to ug$, where $u\in P(\MM,O(1))$ and $g\in \{1,-1\}$. The differentiable structure in $P(\MM,O(1))$ is introduced in the following way. Take $(x^1,\ldots,x^d)$ as a local coordinate system in a coordinate neighborhood $U$ in $\MM$. Since $\ZZ_2$ is a discrete group, we take $\pi^{-1}(U)$ as two disjoint sets $U\times\{1\}$ and $U\times\{-1\}$ and take  $(x^1,\ldots,x^d)$ as their coordinate systems. Clearly $P(\MM,O(1))$ is a principal fiber bundle and we call it the orientation principal bundle.
\end{example}
If we are given a left $G$-space $F$, we can form a {\it fiber bundle} from $P(\MM,G)$ so that its fiber is diffeomorphic to $F$ and its base manifold is $\MM$ in the following way. By denoting the left $G$ action on $F$ by $\cdot$, we have
\[
\mathcal{E}(P(\MM,G),\cdot,F):=P(\MM,G)\times_G F:=P(\MM,G)\times F/G,
\]
where the equivalence relation is defined as
\[
(g\circ p,g^{-1}\cdot f)\sim (p,f)
\]
for all $p\in P(\MM,G)$, $g\in G$ and $f\in F$. The canonical projection from $\mathcal{E}(P(\MM,G),\cdot,F))$ to $\MM$ is denoted as $\pi_{\mathcal{E}}$:
\[
\pi_{\mathcal{E}}:(p,f)\mapsto \pi(p),
\]
for all $p\in P(\MM,G)$ and $f\in F$.
We call $\mathcal{E}(P(\MM,G),\cdot,F)$ the {\it fiber bundle associated with $P(\MM,G)$ with standard fiber $F$} or the {\it associated fiber bundle} whose differentiable structure is induced from $\MM$. Given $p\in P(\MM,G)$, denote $pf$ to be the image of $(p,f)\in P(\MM,G)\times F$ onto $\mathcal{E}(P(\MM,G),\cdot,F)$. By definition, $p$ is a diffeomorphism from $F$ to $\pi_{\mathcal{E}}^{-1}(\pi(p))$ and 
\[
(g\circ p)f=p(g\cdot f).
\]
Note that the associated fiber bundle $\mathcal{E}(P(\MM,G),\cdot,F)$ is a special fiber bundle and its fiber is diffeomorphic to $F$. When there is no danger of confusion, we denote $\mathcal{E}:=\mathcal{E}(P(\MM,G),\cdot,F)$ to simply the notation.

\begin{example}
When $F=V$ is a vector space and the left $G$ action on $F$ is a linear representation, the associated fiber bundle is called the {\it vector bundle associated with the principal bundle $P(\MM,G)$ with fiber $V$}, or simply called the {\it vector bundle} if there is no danger of confusion. For example, take $F=\RR^{q}$, denote $\rho$ to be a representation of $G$ into $GL(q,\RR)$ and assume $G$ acts on $\RR^{q}$ via the representation $\rho$. A particular example of interest is the tangent bundle $T\MM:=\mathcal{E}(P(\MM,GL(d,\RR)),\rho,\RR^d)$ when $\MM$ is a $d$-dim smooth manifold and the representation $\rho$ is identity.  The practical meaning of the frame bundle and its associated tangent bundle is change of coordinate. That is, if we view a point $u_x\in GL(\MM)$ as the basis of the fiber $T_x\MM$, where $x=\pi(u_x)$, then the coordinate of a point on the tangent plane $T_x\MM$ changes, that is, $v_x\to g\cdot v_x$ where $g\in GL(d,\RR)$ and $v_x\in\RR^d$, according to the changes of the basis, that is, $g\to g\circ u_x$. Also notice that we can view a basis of $T_x\MM$ as an invertible linear map from $\RR^d$ to $T_x\MM$ by definition. Indeed, if take $e_i$, $i=1,\ldots,d$ to be the natural basis of $\RR^d$; that is, $e_i$ is the unit vector with $1$ in the $i$-th entry, a linear frame $u_x=(X_1,\ldots,X_d)$ at $x$ can be viewed as a linear mapping $u_x:\RR^d\to T_x\MM$ such that $u_xe_i=X_i$, $i=1,\ldots,d$. 
\end{example}

A (global) {\it section} of a fiber bundle $\mathcal{E}$ with fiber $F$ over $\MM$ is a map
\[
s:\MM\to \mathcal{E}
\]
so that $\pi(s(x))=x$ for all $x\in \MM$. We denote $\Gamma(\mathcal{E})$ to be the set of sections; $C^l(\mathcal{E})$ to be the space of all sections with the $l$-th regularity, where $l\geq 0$. An important property of the principal bundle is that a principal bundle is trivial if and only if $C^0(P(\MM,G))\neq\emptyset$. In other words, all sections on a non-trivial principal bundle are discontinuous. On the other hand, there always exists a continuous section on the associated vector bundle $\mathcal{E}$.

Let $V$ be a vector space. Denote $GL(V)$ to be the group of all invertible linear maps on $V$. If $V$ comes with an inner product, then define $O(V)$ to be the group of all orthogonal maps on $V$ with related to the inner product. From now on we focus on the vector bundle with fiber being a vector space $V$ and the action $\cdot$ being a representation $\rho:G\to GL(V)$, that is, $\mathcal{E}(P(\MM,G),\rho,V)$. 

To introduce the notion of {\it covariant derivative} on the vector bundle $\mathcal{E}$, we have to introduce the notion of {\it connection}. Note that the fiber bundle $\mathcal{E}$ is a manifold. Denote $T\mathcal{E}$ to be the tangent bundle of $\mathcal{E}$ and $T^*\mathcal{E}$ to be the cotangent bundle of $\mathcal{E}$. We call a tangent vector $X$ on $\mathcal{E}$ {\it vertical} if it is tangential to the fibers; that is, $X(\pi_{\mathcal{E}}^*f)=0$ for all $f\in C^\infty(\MM)$. Note that $\pi_{\mathcal{E}}^*f$ is a function defined on $\mathcal{E}$ which is constant on each fiber, so we call $X$ vertical when $X(\pi_{\mathcal{E}}^*f)=0$ for all $f\in C^\infty(\MM)$. Denote the bundle of vertical vectors as $V\mathcal{E}$, which is referred to as {\it the vertical bundle}, and is a subbundle of $T\mathcal{E}$. We call a vector field vertical if it is a section of the vertical bundle. 
Clearly the quotient of $T\mathcal{E}$ by its subbundle $V\mathcal{E}$ is isomorphic to $\pi^*T\MM$, and hence we have a short exact sequence of vector bundles:
\begin{align}
0\to V\mathcal{E}\to T\mathcal{E}\to \pi^*T\MM\to 0.\label{appendix:splitting:definition}
\end{align}
However, there is no canonical splitting of this short exact sequence. A chosen splitting is called a {\it connection}. In other words, a connection is a $G$-invariant distribution $H\subset T\mathcal{E}$ complementary to $V\mathcal{E}$.
\begin{definition}[Connection $1$-form]
Let $P(\MM,G)$ be a principal bundle. A connection $1$-form on $P(\MM,G)$ is an $\mathfrak{g}$-valued $1$-form $\omega\in \Gamma(T^*P(\MM,G)\otimes VP(\MM,G))$ so that $\omega(X)=X$ for any $X\in \Gamma(VP(\MM,G))$ and is invariant under the action of $G$. The kernel of $\omega$ is called the {\it horizontal bundle} and is denoted as $HP(\MM,G)$
\end{definition}
Note that $HP(\MM,G)$ is isomorphic to $\pi^*T\MM$. Clearly, a connection $1$-form determines a splitting of (\ref{appendix:splitting:definition}), or the connection on $P(\MM,G)$. In other words, as a linear subspace, the horizontal subspace $H_pP(\MM,G)\subset T_pP(\MM,G)$ is cut out by $\dim G$ linear equations defined on $T_pP(\MM,G)$.

We call a section $X_{P}$ of $HP(\MM,G)$ a {\it horizontal vector field}. Given $X\in \Gamma(T\MM)$, we say that $X_{P}$ is the {\it horizontal lift} with respect to the connection on $P(\MM,G)$ of $X$ if $X={\pi}_*X_{P}$. Given a smooth curve $\tau:=c(t)$, $t\in[0,1]$ on $\MM$ and a point $u(0)\in P(\MM,G)$, we call a curve $\tau^*=u(t)$ on $P(\MM,G)$ the {\it (horizontal) lift of $c(t)$} if the vector tangent to $u(t)$ is horizontal and $\pi(u(t))=c(t)$ for $t\in[0,1]$. The existence of $\tau^*$ is an important property of the connection theory. We call $u(t)$ the {\it parallel displacement} of $u(0)$ along the curve $\tau$ on $\MM$.

With the connection on $P(\MM,G)$, the connection on an associated vector bundle $\mathcal{E}$ with fiber $V$ is determined. As a matter of fact, we define the connection, or $H\mathcal{E}$, to be the image of $HP(\MM,G)$ under the natural projection $P(\MM,G)\times V\to \mathcal{E}(P(\MM,G),\rho,V)$. Similarly, we call a section $X_{\mathcal{E}}$ of $H\mathcal{E}$ a {\it horizontal vector field}. Given $X\in \Gamma(T\MM)$, we say that $X_{\mathcal{E}}$ is the {\it horizontal lift} with respect to the connection on $\mathcal{E}$ of $X$ if $X={\pi_{\mathcal{E}}}_*X_{\mathcal{E}}$. Given a smooth curve $c(t)$, $t\in[0,1]$ on $\MM$ and a point $v_0\in \mathcal{E}$, we call a curve $v_t$ on $\mathcal{E}$ the {\it (horizontal) lift of $c(t)$} if the vector tangent to $v_t$ is horizontal and $\pi_{\mathcal{E}}(v_t)=c(t)$ for $t\in[0,1]$. The existence of such horizontal life holds in the same way as that of the principal bundle. We call $v_t$ the {\it parallel displacement} of $v_0$ along the curve $\tau$ on $\MM$. Note that we have interest in this connection on the vector bundle since it leads to the {\it covariant derivative} we have interest.
\begin{definition}[Covariant Derivative] Take a vector bundle $\mathcal{E}$ associated with the principal bundle $P(\MM,G)$ with fiber $V$. The covariant derivative $\nabla^{\mathcal{E}}$ of a smooth section $X\in C^1(\mathcal{E})$ at $x\in\MM$ in the direction $\dot{c}_0$ is defined as
\begin{equation}\label{definition:parallel:frame}
\nabla^{\mathcal{E}}_{\dot{c}_0} X=\lim_{h\to 0}\frac{1}{h}[\myP^{c(0)}_{c(h)}X(c(h))-X(x)],
\end{equation}
where $c:[0,1]\to \MM$ is a curve on $\MM$ so that $c(0)=x$ and $\myP^{c(0)}_{c(h)}$ denotes the parallel displacement of $X$ from $c(h)$ to $c(0)$
\end{definition}
Note that in general although all fibers of $\mathcal{E}$ are isomorphic to $V$, the notion of comparison among them is not provided. An explicit example demonstrating the derived problem is given in the appendix of \cite{singer_wu:2012}. However, with the parallel displacement based on the notion of connection, we are able to compare among fibers, and hence define the derivative. With the fact that
\begin{align}\label{definition:parallel:frame2fact}
\myP^{c(0)}_{c(h)}X(c(h))= u(0) u(h)^{-1}X(c(h)),
\end{align}
where $u(h)$ is the horizontal lift of $c(h)$ to $P(\MM,G)$ so that $\pi(u(0))=x$, the covariant derivative (\ref{definition:parallel:frame}) can be represented in the following format:
\begin{equation}\label{definition:parallel:frame2}
\nabla^{\mathcal{E}}_{\dot{c}_0} X=\lim_{h\to 0}\frac{1}{h}[ u(0) u(h)^{-1}(X(c(h)))-X(c(0))],
\end{equation}
which is independent of the choice of $u(0)$. To show (\ref{definition:parallel:frame2fact}), set $v:= u(h)^{-1}(X(c(h)))\in V$. Clearly $ u(t) (v)$, $t\in [0,h]$, is a horizontal curve in $\mathcal{E}$ by definition. It implies that $ u(0) v= u(0) u(h)^{-1}(X(c(h)))$ is the parallel displacement of $X(c(h))$ along $c(t)$ from $c(h)$ to $c(0)$. Thus, although the covariant derivatives defined in (\ref{definition:parallel:frame}) and (\ref{definition:parallel:frame2}) are different in their appearances, they are actually equivalent.
We can understand this definition in the frame bundle $GL(\MM^d)$ and its associated tangent bundle. First, we find the coordinate of a point on the fiber $X(c(h))$, which is denoted as $u(h)^{-1}(X(c(h)))$, and then we put this coordinate $u(h)^{-1}(X(c(h)))$ to $x=c(0)$ and map it back to the fiber $T_x\MM$ by the basis $u(0)$. In this way we can compare two different ``abstract fibers'' by comparing their coordinates. A more abstract definition of the covariant derivative, yet equivalent to the aboves, is the following. A covariant derivative of $\mathcal{E}$ is a differential operator 
\begin{equation}\label{definition:parallel:frame3}
\nabla^{\mathcal{E}} :C^\infty(\mathcal{E})\to C^\infty(T^*\MM\otimes \mathcal{E})
\end{equation}
so that the Leibniz's rule is satisfied, that is, for $X\in C^\infty(\mathcal{E})$ and $f\in C^\infty(\MM)$, we have
\[
\nabla^{\mathcal{E}}(fX)=df\otimes X+f\nabla^{\mathcal{E}} X,
\]
where $d$ is the exterior derivative on $\MM$. Denote $\Lambda^kT^*M$ (resp $\Lambda T^*M$) to be the bundle of $k$-th exterior differentials (resp. the bundle of exterior differentials), where $k\geq1$. Given two vector bundles $\mathcal{E}_1$ and $\mathcal{E}_2$ on $M$ with the covariant derivatives $\nabla^{\mathcal{E}_1}$ and $\nabla^{\mathcal{E}_2}$, we construct a covariant derivative on $\mathcal{E}_1\otimes\mathcal{E}_2$ by
\begin{equation} 
\nabla^{\mathcal{E}_1\otimes \mathcal{E}_2}:=\nabla^{\mathcal{E}_1}\otimes 1+1\otimes \nabla^{\mathcal{E}_2}.
\end{equation}

A fiber metric $g^{\mathcal{E}}$ in a vector bundle $\mathcal{E}$ is a positive-definite inner-product in each fiber $V$ that varies smoothly on $\MM$. For any ${\mathcal{E}}$, if $\MM$ is paracompact, $g^{\mathcal{E}}$ always exists. A connection in $P(\MM,G)$, and also its associated vector bundle $\mathcal{E}$, is called metric if
\[
dg^{\mathcal{E}}(X_1,X_2)=g^{\mathcal{E}}(\nabla^{\mathcal{E}} X_1,X_2)+g^{\mathcal{E}}(X_1,\nabla^{\mathcal{E}} X_2),
\]
for all $X_1,X_2\in C^\infty(\mathcal{E})$.
We mainly focus on metric connection in this work. It is equivalent to say that the parallel displacement of $\mathcal{E}$ preserves the fiber metric. An important fact about the metric connection is that if a connection on $P(\MM,G)$ is metric given a fiber metric $g^{\mathcal{E}}$, than the covariant derivative on the associated vector bundle $\mathcal{E}$ can be equally defined from a sub-bundle $Q(\MM,H)$ of $P(\MM,G)$, which is defined as
\begin{align}\label{appendix_definition:Q}
Q(\MM,H):=\{p\in P(\MM,G):\,g^{\mathcal{E}}(p(u),p(v))=(u,v)\},
\end{align}
where $(\cdot,\cdot)$ is an inner product on $V$ and the structure group $H$ is a closed subgroup of $G$. In other words, $p\in Q(\MM,H)$ is a linear map from $V$ to $\pi_{\mathcal{E}}^{-1}(\pi(p))$ which preserves the inner product. A direct verification shows that the structure group of $Q(\MM,H)$ is
\begin{align} 
H:=\{g\in G:\,\rho(g)\in O(V)\}\subset G.
\end{align}
Since orthogonal property is needed in our analysis, when we work with a metric connection on a principal bundle $P(\MM,G)$ given a fiber metric $g^{\mathcal{E}}$ on $\mathcal{E}(P(\MM,G),\rho,V)$, we implicitly assume we work with its sub bundle $Q(\MM,H)$. With the covariant derivative, we now define the connection Laplacian. Assume $\MM$ is a $d$-dim smooth Riemmanian manifold with the metric $g$. With the metric $g$ we have an induced measure on $M$, denoted as $\ud V$.\footnote{To obtain the most geometrically invariant formulations, we may consider the density bundles as is considered in \cite[Chapter 2]{Berline_Getzler_Vergne:2004}. We choose not to do that in order to simplify the discussion.} 
Denote $L^p(\mathcal{E})$, $1\leq p< \infty$ to be the set of $L^p$ integrable sections, that is, $X\in L^p(\mathcal{E})$ if and only if \[ \int |g^{\mathcal{E}}_x(X(x),X(x))|^{p/2}\ud V(x)<\infty. \] 
Denote $\mathcal{E}^*$ to be the dual bundle of $\mathcal{E}$, which is paired with $\mathcal{E}$ by $g^{\mathcal{E}}$, that is, the pairing between  $\mathcal{E}$ and $\mathcal{E}^*$ is $\langle X,Y\rangle:= g^{\mathcal{E}}(X,Y)$, where $X\in C^\infty(\mathcal{E})$ and $Y\in C^\infty(\mathcal{E}^*)$. The connection on the dual bundle $\mathcal{E}^*$ is thus defined by
\[
d\langle X,Y\rangle=g^{\mathcal{E}}(\nabla^{\mathcal{E}}X, Y)+g^{\mathcal{E}}(X,\nabla^{\mathcal{E}^*} Y).
\]   
Recall that the Riemannian manifold $(M,g)$ possesses a canonical connection referred to as the Levi-Civita connection $\nabla$ \cite[p. 31]{Berline_Getzler_Vergne:2004}. Based on $\nabla$ we define the connection $\nabla^{T^*M\otimes \mathcal{E}}$ on the tensor product bundle $T^*M\otimes \mathcal{E}$.
\begin{definition}
Take the Riemannian manifold $(M,g)$, the vector bundle $\mathcal{E}:=\mathcal{E}(P(\MM,G),\rho,V)$ and its connection $\nabla^{\mathcal{E}}$. The connection Laplacian on $\mathcal{E}$ is defined as $\nabla^2:C^\infty(\mathcal{E})\to C^\infty(\mathcal{E})$ by
\[
\nabla^2:=-\tr (\nabla^{T^*M\otimes \mathcal{E}}\nabla^{\mathcal{E}}),  
\]
where $\tr:C^\infty(T^*M\otimes T^*M\otimes \mathcal{E})\to C^\infty(\mathcal{E})$ by contraction with the metric $g$.
\end{definition}
If we take the normal coordinate $\{\partial_i\}_{i=1}^d$ around $x\in M$, for $X\in C^\infty(\mathcal{E})$, we have
\[
\nabla^2X(x)=-\sum_{i=1}^d\nabla_{\partial_i}\nabla_{\partial_i}X(x).
\]
Given compactly supported smooth sections $X,Y\in C^\infty(\mathcal{E})$, a direct calculation leads to
\begin{align}
&\tr\big[\nabla(g^{\mathcal{E}}(\nabla^{\mathcal{E}} X ,Y ))\big]\nonumber\\
=&\tr\big[g^{\mathcal{E}}(\nabla^{T^* M\otimes \mathcal{E}} \nabla^{\mathcal{E}}X ,Y )+g^{\mathcal{E}}(\nabla^{\mathcal{E}}X ,\nabla^{\mathcal{E}} Y )\big]\nonumber\\
=&g^{\mathcal{E}}(\nabla^2X ,Y )+\tr g^{\mathcal{E}}(\nabla^{\mathcal{E}}X ,\nabla^{\mathcal{E}} Y )\nonumber.
\end{align}
By the divergence theorem, the left hand side disappears after integration over $\MM$, and we obtain $\nabla^2=-\nabla^{\mathcal{E}*}\nabla^{\mathcal{E}}$. Similarly we can show that $\nabla^2$ is self-adjoint. We refer the readers to \cite{gilkey:1974} for further properties of $\nabla^2$, for example the ellipticity, its heat kernel, and its application to the index theorem.

\section{[Proof of Theorem \ref{thm:pointwise_conv:approx_of_identity}]}

The proof is a generalization of \cite[Theorem B.4]{singer_wu:2012} to the general principal bundle structure. Note that in \cite[Theorem B.4]{singer_wu:2012} dependence of the error terms on a given section is not explicitly shown. In order to prove the spectral convergence, we have to make this dependence explicit.
Denote $\widetilde{B}_{t}(x):=\iota^{-1}(B^{\RR^p}_{t}(x)\cap \iota(\MM))$, where $t\geq 0$.

\begin{lemma}\label{Lemma:extra}
Assume Assumption \ref{Assumption:A} and Assumption \ref{Assumption:K} hold. 
Suppose $X\in L^\infty(\mathcal{E})$ and $0<\gamma<1/2$. Then, when $h$ is small enough, for all $x\in\MM$ the following holds:
\[
\left|\int_{\MM\backslash\widetilde{B}_{h^{\gamma}}(x)}h^{-d/2}K_h(x,y)\myP_y^xX(y) \ud V(y)\right|=O(h^2),
\]
where $O(h^2)$ depends on $\|X\|_{L^\infty}$.
\end{lemma}
\begin{proof}
We immediately have
\begin{align}
&\left|\int_{\MM\backslash\widetilde{B}_{h^{\gamma}}(x)}h^{-d/2}K_h(x,y)\myP_y^xX(y) \ud V(y)\right|\leq  \|X\|_{L^\infty}\left|\int_{\MM\backslash\widetilde{B}_{h^{\gamma}}(x)}h^{-d/2}K_h(x,y) \ud V(y)\right|\nonumber\\
=\,& \|X\|_{L^\infty}\Big|\int_{S^{d-1}}\int_{h^\gamma}^\infty h^{-d/2} \left[K\left(\frac{t}{\sqrt{h}}\right) +K'\left(\frac{t}{\sqrt{h}}\right)\frac{\|\Pi(\theta,\theta)\|t^3}{24\sqrt{h}}+O\left(\frac{t^6}{h}\right)\right] \nonumber\\
&\qquad\qquad\qquad\times \big[t^{d-1}+\Ric(\theta,\theta)t^{d+1}+O(t^{d+2})\big]\ud t\ud\theta\Big|\nonumber\\
=\,&\|X\|_{L^\infty} \left[\int_{h^{\gamma-1/2}}^\infty K(s)\left(s^{d-1}+h s^{d+1}\right)\ud s+h\int_{h^{\gamma-1/2}}^\infty  K'(s)s^{d+2}\ud s\right] +O(h^2)  =O(h^2), \nonumber
\end{align}
where $O(h^2)$ depends on $\|X\|_{L^\infty}$ and the last inequality holds by the fact that $K$ and $K'$ decay exponentially. Indeed, $h^{(d-1)(\gamma-1/2)}e^{-h^{\gamma-1/2}}<h^2$ when $h$ is small enough.  
\end{proof}
Next Lemma is needed when we handle the points near the boundary.  Note that when $x$ is near the boundary, the kernel is no longer symmetric, so we do not expect to obtain the second order term. Moreover, due to the possible nonlinearity of the manifold, in order to fully understand the first order term, we have to take care of the domain we have interest.
\begin{lemma}\label{Lemma:BoundarySymmetrization}
Assume Assumption \ref{Assumption:A}. Take $0<\gamma<1/2$ and $x\in \MM_{h^\gamma}$. Suppose $\min_{y\in\partial\MM}d(x,y)=\tilde{h}$. Fix a normal coordinate $\{\partial_1,\ldots,\partial_{d}\}$ on the geodesic ball $B_{h^\gamma}(x)$ around $x$ so that $x_0=\exp_x(\tilde{h}\partial_d(x))$. 
Divide $\exp_{x}^{-1}(B_{h^\gamma}(x))$ into slices $S_\eta$ defined by
\[
S_{\eta}=\{(\vu,\eta)\in\RR^d;\,\exp_x(\vu,\eta)\in B_{h^\gamma}(x),\,\|(u_1,\ldots,u_{d-1},\eta)\|<h^\gamma\},
\]
where $\eta\in[-h^\gamma,h^\gamma]$ and $\vu=(u_1,\ldots,u_{d-1})\in\RR^{d-1}$; that is, $\exp_{x}^{-1}(B_{h^\gamma}(x))=\cup_{\eta\in[-h^\gamma,h^\gamma]} S_\eta\subset \RR^d$. Define the symmetrization of $S_\eta$ by 
\[
\tilde{S}_\eta:=\cap^{d-1}_{i=1}(R_iS_\eta\cap S_\eta),
\]
where $R_i$ is the reflective operator satisfying $R_i(u_1,\ldots,u_i,\ldots,u_{d-1},\eta)=(u_1,\ldots,-u_i,\ldots,u_{d-1}\eta)$ and $i=1,\ldots,d-1$. Then, we have 
\[
\left|\int_{S_\eta}\int_{-h^\gamma}^{h^\gamma}\ud\eta\ud \vu-\int_{\tilde{S}_\eta}\int_{-h^\gamma}^{h^\gamma}\ud\eta\ud \vu\right|=O(h^{2\gamma}).
\]
\end{lemma}
\begin{proof}
Note that in general the slice $S_\eta$ is not symmetric with related to $(0,\ldots,0,\eta)$, while the symmetrization $\tilde{S}_\eta$ is. Recall the following relationship \cite[(B.23)]{singer_wu:2012} when $y=\exp_x(t\theta)$:
\[
\partial_l(\exp_x(t\theta))=\myP^y_x\partial_l(x)+\frac{t^2}{6}\myP^y_x(\mathcal{R}(\theta,\partial_l(x))\theta)+O(t^3),
\]
where $\theta\in T_x\MM$ is of unit norm and $t\ll 1$, which leads to
\begin{align}\label{bdry_proof:basis_diff}
\myP^{x_0}_{x}\partial_l(x)=\partial_l(x_0)+O(\tilde{h}^2),
\end{align}
for all $l=1,\ldots, d$.
Also note that up to error $O(\tilde{h}^3)$, we can express $\partial\MM\cap B_{h^\gamma}(x)$ by a homogeneous degree 2 polynomial with variables $\{\myP^{x_0}_{x}\partial_1(x),\ldots,\myP^{x_0}_x \partial_{d-1}(x)\}$. 
Thus the difference between $\tilde{S}_\eta$ and $S_\eta$ is $O(h^{2\gamma})$ since $\tilde{h}\leq h^\gamma$.  
\end{proof}
Next we elaborate the error term in the kernel approximation. 
\begin{lemma}\label{Kf}
Assume Assumption \ref{Assumption:A} and Assumption \ref{Assumption:K} hold. Take $0<\gamma<1/2$. Fix $x\notin \MM_{h^\gamma}$ and denote $C_x$ to be the cut locus of $x$. Take a vector-valued function $F:\MM\to\RR^q$, where $q\in\NN$ and $F \in C^4(\MM\backslash C_x)\cap L^\infty(\MM)$. Then, when $h$ is small enough, we have
\[
\int_{\MM} h^{-d/2}K_h(x,y)F(y)\ud V(y)=F(x)+h\frac{\mu^{(0)}_{1,2}}{d}\left(\frac{\Delta F(x)}{2}+ w(x)F(x)\right)+O(h^2),
\]
where $ w(x)=s(x) +\frac{\mu^{(1)}_{1,3}z(x)}{24|S^{d-1}|}$, $s(x)$ is the scalar curvature at $x$, and $ z(x)=\int_{S^{d-1}}\|\Pi(\theta,\theta)\|\ud \theta$ and the error term depends on $\|F^{(\ell)}\|_{L^\infty}$, where $\ell=0,1,\ldots,4$.

Fix $x\in \MM_{h^\gamma}$. Then, when $h$ is small enough, we have
\[
\int_{\MM} h^{-d/2}K_h(x,y)F(y)\ud V(y)=m_{h,0}F(x)+\sqrt{h}m_{h,1} \nabla_{\partial_d} F(x)+O(h^{2\gamma}),
\]
where $O(h^{2\gamma})$ depends on $\|F\|_{L^\infty}$, $\|F^{(1)}\|_{L^\infty}$ and $\|F^{(2)}\|_{L^\infty}$ and $m_{h,0}$ and $m_{h,1}$ are of order $O(1)$ and defined in (\ref{meps0}).
\end{lemma}
\begin{proof}
By Lemma \ref{Lemma:extra}, we can focus our analysis on $\widetilde{B}_{h^\gamma}(x)$ since $F$ is a section of the trivial bundle. Also, we can view $F$ as $q$ functions defined on $\MM$ with the same regularity. Then, the proof is exactly the same as that of \cite[Lemma 8]{coifman_lafon:2006} except the explicit dependence of the error term on $F$. Since the main point is the uniform bound of the third derivative of the embedding function $\iota$ and $F$ on $\MM$, we simply list the calculation steps:
\begin{align*}
&\int_{\widetilde{B}_{h^\gamma}(x)} K_h(x,y)F(y)\ud V(y)=\int_{\widetilde{B}_{h^\gamma}(x)}K\Big(\frac{\|x-y\|_{\RR^p}}{\sqrt{h}}\Big)F(y)\ud V(y)\\ 
=&\int_{S^{d-1}}\int_0^{h^\gamma} \left[K\left(\frac{t}{\sqrt{h}}\right)+K'\left(\frac{t}{\sqrt{h}}\right)\frac{\|\Pi(\theta,\theta)\|t^3}{24\sqrt{h}}+O\left(\frac{t^6}{h}\right)\right]\\
&\qquad\qquad\times \big[F(x)+\nabla_\theta F(x)t+\nabla^2_{\theta,\theta}F(x)\frac{t^2}{2}+\nabla^3_{\theta,\theta,\theta}F(x)\frac{t^3}{6}+O(t^3)\big]\nonumber\\
&\qquad\qquad\times \big[t^{d-1}+\Ric(\theta,\theta)t^{d+1}+O(t^{d+2})\big]\ud t\ud\theta.
\end{align*}
By a direct expansion, the regularity assumption and the compactness of $\MM$, we conclude the first part of the proof.
 
{\allowdisplaybreaks

Next, suppose $x\in\MM_{h^\gamma}$. By Taylor's expansion and Lemma \ref{Lemma:BoundarySymmetrization}, we obtain
\begin{align*}
&\int_{B_{h^\gamma}(x)} h^{-d/2}K_{h}(x,y) F(y)\ud V(y) \\
= &\, \int_{S_\eta}\int^{h^\gamma}_{-h^\gamma} h^{-d/2}\left[K\left( \frac{\sqrt{\|\vu\|^2+\eta^2}}{\sqrt{h}} \right)+K'\left( \frac{\sqrt{\|\vu\|^2+\eta^2}}{\sqrt{h}} \right) \frac{\|\Pi((\vu,\eta),(\vu,\eta))\|(\|\vu\|^2+\eta^2)^{3/2}}{24\sqrt{h}}\right.\\
&\left.\qquad\qquad +O\left(\frac{(\|\vu\|^2+\eta^2)^3}{h}\right)\right] \left(F(x)+\sum_{i=1}^{d-1}u_i\nabla_{\partial_i}F(x)+\eta\nabla_{\partial_d} F(x)+O(\tilde{h}^2)\right) \ud \eta\ud \vu\\
=&\, \int_{\tilde{S}_\eta}\int^{h^\gamma}_{-h^\gamma} h^{-d/2}K\left( \frac{\sqrt{\|\vu\|^2+\eta^2}}{\sqrt{h}} \right) \left(F(x)+\sum_{i=1}^{d-1}u_i\nabla_{\partial_i}F(x)+\eta\nabla_{\partial_d} F(x)+O(\tilde{h}^2)\right) \ud \eta\ud \vu +O(h^{2\gamma})\\
=&\, \int_{\tilde{S}_\eta}\int^{h^\gamma}_{-h^\gamma} h^{-d/2}K\left( \frac{\sqrt{\|\vu\|^2+\eta^2}}{\sqrt{h}} \right) \left(F(x)+\eta\nabla_{\partial_d} F(x)+O(\tilde{h}^2)\right) \ud \eta\ud \vu +O(h^{2\gamma})\\
=&\,m_{h,0}F(x)+\sqrt{h}m_{h,1} \nabla_{\partial_d} F(x)+O(h^{2\gamma}),
\end{align*}
where the third equality holds due to the symmetry of the kernel and
\begin{equation}\label{meps0}
\left\{
\begin{array}{l}
\displaystyle m_{h,0}:=\int_{\tilde{S}_\eta}\int^{h^\gamma}_{-h^\gamma} h^{-d/2}K\left( \frac{\sqrt{\|u\|^2+\eta^2}}{\sqrt{h}}\right)\ud \eta\ud x=O(1)\\
\displaystyle m_{h,1}:=\int_{\tilde{S}_\eta}\int^{h^\gamma}_{-h^\gamma} h^{-d/2-1/2}K\left( \frac{\sqrt{\|u\|^2+\eta^2}}{\sqrt{h}}\right)\eta\ud \eta\ud x=O(1).
\end{array}
\right.
\end{equation} 
}
\end{proof}

With the above Lemmas, we are able to finish the proof of Theorem \ref{thm:pointwise_conv:approx_of_identity}.

\begin{proof}[Proof of Theorem \ref{thm:pointwise_conv:approx_of_identity}]
Take $0<\gamma<1/2$. By Lemma \ref{Lemma:extra}, we can focus our analysis of the numerator and denominator of $T_{h,\alpha}X$ on $\widetilde{B}_{h^\gamma}(x)$, no matter $x$ is away from the boundary or close to the boundary. Suppose $x\notin \MM_{h^\gamma}$. 
By Lemma \ref{Kf}, we get
\begin{equation*} 
p_h(y)=\sfp(y)+h\frac{\mu^{(0)}_{1,2}}{d}\left(\frac{\Delta \sfp(y)}{2}+w(y)\sfp(y)\right)+O(h^{3/2}),
\end{equation*}
which leads to
\begin{equation}\label{approximatepalpha}
\frac{\sfp(y)}{p^{\alpha}_h(y)}=\sfp^{1-\alpha}(y)\left[1-\alpha h\frac{\mu^{(0)}_{1,2}}{d}\left(w(y)+\frac{\Delta \sfp(y)}{2\sfp(y)}\right)\right]+O(h^{3/2}).
\end{equation}
Plug (\ref{approximatepalpha}) into the numerator of $T_{h,\alpha}X(x)$:
\begin{align*}
\begin{split}
&\ \ \ \ \int_{\widetilde{B}_{h^\gamma}(x)} K_{h,\alpha}(x,y)\myP_{y}^{x} X(y)\sfp(y)\ud V(y) =p_h^{-\alpha}(x)\int_{\widetilde{B}_{h^\gamma}(x)} K_{h}(x,y)\myP_{y}^{x} X(y)p_h^{-\alpha}(y)\sfp(y)\ud V(y)\\
&=p_h^{-\alpha}(x)\int_{\widetilde{B}_{h^\gamma}(x)} K_{h}(x,y)\myP_{y}^{x} X(y)\sfp^{1-\alpha}(y)\left[1-\alpha h\frac{\mu^{(0)}_{1,2}}{d}\left(w(y)+\frac{\Delta \sfp(y)}{2\sfp(y)}\right)\right]\ud V(y)+O(h^{d/2+3/2})\\
&:=p_h^{-\alpha}(x)\left(A- h\frac{\alpha \mu^{(0)}_{1,2}}{d}B\right)+O(h^{d/2+3/2}).
\end{split}
\end{align*}
where
\[
\left\{
\begin{array}{l}
\displaystyle A:=\int_{\widetilde{B}_{h^\gamma}(x)} K_{h}(x,y)\myP_{y}^{x} X(y)\sfp^{1-\alpha}(y)\ud V(y),\\
\displaystyle B:=\int_{\widetilde{B}_{h^\gamma}(x)} K_{h}(x,y)\myP_{y}^{x} X(y)\sfp^{1-\alpha}(y)\left(w(y)+\frac{\Delta \sfp(y)}{2\sfp(y)}\right)\ud V(y).\end{array}\right.\]
When we evaluate $A$ and $B$, the odd monomials in the integral vanish because the kernel we use has the symmetry property. By Taylor's expansion, $A$ becomes
\begin{align*}
A=&\int_{S^{d-1}}\int_0^{h^\gamma} \left[K\left(\frac{t}{\sqrt{h}}\right)+K'\left(\frac{t}{\sqrt{h}}\right)\frac{\|\Pi(\theta,\theta)\|t^3}{24\sqrt{h}}+O\left(\frac{t^6}{h}\right)\right]\\
&\times\left[X(x)+\nabla_\theta X(x)t+\nabla^2_{\theta,\theta}X(x)\frac{t^2}{2}+\nabla^3_{\theta,\theta,\theta}X(x)\frac{t^3}{6}+O(t^4)\right] \\
&\times\left[\sfp^{1-\alpha}(x)+\nabla_\theta(\sfp^{1-\alpha})(x)t+\nabla^2_{\theta,\theta}(\sfp^{1-\alpha})(x)\frac{t^2}{2}+\nabla^3_{\theta,\theta,\theta}(\sfp^{1-\alpha})(x)\frac{t^3}{6}+O(t^3)\right] \\
&\times\left[t^{d-1}+\Ric(\theta,\theta)t^{d+1}+O(t^{d+2})\right]\ud t\ud\theta.
\end{align*}
Due to the fact that $K$ and $K'$ decay exponentially, by the same argument as that of Lemma \ref{Lemma:extra}, we can replace the integrals $\int_{S^{d-1}}\int_0^{h^\gamma}$ by $\int_{S^{d-1}}\int_0^\infty$ by paying the price of error of order $h^2$ which depends on $\|X^{(\ell)}\|_{L^\infty}$, where $\ell=0,1,\ldots,4$. Thus, after rearrangement we have 
\begin{align*}
A&=\sfp^{1-\alpha}(x)X(x)\int_{S^{d-1}}\int_0^\infty \Big\{K\left(\frac{t}{\sqrt{h}}\right)\left[1+\Ric(\theta,\theta)t^2\right]+ K'\left(\frac{t}{\sqrt{h}}\right)\frac{\|\Pi(\theta,\theta)\|t^3}{24\sqrt{h}}\Big\}t^{d-1}\ud t\ud\theta\\
&\quad+\sfp^{1-\alpha}(x)\int_{S^{d-1}}\int_0^\infty K\left(\frac{t}{\sqrt{h}}\right)\nabla^2_{\theta,\theta}X(x)\frac{t^{d+1}}{2}\ud t\ud\theta \\
&\quad+X(x)\int_{S^{d-1}}\int_0^\infty K\left(\frac{t}{\sqrt{h}}\right)\nabla^2_{\theta,\theta}(\sfp^{1-\alpha})(x)\frac{t^{d+1}}{2}\ud t\ud\theta\\
&\quad+\int_{S^{d-1}}\int_0^\infty K\left(\frac{t}{\sqrt{h}}\right)\nabla_{\theta}X(x)\nabla_\theta(\sfp^{1-\alpha})(x)t^{d+1}\ud t\ud\theta+O(h^{d/2+ 2}), 
\end{align*} 
where $O(h^{d/2+2})$ depends on  $\|X^{(\ell)}\|_{L^\infty}$,  $\ell=0,1,\ldots,4$.
Following the same argument as that in \cite{singer_wu:2012}, we have
\begin{align*}
\int_{S^{d-1}}\nabla^2_{\theta,\theta}X(x)\ud\theta=\frac{|S^{d-1}|}{d}\nabla^2 X(x)
\quad\mbox{and}\quad
\int_{S^{d-1}}\Ric(\theta,\theta)d\theta=\frac{|S^{d-1}|}{d}s(x).
\end{align*}
Therefore, 
\begin{align*}
A=\,&h^{d/2}\sfp^{1-\alpha}(x)\left\{\left(1+\frac{h \mu^{(0)}_{1,2}}{d}\frac{\Delta(\sfp^{1-\alpha})(x)}{2\sfp^{1-\alpha}(x)}+\frac{h \mu^{(0)}_{1,2}}{d}w(x)\right)X(x)+\frac{h \mu^{(0)}_{1,2}}{2d}\nabla^2X(x)\right\}\\
&\qquad+h^{d/2+1}\frac{\mu^{(0)}_{1,2}}{d} \nabla X(x)\cdot\nabla(\sfp^{1-\alpha})(x)+O(h^{d/2+ 2}),
\end{align*}
where $O(h^{d/2+ 2})$ depends on $\|X^{(\ell)}\|_{L^\infty}$, $\ell=0,1,\ldots,4$. 

To evaluate $B$, denote $Q(y):=\sfp^{1-\alpha}(y)\left(w(y)+\frac{\Delta \sfp(y)}{2\sfp(y)}\right)\in C^2(\MM)$ to simplify notation. We have
\begin{align*}
B=\,&\int_{B_{h^\gamma}(x)} K_{h}(x,y)\myP_{y}^{x} X(y)Q(y)\ud V(y)\\
=\,&\int_{S^{d-1}}\int_0^{h^\gamma} \left[K\left(\frac{t}{\sqrt{h}}\right)+O\left(\frac{t^3}{\sqrt{h}}\right)\right]\left[X(x)+\nabla_\theta X(x)t+O(t^2)\right]\left[Q(x)+\nabla_\theta Q(x)t+O(t^2)\right]\left[t^{d-1}+O(t^{d+1})\right]\ud t\ud\theta\\
=\,&h^{d/2}X(x)Q(x)+O(h^{d/2+1 }),
\end{align*} 
where $O(h^{d/2+1})$ depends on $\|X\|_{L^\infty},\|X^{(1)}\|_{L^\infty}$ and $\|X^{(2)}\|_{L^\infty}$.
In conclusion, the numerator of $T_{h,\alpha}X(x)$ becomes
\begin{align*}
&\quad h^{d/2} \frac{\sfp^{1-\alpha}(x)}{p_h^{\alpha}(x)}\left\{1+h \frac{\mu^{(0)}_{1,2}}{d}\left[\frac{\Delta(\sfp^{1-\alpha})(x)}{2\sfp^{1-\alpha}(x)}-\alpha\frac{\Delta \sfp(x)}{2\sfp(x)}\right]\right\}X(x)\\
&+h^{d/2+1}\frac{\mu^{(0)}_{1,2}\sfp^{1-\alpha}(x)}{d p^\alpha_h(x)} \left\{ \frac{\nabla^2X(x)}{2}
+\frac{\nabla X(x)\cdot\nabla(\sfp^{1-\alpha})(x)}{\sfp^{1-\alpha}(x)}\right\}+O(h^{d/2+ 2}),
\end{align*} 
where $O(h^{d/2+ 2})$ depends on  $\|X^{(\ell)}\|_{L^\infty}$, $\ell=0,1,\ldots,4$.
Similar calculation of the denominator of the $T_{h,\alpha} X(x)$ gives
\begin{align*}
h^{d/2}\frac{\sfp^{1-\alpha}(x)}{p_h^{\alpha}(x)}\left\{1+h\frac{ \mu^{(0)}_{1,2}}{d}\left(\frac{\Delta(\sfp^{1-\alpha})(x)}{2\sfp^{1-\alpha}(x)}-\alpha\frac{\Delta \sfp(x)}{2\sfp(x)}\right)\right\}+O(h^{d/2+ 2}).
\end{align*}
Putting all the above together, we have when $x\in \MM\backslash \MM_{h^\gamma}$,
\begin{align*}
T_{h,\alpha} X(x)=X(x)+h\frac{\mu^{(0)}_{1,2}}{2d}\left(\nabla^2X(x)+\frac{2\nabla X(x)\cdot\nabla(\sfp^{1-\alpha})(x)}{\sfp^{1-\alpha}(x)}\right)+O(h^{ 2}),
\end{align*}
where $O(h^{ 2})$ depends on  $\|X^{(\ell)}\|_{L^\infty}$, $\ell=0,1,\ldots,4$.

Next we consider the case when $x\in \MM_{h^\gamma}$. By Lemma \ref{Kf}, we get
\begin{equation*} 
p_h(y)=m_{h,0}\sfp(y)+\sqrt{h}m_{h,1}\partial_d\sfp(x)+O(h^{2\gamma}),
\end{equation*}
which leads to
\begin{equation*} 
\frac{\sfp(y)}{p^{\alpha}_h(y)}=\frac{\sfp^{1-\alpha}(y)}{m_{h,0}^\alpha}\left[1-\sqrt{h}\frac{\alpha m_{h,1}}{m_{h,0}} \frac{\partial_d\sfp(y)}{\sfp(y)}+O(h^{2\gamma})\right].
\end{equation*} 
By Taylor's expansion and Lemma \ref{Lemma:BoundarySymmetrization}, the numerator of $T_{h,\alpha}X$ becomes
\begin{align}
&\int_{B_{h^\gamma}(x) } K_{h,\alpha}(x,y) \myP_{y}^{x}X(y)\sfp(y)\ud V(y)\nonumber\\
=\,&\frac{p_{h}^{-\alpha}(x)}{m_{h,0}^\alpha}\int_{S_\eta}\int^{h^\gamma}_{-h^\gamma} K\left( \frac{\sqrt{\|\vu\|^2+\eta^2}}{\sqrt{h}} \right) \left(X(x)+\sum_{i=1}^{d-1}u_i\nabla_{\partial_i}X(x)+\eta\nabla_{\partial_d} X(x)+O(h^2)\right)\nonumber \\
&\qquad\qquad\qquad \times\left(\sfp^{1-\alpha}(x)+\sum_{i=1}^{d-1}u_i\nabla_{\partial_i}\sfp^{1-\alpha}(x)+\eta\nabla_{\partial_d} \sfp^{1-\alpha}(x)+O(h^2)\right)\nonumber\\
&\qquad\qquad\qquad \times \left[1-\sqrt{h}\frac{\alpha m_{h,1}}{m_{h,0}} \frac{\partial_d\sfp(y)}{\sfp(y)} \right]\ud \eta\ud \vu+O(h^{d/2+2\gamma})\nonumber\\
=\, &\frac{p_h^{-\alpha}(x)}{m_{h,0}^\alpha}\int_{\tilde{S}_\eta}\int^{h^\gamma}_{-h^\gamma} K\left( \frac{\sqrt{\|\vu\|^2+\eta^2}}{\sqrt{h}} \right)\left(X(x)+\sum_{i=1}^{d-1}u_i\nabla_{\partial_i}X(x)+\eta\nabla_{\partial_d} X(x)+O(h^2)\right)\nonumber \\
&\qquad\qquad\qquad \times\left(\sfp^{1-\alpha}(x)+\sum_{i=1}^{d-1}u_i\nabla_{\partial_i}\sfp^{1-\alpha}(x)+\eta\nabla_{\partial_d} \sfp^{1-\alpha}(x)+O(h^2)\right)\nonumber\\
&\qquad\qquad\qquad \times \left[1-\sqrt{h}\frac{\alpha m_{h,1}}{m_{h,0}} \frac{\partial_d\sfp(x)}{\sfp(x)} \right]\ud \eta\ud \vu+O(h^{d/2+2\gamma})\nonumber
\end{align}
where $O(h^{d/2+2\gamma})$ depends on  $\|X^{(\ell)}\|_{L^\infty}$, $\ell=0,1,2$, and the last equality holds due to Lemma \ref{Lemma:BoundarySymmetrization}. The symmetry of the kernel implies that for $i=1,\ldots,d-1$,
\begin{equation*} 
\int_{\tilde{S}_\eta}K\left( \frac{\sqrt{\|\vu\|^2+\eta^2}}{\sqrt{h}}\right)u^i\ud \vu=0,
\end{equation*}
and hence the numerator of $T_{h,\alpha}X(x)$ becomes
\begin{equation*} 
h^{d/2}\frac{m_{h,0}^{1-\alpha}}{p_h^{\alpha}(x)}\left[X(x)\sfp^{1-\alpha}(x)+\sqrt{h}\frac{m_{h,1}}{m_{h,0}} \left(X(x) \partial_d \sfp^{1-\alpha}(x)+\sfp^{1-\alpha}(x)\nabla_{\partial_d}X(x)+\frac{\alpha X(x)\partial_d \sfp(x)}{m_{h,0}\sfp(x)}\right)\right]+O(h^{d/2+2\gamma}) ,
\end{equation*}
where $O(h^{d/2+2\gamma})$ depends on $\|X\|_{L^\infty},\|X^{(1)}\|_{L^\infty}$ and $\|X^{(2)}\|_{L^\infty}$ and $m_{h,0}$ and $m_{h,1}$ are defined in (\ref{meps0}).
Similarly, the denominator of $T_{h,\alpha}X$ can be expanded as:
\begin{equation*} 
\int_{B_{h^\gamma}(x)}K_{h,\alpha}(x,y)\sfp(y)\ud V(y)=h^{d/2}\frac{m_{h,0}^{1-\alpha}}{p_h^{\alpha}(x)}\left[\sfp^{1-\alpha}(x)+\sqrt{h}\frac{m_{h,1}}{m_{h,0}} \left( \partial_d \sfp^{1-\alpha}(x)+ \frac{\alpha \partial_d \sfp(x)}{m_{h,0}\sfp(x)} \right)\right]+O(h^{d/2+2\gamma}).
\end{equation*}
Moreover, by (\ref{bdry_proof:basis_diff}), we have
\begin{align*} 
\myP^{x_0}_{x}\partial_l(x)=\partial_l(x_0)+O(h^{2\gamma}),
\end{align*}
for all $l=1,\ldots, d$. 
Thus, together with the expansion of the numerator and denominator of $T_{h,\alpha}X$, we have the following asymptotic expansion:
\begin{equation*} 
T_{h,\alpha} X(x)=X(x)+\sqrt{h}\frac{m_{h,1}}{m_{h,0}}\myP_{x_0}^x\nabla_{\partial_d}X(x_0)+O(h^{2\gamma}),
\end{equation*}
where $O(h^{2\gamma})$ depends on $\|X^{(\ell)}\|_{L^\infty}$, $\ell=0,1,2$, which finish the proof.
\end{proof}

\section{[Proof of Theorem \ref{thm:pointwise_conv_to_integral}]}
The proof is a generalization of that of \cite[Theorem B.3]{singer_wu:2012} to the principal bundle structure. Note that in \cite[Theorem B.3]{singer_wu:2012} only the uniform sampling p.d.f. case was discussed. 
The main ingredient in the stochastic fluctuation analysis of the GCL when $n$ is finite is the large deviation analysis. We emphasize that since the term we have interest, the connection Laplacian (or Laplace-Beltrami operator when we consider GL), is the $2$-th order term, that is, $h$, which is much smaller than the $0$-th order term, by applying the Berstein's inequality with the large deviation much smaller than $h$, we are able to achieve this rate. Here, for the sake of self-containment and clarifying some possible confusions in \cite{singer:2006}, we provide a detailed proof for this large deviation bound. 

\begin{lemma}\label{lemma:LargeDeviation}
Assume Assumption \ref{Assumption:A}, Assumption \ref{Assumption:B} and Assumption \ref{Assumption:K} hold. With probability higher than $1-O(1/n^2)$, the following kernel density estimation holds for all $i=1,\ldots,n$ 
\begin{align*}
\widehat{p}_{h,n}(x_i)=p_h(x_i) +O\left(\frac{\sqrt{\log(n)}}{n^{1/2}h^{d/4}}\right).
\end{align*}

Take $f\in C^4(\MM)$ and $1/4<\gamma<1/2$.  
For the points away from the boundary, suppose we focus on the situation that the stochastic fluctuation of $\frac{\sum_{j=1}^n K_{h}(x_i,x_j) (f(x_j)-f(x_i))}{\sum_{j=1}^n K_{h}(x_i,x_j)}$ is $o(h)$ for all $i$. Then, with probability higher than $1-O(1/n^2)$, the following holds for all $x_i\notin \MM_{h^\gamma}$:
\begin{align*}
\frac{1}{h}\frac{\sum_{j=1}^n K_{h}(x_i,x_j) (f(x_j)-f(x_i))}{\sum_{j=1}^n K_{h}(x_i,x_j)}&\,=\left(\frac{T_{h,0}f-f}{h}\right)(x_i)+O\left(\frac{\sqrt{\log(n)}}{n^{1/2}h^{d/4+1/2}}\right).
\end{align*} 
For the points near the boundary, suppose we focus on the situation that the stochastic fluctuation of $\frac{\sum_{j=1}^n K_{h}(x_i,x_j) (f(x_j)-f(x_i))}{\sum_{j=1}^n K_{h}(x_i,x_j)}$ is $o(\sqrt{h})$ for all $i$. Then, with probability higher than $1-O(1/n^2)$, the following holds for all $x_i\in \MM_{h^\gamma}$:
\begin{align*}
\frac{\sum_{j=1}^n K_{h}(x_i,x_j) (f(x_j)-f(x_i))}{\sum_{j=1}^n K_{h}(x_i,x_j)}&\,=(T_{h,0}f-f)(x_i)+O\left(\frac{\sqrt{\log(n)}}{n^{1/2}h^{d/4-1/4}}\right). 
\end{align*} 

\end{lemma}

\begin{proof}
Fix $x_i$. Note that $\frac{ \sum_{j=1}^n K_{h}(x_i,x_j) (f(x_j)-f(x_i))}{ \sum_{j=1}^n K_{h}(x_i,x_j)}$ is actually the un-normalized GL. Denote $F_j:=h^{-d/2}K_{h}(x_i,x_j) (f(x_j)-f(x_i))$ and $G_j:=h^{-d/2}K_{h}(x_i,x_j)$, then we have
$$
\frac{ \sum_{j=1}^n K_{h}(x_i,x_j) (f(x_j)-f(x_i))}{ \sum_{j=1}^n K_{h}(x_i,x_j)}=\frac{\frac{1}{n}\sum_{j=1}^n F_j}{\frac{1}{n}\sum_{j=1}^n G_j}.
$$
Clearly, $F_j$ and $G_j$, when $j\neq i$, can be viewed as randomly sampled i.i.d. from two random variables $F$ and $G$ respectively.
Note that the un-normalized GL is a ratio of two dependent random variables, therefore the variance cannot be simply computed. We want to show that
$$
\frac{\frac{1}{n}\sum_{j=1}^n F_j}{\frac{1}{n}\sum_{j=1}^n G_j} \approx \frac{\mathbb{E}[F]}{\mathbb{E}[G]}
$$
and to control the size of the fluctuation as a function of $n$ and $h$. Note that we have
\[
\frac{1}{n}\sum_{j=1}^n F_j=\frac{n-1}{n}\left[\frac{1}{n-1}\sum_{j=1,j\neq i}^n F_j\right] 
\]
since $K_{h}(x_i,x_j) (f(x_j)-f(x_i))=0$. Also, since $\frac{n-1}{n}\to 1$ surely as $n\to\infty$, we can simply focus on analyzing $\frac{1}{n-1}\sum_{j=1,j\neq i}^n F_j$. A similar argument holds for $\frac{1}{n}\sum_{j=1}^n G_j$ -- clearly, $K_{h}(x_i,x_i)=K(0)>0$, so this term will contribute to the error term of order $\frac{1}{n}$. Thus, we have
\[
\frac{\frac{1}{n}\sum_{j=1}^n F_j}{\frac{1}{n}\sum_{j=1}^n G_j}=\frac{\frac{1}{n-1}\sum_{j=1,j\neq i}^n F_j}{\frac{1}{n-1}\sum_{j=1,j\neq i}^n G_j}+O\left(\frac{1}{n}\right).
\]
As we will see shortly, the $O(1/n)$ term will be dominated and can thus be ignored. 
 
First of all, we consider $x_i\notin\MM_{h^\gamma}$, By Theorem \ref{thm:pointwise_conv:approx_of_identity}, we have
\begin{align*}
\mathbb{E}[F]  =& \int_{\MM} h^{-d/2}K_{h}(x_i,y) (f(y)-f(x_i))\sfp(y)\,\ud V(y)= h \frac{\mu^{(0)}_{1,2}}{2}\Delta ((f(y)-f(x_i))\sfp(y))|_{y=x_i} + O\left(h^2\right) \\
\mathbb{E}[G] =& \int_{\MM} h^{-d/2}K_{h}(x_i-y) \sfp(y)\,\ud V(y)= \sfp(x_i) + O(h)
\end{align*}
and
\begin{align*}
\mathbb{E}[F^2]  =\,& \int_{\MM} h^{-d}K_{h}^2(x_i-y) (f(x_i)-f(y))^2 \sfp(y)\,\ud V(y) \\
=\,& \frac{1}{h^{d/2-1}} \frac{\mu^{(0)}_{2,2}}{2}\Delta ((f(x_i)-f(y))^2p(y))|_{y=x_i} + O\left(\frac{1}{h^{d/2-2}}\right) \\
\mathbb{E}[G^2]  =\,& \int_{\MM} h^{-d}K_{h}^2(x_i-y) \sfp(y)\,\ud V(y)= \frac{1}{h^{d/2}}\mu^{(0)}_{2,0}\sfp(x_i) + O\left(\frac{1}{h^{d/2-1}}\right).
\end{align*}
Thus, we conclude that \footnote{Note that since 
\begin{align*}
\mathbb{E}[F G]  =& \int_{\MM} K_{h}^2(x_i-y) (f(x_i)-f(y))\sfp(y)\,\ud V(y)= \frac{1}{h^{d/2-1}} \frac{\mu^{(0)}_{2,2}}{2}\Delta ((f(x_i)-f(y))p(y))|_{y=x_i} + O\left(\frac{1}{h^{d/2-2}}\right)\\ 
\text{Cov}(F, G )  =& \mathbb{E}[F G ] - \mathbb{E}[F]\mathbb{E}[G] =\frac{1}{h^{d/2-1}} \frac{\mu^{(0)}_{2,2}}{2}\Delta ((f(y)-f(x_i))\sfp(y))|_{y=x_i} + O\left(h,\frac{1}{h^{d/2-2}}\right) ,
\end{align*}
the correlation between $F$ and $G$ is
\begin{align*}
\rho (F ,G )=\frac{\text{Cov}(F ,G )}{\sqrt{\text{Var}(F)}\sqrt{\text{Var}(G)}}= O\left(\frac{\sqrt{h^{d/2 + d/2 - 1}}}{h^{d/2-1}} \right)=O(\sqrt{h}).
\end{align*}}
\begin{align*}
\text{Var}(F)  =\,& \frac{1}{h^{d/2-1}} \frac{\mu^{(0)}_{2,2}}{2}\Delta ((f(y)-f(x_i))^2\sfp(y))|_{y=x_i} + O\left(\frac{1}{h^{d/2-2}}\right) \\
\text{Var}(G)  =\,& \frac{1}{h^{d/2}}\mu^{(0)}_{2,0}\sfp(x_i) + O\left(1,\frac{1}{h^{d/2-1}}\right). \\
\end{align*} 

With the above bounds, we bounds on large deviation with high probability. First, note that the random variables $F$ are uniformly bounded by 
$$
c=O(h^{-d/2})
$$ 
and its variance is 
$$
\sigma^2 = O(h^{-(d/2-1)}).
$$ 
We see that 
$$
\sigma^2 \ll c,
$$ 
so Bernstein's inequality could in principle provide a large deviation bound that is tighter than that provided by Hoeffding's inequality. Recall Bernstein's inequality
$$
\Pr \left\{\frac{1}{n-1}\sum_{j=1,j\neq i}^n (F_j - \mathbb{E}[F]) > \alpha \right\} \leq e^{-\frac{n\alpha^2}{2\sigma^2 + \frac{2}{3}c\alpha}},
$$
where $\alpha>0$.
Since our goal is to estimate a quantity of size $O(h)$ (the prefactor of the Laplacian), we need to take $\alpha \ll h$. Let us take $\alpha=o(h)$. The exponent in Bernstein's inequality takes the form
$$
\frac{n\alpha^2}{2\sigma^2 + \frac{2}{3}c\alpha} = \frac{n \alpha^2}{O(h^{-(d/2-1)}) + o(h^{-d/2}h)} = O(n\alpha^2h^{d/2-1}),
$$
and by a simple union bound, we have
$$
\Pr \left\{\frac{1}{n-1}\sum_{j=1,j\neq i}^n (F_j - \mathbb{E}[F]) > \alpha;\,i=1,\ldots,n \right\} \leq ne^{-\frac{n\alpha^2}{2\sigma^2 + \frac{2}{3}c\alpha}}.
$$
Suppose $n,h$ further satisfy
$$
n\alpha^2h^{d/2-1} = O(\log (n));
$$
that is,
\begin{align}\label{proof:alphaChoice}
\alpha = O\left(\frac{\sqrt{\log (n)}}{n^{1/2}h^{d/4-1/2}}\right)\ll h.
\end{align}
It implies that for all $i=1,\ldots,n$, the deviation happens with probability less than $O(1/n^2)$ that goes to 0 as $n\to \infty$. It is clear that (\ref{proof:alphaChoice}) can be easily satisfied if we choose $h\gg \left(\frac{\log (n)}{n}\right)^{\frac{1}{\frac{d}{2}+1}}$.  
A simple bound by Hoeffding's inequality holds for the denominator when $\alpha=o(1)$ -- with probability higher than $1-O(1/n^2)$, for all $i=1,\ldots,n$, we have
\[
\left|\frac{1}{n-1}\sum_{j=1,j\neq i}^n (G_j - \mathbb{E}[G])\right|=O\left(\frac{\sqrt{\log (n)}}{n^{1/2}h^{d/4}}\right).
\]
Altogether, with probability higher than $1-O(1/n^2)$, for all $i=1,\ldots,n$, we have
\begin{align*}
\frac{\frac{1}{n-1}\sum_{j=1,j\neq i}^n F_j}{\frac{1}{n-1}\sum_{j=1,j\neq i}^n G_j}&=\,\frac{\mathbb{E}[F] + O\left(\frac{\sqrt{\log(n)}}{n^{1/2}h^{d/4-1/2}}\right) }{\mathbb{E}[G] + O\left(\frac{\sqrt{\log(n)}}{n^{1/2}h^{d/4}}\right)}=h\left[\frac{h^{-1}\mathbb{E}[F] + O\left(\frac{\sqrt{\log(n)}}{n^{1/2}h^{d/4+1/2}}\right) }{\mathbb{E}[G] + O\left(\frac{\sqrt{\log(n)}}{n^{1/2}h^{d/4}}\right)}\right]\\
&=\,h\left[\frac{h^{-1}\mathbb{E}[F]}{\mathbb{E}[G]} + O\left(\frac{\sqrt{\log(n)}}{n^{1/2}h^{d/4+1/2}}\right)\right],
\end{align*}
where the last equality holds since $h^{-1}\mathbb{E}[F]$ is of order $O(1)$.
Therefore, since $O\left(\frac{\sqrt{\log(n)}}{n^{1/2}h^{d/4+1/2}}\right)$ dominates $O\left(\frac{1}{nh}\right)$, we obtain the conclusion when $x_i\notin\MM_{h^\gamma}$.

For $x_i\in\MM_{h^\gamma}$, a similar argument holds. Indeed, the random variables $F$ are uniformly bounded by 
$$
c=O(h^{-d/2})
$$ 
and by (\ref{bdryrslt3}), its variance is 
$$
\sigma^2 = O(h^{-(d/2-1/2)}). 
$$ 
Indeed, when $x_i$ is near the boundary, the first order term cannot be canceled, so the variance is $O(h^{-(d/2-1/2)})$ instead of $O(h^{-(d/2-1)})$.  Thus, under the assumption that $\alpha=o(h^{1/2})$, the Berstein's inequality leads to the large deviation bound of the numerator. As a result, we obtain
\begin{align*}
\frac{\frac{1}{n-1}\sum_{j=1,j\neq i}^n F_j}{\frac{1}{n-1}\sum_{j=1,j\neq i}^n G_j}&=\,\frac{\mathbb{E}[F] + O\left(\frac{\sqrt{\log(n)}}{n^{1/2}h^{d/4-1/4}}\right) }{\mathbb{E}[G] + O\left(\frac{\sqrt{\log(n)}}{n^{1/2}h^{d/4}}\right)}=h^{1/2}\left[\frac{h^{-1/2}\mathbb{E}[F] + O\left(\frac{\sqrt{\log(n)}}{n^{1/2}h^{d/4+1/4}}\right) }{\mathbb{E}[G] + O\left(\frac{\sqrt{\log(n)}}{n^{1/2}h^{d/4}}\right)}\right]\\
&=\,h^{1/2}\left[\frac{h^{-1/2}\mathbb{E}[F]}{\mathbb{E}[G]} + O\left(\frac{\sqrt{\log(n)}}{n^{1/2}h^{d/4+1/4}}\right)\right],
\end{align*}
where the last term holds since $\frac{h^{-1/2}\mathbb{E}[F]}{\mathbb{E}[G]}$ is of order $O(1)$ when $x_i$ is near the boundary.
\end{proof}

\begin{proof}[Proof of Theorem \ref{thm:pointwise_conv_to_integral}]
Fix $i$ and $0<\alpha\leq 1$. By definition we have
\begin{align}
\,&(\vD_{h,\alpha,n}^{-1}\vP_{h,\alpha,n}\vX-\vX)[i] 
=\frac{\frac{1}{n}\sum_{j=1}^n \frac{K_h(x_i,x_j)}{\widehat{p}^\alpha_{h,n}(x_j)}(g_{ij}\vX[j]-\vX[i])}{\frac{1}{n}\sum_{l=1}^n  \frac{K_{h}(x_i,x_l)}{\widehat{p}^\alpha_{h,n}(x_l)}}\nonumber\\ 
=\,&\frac{\frac{1}{n}\sum_{j=1}^n \frac{K_h(x_i,x_j)}{p^\alpha_h(x_j)}(g_{ij}\vX[j]-\vX[i])}{\frac{1}{n}\sum_{l=1}^n  \frac{K_{h}(x_i,x_l)}{p^\alpha_h(x_l)}}\label{proof:finite:nonuniform1}\\
&\,+\frac{\frac{1}{n}\sum_{j=1}^n \Big(\frac{1}{\widehat{p}^\alpha_{h,n}(x_j)}-\frac{1}{p^\alpha_h(x_j)}\Big)K_h(x_i,x_j)(g_{ij}\vX[j]-\vX[i])}{\frac{1}{n}\sum_{l=1}^n  \frac{K_{h}(x_i,x_l)}{p^\alpha_h(x_l)}}\label{proof:finite:nonuniform2}\\
&\,+\frac{1}{n}\sum_{j=1}^n \frac{K_h(x_i,x_j)}{\widehat{p}^{\alpha}_{h,n}(x_j)}(g_{ij}\vX[j]-\vX[i])\left(\frac{1}{\frac{1}{n}\sum_{l=1}^n  \frac{K_{h}(x_i,x_l)}{\widehat{p}^\alpha_{h,n}(x_l)}}-\frac{1}{\frac{1}{n}\sum_{l=1}^n  \frac{K_{h}(x_i,x_l)}{p^\alpha_h(x_l)}}\right)\label{proof:finite:nonuniform3}.
\end{align}
Note that when $j=i$, $ \frac{K_h(x_i,x_j)}{p^\alpha_h(x_j)}(g_{ij}\vX[j]-\vX[i])= 0$, thus we have the following re-formulation
\[
\frac{1}{n}\sum_{j=1}^n \frac{K_h(x_i,x_j)}{p^\alpha_h(x_j)}(g_{ij}\vX[j]-\vX[i])=\frac{n-1}{n}\left(\frac{1}{n-1}\sum_{j=1,j\neq i}^n \frac{K_h(x_i,x_j)}{p^\alpha_h(x_j)}(g_{ij}\vX[j]-\vX[i])\right).
\]
Note that $\frac{n-1}{n}$ will converge to $1$. 
Thus, we can focus on analyzing the stochastic fluctuation of $\frac{1}{n-1}\sum_{j=1,j\neq i}^n \frac{K_h(x_i,x_j)}{p^\alpha_h(x_j)}(g_{ij}\vX[j]-\vX[i])$. The same comment applies to the other terms.
Clearly, $F_j:=\frac{K_h(x_i,x_j)}{p^\alpha_h(x_j)}(g_{ij}\vX[j]-\vX[i])$, $j\neq i$, are i.i.d. sampled from a $q$-dim random vector $F$, and $G_j:=\frac{K_h(x_i,x_j)}{p_h^\alpha(x_j)}$ are i.i.d. sampled from a random variable $G$. Thus, the analysis of the random vector $\frac{\frac{1}{n-1}\sum_{j=1,j\neq i}^n F_j}{\frac{1}{n-1}\sum_{j=1,j\neq i}^n G_j}$ can be viewed as an analysis of $q$ random variables. To apply Lemma \ref{lemma:LargeDeviation}, we have to clarify the regularity issue of $g_{ij}\vX[j]-\vX[i]$. Note that by definition, $g_{ij}\vX[j] := u_i^{-1}\myP_{j}^iX(x_j)$, thus we can view $g_{ij}\vX[j]$ as the value of the vector-valued function $u_i^{-1}\myP_{y}^{x_i}X(y)$ at $y=x_j$. Clearly, $u_i^{-1}\myP_{\cdot}^{x_i}X(\cdot)\in C^4(\MM\backslash C_{x_i})\cap L^\infty(\MM)$. Thus, Lemma \ref{lemma:LargeDeviation} can be applied. Indeed, we view $\frac{ g_{ij}\vX[j]-\vX[i] }{p^\alpha_h(x_j)}$ (resp. $\frac{1}{p^\alpha_h(x_j)}$) in the numerator as a discretization of the function $\frac{u_i^{-1}\myP_{y}^{x_i}X(y)-X(x_i)}{p^\alpha_h(y)}$ (resp. $\frac{1}{p^\alpha_h(y)}$).

As a result, for all $x_i\notin \MM_{h^\gamma}$, with probability higher than $1-O(1/n^2)$
\begin{align}\label{proof:T_haFinite}
 \frac{\frac{1}{n-1}\sum_{j=1,j\neq i}^n \frac{K_h(x_i,x_j)}{p^\alpha_h(x_j)}(g_{ij}\vX[j]-\vX[i])}{\frac{1}{n-1}\sum_{l=1,l\neq i}^n  \frac{K_{h}(x_i,x_l)}{p^\alpha_h(x_l)}}=u_i^{-1}\left( T_{h,\alpha}X-X \right)(x_i)+O\left(\frac{\sqrt{\log(n)}}{n^{1/2}h^{d/4-1/2}}\right).
\end{align}
Denote $\Omega_1$ to be the event space that (\ref{proof:T_haFinite}) holds.
Similarly, by Lemma \ref{lemma:LargeDeviation}, with probability higher than $1-O(1/n^2)$, 
$$
\left|\widehat{p}_{h,n}(x_j)- p_h(x_j)\right| =O\left(\frac{\sqrt{\log(n)}}{n^{1/2}h^{d/4}}\right) 
$$ 
for all $j=1,...,n$. Thus, by Assumption \ref{Assumption:K}, when $h$ is small enough, we have for all $x_i\in\mathcal{X}$
\begin{equation}\label{proof:pointwise:bound1}
\begin{split}
p_m/2\leq |p_{h}(x_i)|\leq p_M,\quad p_m/4\leq |\widehat{p}_{h,n}(x_i)|\leq 2p_M.
\end{split}
\end{equation}
Denote $\Omega_2$ to be the event space that (\ref{proof:pointwise:bound1}) holds. Thus, under  $\Omega_2$, by Taylor's expansion and (\ref{proof:pointwise:bound1}) we have
\[
|\widehat{p}_{h,n}^{-\alpha}(x_i)-p_h(x_i)^{-\alpha}|\leq  \frac{\alpha}{(p_m/4)^{1+\alpha}}|\widehat{p}_{h,n}(x_i)-p_h(x_i)|=O\left(\frac{\sqrt{\log(n)}}{n^{1/2}h^{d/4}}\right).
\]
With these bounds, under $\Omega_2$, (\ref{proof:finite:nonuniform2}) is simply bounded by $O\left(\frac{\sqrt{\log(n)}}{n^{1/2}h^{d/4}}\right)$, where the constant depends on $\|X\|_{L^\infty}$.
Similarly, under $\Omega_2$ we have the following bound for (\ref{proof:finite:nonuniform3}):
\begin{align*}
 \left|\frac{1}{\frac{1}{n}\sum_{l=1}^n  \frac{K_{h}(x_i,x_l)}{\widehat{p}^\alpha_{h,n}(x_l)}}-\frac{1}{\frac{1}{n}\sum_{l=1}^n  \frac{K_{h}(x_i,x_l)}{p^\alpha_h(x_l)}}\right|=\left|\frac{\frac{1}{n}\sum_{l=1}^n K_{h}(x_i,x_l)\left(\frac{1}{p^\alpha_h(x_l)}-\frac{1}{\widehat{p}^\alpha_{h,n}(x_l)}\right)}{\frac{1}{n}\sum_{l=1}^n  \frac{K_{h}(x_i,x_l)}{\widehat{p}^\alpha_{h,n}(x_l)}\frac{1}{n}\sum_{l=1}^n  \frac{K_{h}(x_i,x_l)}{p^\alpha_h(x_l)}} \right| =O\left(\frac{\sqrt{\log(n)}}{n^{1/2}h^{d/4}}\right).
\end{align*}
Hence, (\ref{proof:finite:nonuniform3}) is bounded by $O\left(\frac{\sqrt{\log(n)}}{n^{1/2}h^{d/4}}\right)$, where the constant depends on $\|X\|_{L^\infty}$ under $\Omega_2$. 

Putting the above together, under $\Omega_1\cap \Omega_2$, we have
$$
 (\vD_{h,\alpha,n}^{-1}\vP_{h,\alpha,n}\vX-\vX)[i]= u_i^{-1}(T_{h,\alpha}X-X)(x_i)+O\left(\frac{\sqrt{\log(n)}}{n^{1/2}h^{d/4}}\right),
$$
for all $i=1,\ldots,n$. Note that the measure of $\Omega_1\cap\Omega_2$ is greater than $1-O(1/n^2)$, so we finish the proof when $0<\alpha\leq 1$. 

When $\alpha=0$, clearly (\ref{proof:finite:nonuniform2}) and (\ref{proof:finite:nonuniform3}) disappear, and we only have (\ref{proof:finite:nonuniform1}). Since the convergence behavior of  (\ref{proof:finite:nonuniform1}) has been shown in (\ref{proof:T_haFinite}), we thus finish the proof when $\alpha=0$.  
A similar argument holds for $x_i\in\MM_{h^\gamma}$, and we skip the details.
\end{proof}

\section{Symmetric Isometric Embedding}\label{appendix_proof_symmetric_embedding}

Suppose we have a closed, connected and smooth $d$-dim Riemannian manifold $(\MM,g)$ with free isometric $\ZZ_2:=\{1,z\}$ action on it. Note that $\MM$ can be viewed as a principal bundle $P(\MM/\ZZ_2,\ZZ_2)$ with the group $\ZZ_2$ as the fiber. Without loss of generality, we assume the diameter of $\MM$ is less than $1$. The eigenfunctions $\{\phi_{j}\}_{j\geq0}$ of the Laplace-Beltrami operator $\Delta_\MM$ are known to form an orthonormal basis of $L^2(\MM)$, where $\Delta_\MM\phi_{j}=-\lambda_j\phi_{j}$ with $\lambda_j\geq0$. Denote $E_\lambda$ the eigenspace of $\Delta_\MM$ with eigenvalue $\lambda$. Since $\ZZ_2$ commutes with $\Delta_\MM$, $E_\lambda$ is a representation of $\ZZ_2$, where the action of $z$ on $\phi_j$ is defined by $z\circ \phi_j(x):=\phi_j(z\circ x)$.

We claim that all the eigenfunctions of $\Delta_\MM$ are either even or odd. Indeed, since $\ZZ_2$ is an abelian group and all the irreducible representations of $\ZZ_2$ are real, we know $z\circ \phi_{i}=\pm\phi_{i}$ for all $i\geq 0$. We can thus distinguish two different types of eigenfunctions:
\[
\phi^e_{i}(z\circ x)=\phi^e_{i}(x)
\quad\mbox{and}\quad
\phi^o_{i}(z\circ x)=-\phi^o_{i}(x),
\]
where the superscript $e$ (resp. o) means even (resp. odd) eigenfunctions.

It is well known that the heat kernel $k(x,y,t)$ of $\Delta_\MM$ is a smooth function over $x$ and $y$ and analytic over $t>0$, and can be written as
\[
k(x,y,t)=\sum_i e^{-\lambda_it}\phi_i(x)\phi_i(y),
\]
we know for all $t>0$ and $x\in\MM$, $\sum_j e^{-\lambda_jt}\phi_j(x)\phi_j(x)<\infty$. Thus we can define a family of maps by exceptionally taking odd eigenfunctions into consideration:
\[
\begin{array}{lccll}
\Psi^o_t:&\MM & \rightarrow &\ell^2 &\mbox{  for }t>0,\\
&x & \mapsto & \{e^{-\lambda_j t/2}\phi^o_j(x)\}_{j\geq 1} &
\end{array}
\]

\begin{lemma}\label{psio}
For $t>0$, the map $\Psi^o_t$ is an embedding of $\MM$ into $\ell^2$.
\end{lemma}
\begin{proof}
If $x_n\rightarrow x$, we have by definition
\begin{align*}
\|\Psi^o_t(x_n)-\Psi^o_t(x)\|^2_{\ell^2}&=\sum_j\left|e^{-\lambda_jt/2}\phi^o_j(x_n)-e^{-\lambda_jt/2}\phi^o_j(x)\right|^2\\
&\leq\sum_j\left|e^{-\lambda_jt/2}\phi^o_j(x_n)-e^{-\lambda_jt/2}\phi^o_j(x)\right|^2+\sum_j\left|e^{-\lambda_jt/2}\phi^e_j(x_n)-e^{-\lambda_jt/2}\phi^e_j(x)\right|^2\\
&=k(x_n,x_n,t)+k(x,x,t)-2k(x_n,x,t),
\end{align*}
which goes to $0$ as $n\rightarrow \infty$ due to the smoothness of the heat kernel. Thus $\Psi^o_t$ is continuous.

Since the eigenfunctions $\{\phi_j\}_{j\geq0}$ of the Laplace-Beltrami operator form an orthonormal basis of $L^2(\MM)$, it follows that they separate points. We now show that odd eigenfunctions are enough to separate points. Given $x\neq y$ two distinct points on $\MM$, we can find a small enough neighborhood $N_x$ of $x$ that separates it from $y$. Take a characteristic odd function $f$ such that $f(x)=1$ on $N_x$, $f(z\circ x)=-1$ on $z\circ N_x$ and $0$ otherwise. Clearly we know $f(x)\neq f(y)$. Since $f$ is odd, it can be expanded by the odd eigenfunctions:
\[
f = \sum_j a_j \phi^o_j.
\]
Hence $f(x)\neq f(y)$ implies that there exists $\alpha$ such that $\phi^o_\alpha(x) \neq\phi^o_\alpha(y)$.

Suppose we have $\Psi^o_t(x)=\Psi^o_t(y)$, then $\phi^o_i(x)=\phi^o_i(y)$ for all $i$. By the above argument we conclude that $x=y$, that is, $\Psi^o_t$ is an 1-1 map. 
To show that $\Psi^o_t$ is an immersion, consider a neighborhood $N_x$ so that $N_x\cap z\circ N_x=\emptyset$. Suppose there exists $x\in\MM$ so that $\ud\Psi^o_t(X)=0$ for $X\in T_x\MM$, which implies $\ud \phi^o_i(X)=0$ for all $i$. Thus by the same argument as above we know $\ud f(X)=0$ for all $f\in C^\infty_c(N_x)$, which implies $X=0$.

In conclusion, $\Psi^o_t$ is continuous and 1-1 immersion from $\MM$, which is compact, onto $\Psi^o_t(\MM)$, so it is an embedding.
\end{proof}

Note that $\Psi^o_t(\MM)$ is symmetric with respect to $0$, that is, $\Psi^o_t(z\circ x)=-\Psi^o_t(x)$. However, it is not an isometric embedding and the embedded space is of infinite dimension. Now we construct an isometric symmetric embedding of $\MM$ to a finite dimensional space by extending the Nash embedding theorem \cite{Nash1954,Nash1956}.
We start from considering an open covering of $\MM$ in the following way. Since $\Psi_t^o$, $t>0$, is an embedding of $\MM$ into $\ell^2$, for each given $p\in\MM$, there exists $d$ odd eigenfunctions $\{\phi^o_{i^p_j}\}_{j=1}^{d}$ so that
\begin{equation}\label{definition:open:covering:map}
\begin{split}
&v_p:x\in\MM\mapsto (\phi^o_{i^p_1}(x),...,\phi^o_{i^p_d}(x))\in\RR^d\\
&v_{z\circ p}:z\circ x\in\MM\mapsto -(\phi^o_{i^p_1}(x),...,\phi^o_{i^p_d}(x))\in\RR^d
\end{split}
\end{equation}
are of full rank at $p$ and $z\circ p$. We choose a small enough neighborhood $N_p$ of $p$ so that $N_p\cap z\circ N_p=\emptyset$ and $v_p$ and $v_{z\circ p}$ are embedding of $N_p$ and $z\circ N_p$. It is clear that $\{N_p,\,\,z\circ N_p\}_{p\in\MM}$ is an open covering of $\MM$.

With the open covering $\{N_p,\,\,z\circ N_p\}_{p\in\MM}$, it is a well known fact \cite{sternberg:1999} that there exists an atlas of $\MM$
\begin{align}
\mathcal{A}=\{(V_j,h_j),\,\,(z\circ V_j,h^z_j)\}_{j=1}^L\label{definition:atlas:A}
\end{align}
where $V_j\subset M$, $z\circ V_j\subset M$, $h_j:\MM\to \RR^d$, $h_j^z:\MM\to \RR^d$, so that the following holds and the symmetry is taken into account:
\begin{enumerate}
\item[(a)] $\mathcal{A}$ is a locally finite refinement of $\{N_p,\,\,z\circ N_p\}_{p\in\MM}$, that is, for every $V_i$  (resp. $z\circ V_i$), there exists a $p_i\in\MM$ (resp. $z\circ p_i\in\MM$) so that $V_i\subset N_{p_i}$ (resp. $z\circ V_i\subset z\circ N_{p_i}$),
\item[(b)] $h_j(V_j)=B_2$, $h^z_j(z\circ V_j)=B_2$, and $h_j(x)=h^z_j(z\circ x)$ for all $x\in V_j$,
\item[(c)] for the $p_i$ chosen in (a), there exists $\phi^o_{i_{p_i}}$ so that
$\phi^o_{i_{p_i}}(x)\neq \phi^o_{i_{p_i}}(z\circ x)$ for all $x\in V_i$.
\item[(d)] $\MM=\cup_j \left(h_j^{-1}(B_1)\cup (h^z_j)^{-1}(B_1)\right)$. Denote $O_j=h_j^{-1}(B_1)$.
\end{enumerate}
where $B_r=\{x\in\RR^d:\|x\|<1\}$. We fix the point $p_i\in\MM$ when we determine $\mathcal{A}$, that is, if $V_i\in\mathcal{A}$, we have a unique $p_i\in\MM$ so that $V_i\subset N_{p_i}$.  Note that (c) holds since $\Psi_t^o$, $t>0$, is an embedding of $\MM$ into $\ell^2$ and the eigenfunctions of $\Delta_\MM$ are smooth.
We will fix a partition of unity $\{\eta_i\in C_c^\infty(V_i),\,\,\eta^z_i\in C_c^\infty(z\circ V_i)\}$ subordinate to $\{V_j,\,\,z\circ V_j\}_{j=1}^L$. Due to symmetry, we have $\eta_i(x)=\eta^z_i(z\circ x)$ for all $x\in V_i$. To ease notation, we define
\begin{equation}\label{definition:partition_of_unity_over_A}
\begin{split}
\psi_i(x)=\left\{
\begin{array}{ll}
\eta_i(x) & \mbox{ when }x\in V_i\\
\eta^z_i(x) & \mbox{ when }x\in z\circ V_i
\end{array}\right.
\end{split}
\end{equation}
so that $\{\psi_i\}_{i=1}^L$ is a partition of unit subordinate to $\{V_i\cup z\circ V_i\}_{i=1}^L$.

\begin{lemma}\label{Z}
There exists a symmetric embedding $\tilde{u}:\MM^d\hookrightarrow \RR^{N}$ for some $N\in\NN$.
\end{lemma}

\begin{proof}
Fix $V_i$ and hence $p_i\in\MM$.
Define
\[
u_i:x\in\MM\mapsto (\phi^o_{i_{p_i}}(x),v_{p_i}(x))\in \RR^{d+1},
\]
where $v_{p_i}$ is defined in (\ref{definition:open:covering:map}). Note that $u_i$ is of full rank at $p_i$.  Due to symmetry, the assumption (c) and the fact that $V_i\cap z\circ V_i=\emptyset$, we can find a rotation $R_i\in SO(d+1)$ and modify the definition of $u_i$: 
\[
u_i:
x \mapsto  R_i(\phi^o_{i_{p_i}}(x),v_{p_i}(x)) ,
\]
which is an embedding of $V_i\cup z\circ V_i$ onto $\RR^{d+1}$ so that $u_i(V_i\cup z\circ V_i)$ does not meet all the axes of $\RR^{d+1}$. Note that since $v_{z\circ p}(z\circ x)=- v_\sfp(x)$ and $\phi^o_{i_{p_i}}(z\circ x)=-\phi^o_{i_{p_i}}(x)$, we have $u_i(z\circ x)=-u_i(x)$.
Define
\[
\bar{u}:x\mapsto (u_{1}(x),...,u_{L}(x)).
\]
Since locally $\ud \bar{u}$ is of full rank and
\[
\bar{u}(z\circ x)=(u_{1}(z\circ x),...,u_{L}(z\circ x))=-(u_{1}(x),...,u_{L}(x))=-\bar{u}(x),
\]
$\bar{u}$ is clearly an symmetric immersion from $\MM$ to $\RR^{L(d+1)}$. Denote
\[
\epsilon=\min_{i=1,...,L}\min_{x\in V_i\cup z\circ V_i}\min_{k=1,...,d+1}\langle u_{i}(x),e_k\rangle,
\]
where $\{e_k\}_{k=1,...,d+1}$ is the canonical basis of $\RR^{d+1}$. By the construction of $u_i$, $\epsilon>0$.

By the construction of the covering $\{O_i\cup g\circ O_i\}_{i=1}^L$, we know $L\geq 2$. We claim that by properly perturbing $\bar{u}$ we can generate a symmetric 1-1 immersion from $\MM$ to $\RR^{L(d+1)}$.

Suppose $\bar{u}$ is 1-1 in $W\subset \MM$, which is invariant under $\ZZ_2$ action by the construction of $\bar{u}$. Consider a symmetric closed subset $K\subset W$. Let $O^1_i=W\cap (O_i\cup g\circ O_i)$ and $O^2_i=(\MM\backslash K)\cap (O_i\cup g\circ O_i)$. Clearly $\{O^1_i,O^2_i\}_{i=1}^L$ is an covering of $\MM$. Consider a partition of unity $\mathcal{P}=\{\theta_\alpha\}$ subordinate to this covering so that $\theta_\alpha(z\circ x)=\theta_\alpha(x)$ for all $\alpha$. Index $\mathcal{P}$ by integer numbers so that for all $i>0$, we have $\mbox{supp}\theta_i\subset O^2_i$.

We will inductively define a sequence $\tilde{u}_k$ of immersions by properly choosing constants $b_i\in\RR^{L(d+1)}$:
\[
\tilde{u}_k=\bar{u}+\sum^k_{i=1}b_is_i\theta_i,
\]
where $s_i\in C_c^\infty(\MM)$ so that $\mbox{supp}(s_i)\subset N_i\cup z\circ N_i$ and
\[
s_i(x)=\left\{
\begin{array}{rl}
1&\mbox{ when }x\in V_i\\
-1&\mbox{ when }x\in z\circ V_i
\end{array}
\right..
\]
Note that $u_k$ by definition will be symmetric. Suppose $u_{k}$ is properly defined to become an immersion and $\|\tilde{u}_{j}-\tilde{u}_{j-1}\|_{C^\infty}<2^{-j-2}\epsilon$ for all $j\leq k$.

Denote
\[
D_{k+1}=\{(x,y)\in\MM\times\MM: s_{k+1}(x)\theta_{k+1}(x)\neq s_{k+1}(y)\theta_{k+1}(y)\},
\]
which is of dimension $2d$. Define $G_{k+1}:D_{k+1}\rightarrow \RR^{L(d+1)}$ as
\[
G_{k+1}(x,y)=\frac{\tilde{u}_k(x)-\tilde{u}_k(y)}{s_{k+1}(x)\theta_{k+1}(x)-s_{k+1}(y)\theta_{k+1}(y)}.
\]
Since $G_{k+1}$ is differentiable and $L\geq 2$, by Sard's Theorem $G_{k+1}(D_{k+1})$ is of measure zero. By choosing $b_{k+1}\notin G_{k+1}(D_{k+1})$ small enough, $\tilde{u}_{k+1}$ can be made an immersion and $\|\tilde{u}_{k+1}-\tilde{u}_{k}\|<2^{-k-3}\epsilon$. In this case $\tilde{u}_{k+1}(y_1)=\tilde{u}_{k+1}(y_2)$ implies
\[
b_{k+1}(s_{k+1}(x)\theta_{k+1}(x)-s_{k+1}(y)\theta_{k+1}(y))=\tilde{u}_k(x)-\tilde{u}_k(y).
\]
Since $b_{k+1}\notin G_{k+1}(D_{k+1})$, this can happen only if $s_{k+1}(x)\theta_{k+1}(x)=s_{k+1}\theta_{k+1}(y)$ and $\tilde{u}_k(x)=\tilde{u}_k(y)$.

Define
\[
\tilde{u}=\tilde{u}_L.
\]
By definition $\tilde{u}$ is a symmetric immersion and differs from $\bar{u}$ by $\epsilon/2$ in $C^\infty$.

Now we claim that $\tilde{u}$ is 1-1. Suppose $\tilde{u}(x)=\tilde{u}(y)$. Note that by the construction of $b_j$ this implies $s_{L}(x)\theta_L(x)=s_{L}(y)\theta_L(y)$ and $u_{L-1}(x)=u_{L-1}(y)$. Inductively we have $\bar{u}(x)=\bar{u}(y)$ and $s_{j}(x)\theta_j(x)=s_{j}(y)\theta_j(y)$ for all $j>0$. Suppose $x\in W$ but $y\notin W$, then $s_{j}(y)\theta_j(y)=s_{j}(x)\theta_j(x)=0$ for all $j>0$, which is impossible. Suppose both $x$ and $y$ are outside $W$, then there are two cases to discuss. First, if $x$ and $y$ are both inside $V_i$ for some $i$, then $s_{j}(x)\theta_j(x)=s_{j}(y)\theta_j(y)$ for all $j>0$ and $\bar{u}(x)=\bar{u}(y)$ imply $x=y$ since $\bar{u}$ embeds $V_i$. Second, if $x\in V_i\backslash V_j$ and $y\in V_j\backslash V_i$ where $i\neq j$, then $s_{j}(x)\theta_j(x)=s_{j}(y)\theta_j(y)$ for all $j>0$ is impossible. In conclusion, $\tilde{u}$ is 1-1.

Since $\MM$ is compact and $\tilde{u}$ is continuous, we conclude that $\tilde{u}$ is a symmetric embedding of $\MM$ into $\RR^{L(d+1)}$.

\end{proof}

The above Lemma shows that we can always find a symmetric embedding of $\MM$ into $\RR^{L(d+1)}$ for some $L>0$. The next Lemma helps us to show that we can further find a symmetric embedding of $\MM$ into $\RR^{p}$ for some $p>0$ which is isometric. We define $s_p:=\frac{p(p+1)}{2}$ in the following discussion.

\begin{lemma}\label{PHI}
There exists a symmetric smooth map $\Phi$ from $\RR^p$ to $\RR^{s_p+p}$ so that $\partial_i\Phi(x)$ and $\partial_{ij}\Phi(x)$, $i,j=1,...p$, are linearly independent as vectors in $\RR^{s_p+p}$ for all $x\neq 0$.
\end{lemma}

\begin{proof}
Denote $x=(x_1,...x_{p})\in\RR^p$. We define the map $\Phi$ from $\RR^{p}$ to $\RR^{s_{p}+p}$ by
\[
\Phi:x\mapsto \left(x_1\,\,,\,\,...,x_p\,\,,\,\,x_1\frac{e^{x_1}+e^{-x_1}}{2}\,\,,\,\,x_1\frac{e^{x_2}+e^{-x_2}}{2}\,\,,...\,\,,\,\,x_{p}\frac{e^{x_p}+e^{-x_p}}{2}\right).
\]
where $i,j=1,...,p$ and $i\neq j$. It is clear that $\Phi$ is a symmetric smooth map, that is, $\Phi(-x)=-\Phi(x)$. Note that
\[
\partial_{ij}\left(x_k\frac{e^{x_\ell}+e^{-x_\ell}}{2}\right)=\delta_{jk}\frac{e^{x_i}-e^{-x_i}}{2}+\delta_{ik}\frac{e^{x_j}-e^{-x_j}}{2}+x_k\delta_{j\ell}\frac{e^{x_i}+e^{-x_i}}{2}
\]
Thus when $x\neq 0$, for all $i=1,...,p$, $\partial_i\Phi(x)$ and $\partial_{ij}\Phi(x)$, $i,j=1,...p$, are linearly independent as vectors in $\RR^{s_p+p}$.
\end{proof}

Combining Lemma \ref{Z} and \ref{PHI}, we know there exists a symmetric embedding $u:\MM^d\hookrightarrow \RR^{s_{L(d+1)}+L(d+1)}$ so that $\partial_iu(x)$ and $\partial_{ij}u(x)$, $i,j=1,...,d$, are linearly independent as vectors in $\RR^{s_{L(d+1)}+L(d+1)}$ for all $x\in\MM$. Indeed, we define
\[
u=\Phi\circ \tilde{u}.
\]
Clearly $u$ is a symmetric embedding of $\MM$ into $\RR^{s_{L(d+1)}+L(d+1)}$. Note that $\tilde{u}(x)\neq 0$ otherwise $\tilde{u}$ is not an embedding. Moreover, by the construction of $\tilde{u}$, we know $u_i(V_i\cup z\circ V_i)$ is away from the axes of $\RR^{L(d+1)}$ by $\epsilon/2$, so the result.

Next we control the metric on $u(\MM)$ induced by the embedding. By properly scaling $u$, we have $g-\ud u^2>0$. We will assume properly scaled $u$ in the following.

\begin{lemma}\label{Y}
Given the atlas $\mathcal{A}$ defined in (\ref{definition:atlas:A}), there exists $\xi_i\in C^\infty(V_i,\,\,\RR^{s_d+d})$ and $\xi^z_i\in C^\infty(z\circ V_i,\,\,\RR^{s_d+d})$ so that $\xi^z_i-\xi_i>cI_{s_d+d}$ for some $c>0$ and
\[
g-\ud u^2=\sum_{j=1}^m \eta_j^2\ud \xi_j^2+\sum_{j=1}^m (\eta^z_j)^2(\ud \xi^z_j)^2.
\]
\end{lemma}

\begin{proof}
Fix $V_i$. By applying the local isometric embedding theorem \cite{sternberg:1999} we have smooth maps $x_i:h_i(V_i)\hookrightarrow \RR^{s_d+d}$ and $x^z_i:h^z_i(z\circ V_i)\hookrightarrow \RR^{s_d+d}$ so that
\[
(h_i^{-1})^*g=\ud x_i^2
\quad\mbox{and}\quad
((h^z_i)^{-1})^*g=(\ud x^z_i)^2,
\]
where $\ud x_i^2$ (resp. $(\ud x^z_i)^2$) means the induced metric on $h_i(V_i)$ (resp. $h^z_i(z\circ V_i$)) from $\RR^{s_d+d}$. Note that the above relationship is invariant under affine transformation of $x_i$ and $x_i^z$.
By assumption (b) of $\mathcal{A}$ we have $h_i(x)=h^z_i(z\circ x)$ for all $x\in V_i$, so we modify $x_i$ and $x_i^z$ so that
\[
x^z_i=x_i+c_iI_{s_d+d},
\]
where $c_i>0$, $I_{s_d+d}=(1,...,1)^T\in\RR^{s_d+d}$, and $x_i(B_1)\cap x_i^z(B_1)=\emptyset$. Denote $c=\mbox{max}_{i=1}^L\{c_i\}$ and further set
\[
x^z_i=x_i+cI_{s_d+d}
\]
for all $i$. By choosing $x_i$ and $x^z_i$ in this way, we have embedded $V_i$ and $z\circ V_i$ simultaneously into the same Euclidean space. Note that
\[
g=h_i^*(h_i^{-1})^*g=\ud(x_i\circ h_i)^2
\]
on $V_i$ and
\[
g=(h^z_i)^*((h^z_i)^{-1})^*g=\ud(x^z_i\circ h^z_i)^2
\]
on $z\circ V_i$. Thus, by defining $\xi_i=x_i\circ h_i$ and $\xi^z_i=x^z_i\circ h^z_i$, and applying the partition of unity with (\ref{definition:partition_of_unity_over_A}), we have the results.
\end{proof}

\begin{theorem}\label{mainthm}
Any smooth, closed manifold $(\MM,g)$ with free isometric $\ZZ_2$ action admits a smooth symmetric, isometric embedding in $\RR^p$ for some $p\in\NN$.
\end{theorem}
\begin{proof}
By the remark following Lemma \ref{Z} and \ref{PHI} we have a smooth embedding $u:\MM\hookrightarrow \RR^{N}$ so that $g-\ud u^2>0$, where $N=s_{L(d+1)}+L(d+1)$. By Lemma \ref{Y}, with atlas $\mathcal{A}$ fixed we have
\[
g-\ud u^2=\sum_j \eta_j^2\ud \xi_j^2+\sum_j (\eta^z_j)^2(\ud \xi^z_j)^2.
\]
where $\xi^z_i-\xi_i=cI_{s_d+d}$. Denote $c=\frac{(2\ell+1)\pi}{\lambda}$, where $\lambda$ and $\ell$ will be determined later. To ease the notion, we define
\[
\gamma_i(x)=\left\{
\begin{array}{ll}
\xi_i(x) & \mbox{ when }x\in N_i\\
\xi^z_i(x) & \mbox{ when }x\in g\circ N_i
\end{array}\right..
\]
Then by the definition (\ref{definition:partition_of_unity_over_A}) we have
\[
g-\ud u^2=\sum_{j=1}^L \psi_j^2\ud \gamma_j^2.
\]
Given $\lambda>0$ we can define the following map $u_\lambda:\MM\rightarrow \RR^{2L}$:
\[
u_\lambda=\left(\frac{1}{\lambda}\psi_i\cos\left(\lambda\gamma_i\right),\frac{1}{\lambda}\psi_i\sin\left(\lambda\gamma_i\right)\right)_{i=1}^{L},
\]
where $\cos\left(\lambda\gamma_i\right)$ means taking cosine on each entry of $\lambda\gamma_i$.
Set $\ell$ so that $\frac{(2\ell+1)\pi}{\lambda}>1$ and we claim that $u_\lambda$ is a symmetric map. Indeed,
\begin{align*}
\begin{split}
\psi_i(z\circ x)\cos\left(\lambda\gamma_i(z\circ x)\right)
=\psi_i(x)\cos\left(\lambda\left(\gamma_i(x)+\frac{(2\ell+1)\pi}{\lambda}\right)\right)
=-\psi_i(x)\cos\left(\lambda\gamma_i(x)\right).
\end{split}
\end{align*}
and
\begin{align*}
\begin{split}
\psi_i(z\circ x)\sin\left(\lambda\gamma_i(z\circ x)\right)
=\psi_i(x)\sin\left(\lambda\left(\gamma_i(x)+\frac{(2\ell+1)\pi}{\lambda}\right)\right)
=-\psi_i(x)\sin\left(\lambda\gamma_i(x)\right).
\end{split}
\end{align*}
Direct calculation gives us
\[
g-\ud u^2=\ud u_\lambda^2-\frac{1}{\lambda^2}\sum_{j=1}^L\ud \psi_j^2.
\]
We show that when $\lambda$ is big enough, there exists a smooth symmetric embedding $w$ so that
\begin{equation}\label{pde1}
\ud w^2=\ud u^2-\frac{1}{\lambda^2}\sum_i^L\ud \psi_i^2.
\end{equation}
Since for all $\lambda>0$ we can find a $\ell$ so that $u_\lambda$ is a symmetric map without touching $\psi_i$, we can thus choosing $\lambda$ as large as possible so that (\ref{pde1}) is solvable. The solution $w$ provides us with a symmetric isometric embedding $(w,u_\lambda):\MM\hookrightarrow \RR^{N+2L}$ so that we have
\[
g=\ud u_\lambda^2+\ud w^2.
\]
Now we solve (\ref{pde1}). Fix $V_i$ and its relative $p\in V_i$. Suppose $w=u+a^2v$ is the solution where $a\in C^\infty_c(V_i)$ with $a=1$ on $\mbox{supp}\eta$. We claim if $\epsilon:=\lambda^{-1}$ is small enough we can find a smooth map $v:N_i\rightarrow \RR^N$ so that Equation \ref{pde1} is solved on $V_i$.

Equation \ref{pde1} can be written as
\begin{equation}\label{pde1x}
\ud (u+a^2v)^2=\ud u^2-\frac{1}{\lambda^2}\sum_i^L\ud \psi_i^2.
\end{equation}
which after expansion is
\begin{align}\label{pde2}
\begin{split}
\partial_j(a^2\partial_iu\cdot v)+\partial_i(a^2\partial_ju\cdot v)-2a^2\partial_{ij}u\cdot v&+a^4\partial_iv\cdot\partial_jv+\partial_i(a^3\partial_ja|v|^2)+\partial_j(a^3\partial_ia|v|^2)\\
&=-\frac{1}{\lambda^2}\ud \psi_i^2+2a^2(\partial_ia\partial_ja+a\partial_{ij}a)|v|^2
\end{split}
\end{align}
To simplify this equation we will solve the following Dirichlet problem:
\[
\left\{
\begin{array}{l}
\Delta(a\partial_iv\cdot\partial_jv)=\partial_i(a\Delta v\cdot\partial_jv)+\partial_j(a\Delta v\cdot\partial_iv)+r_{ij}(v,a)\\
a\partial_iv\cdot\partial_jv|_{\partial V_i}=0
\end{array}
\right.
\]
where
\[
r_{ij}=\Delta a\partial_iv\cdot \partial_jv-\partial_ja\Delta v\cdot \partial_jv-\partial_ja\partial_iv\Delta v+2\partial_\ell a\partial_\ell(\partial_iv\cdot\partial_jv)+2a(\partial_{i\ell}v\cdot\partial_{j\ell}v-\Delta v\cdot\partial_{ij}v)
\]
By solving this equation and multiplying it by $a^3$, we have
\begin{align}\label{simplify1}
\begin{split}
a^4\partial_iv\cdot\partial_jv=&\partial_i(a^3\Delta^{-1}(a\Delta v\cdot\partial_jv))+\partial_j(a^3\Delta^{-1}(a\Delta v\cdot\partial_iv))-3a^2\partial_ia\Delta^{-1}(a\Delta v\cdot\partial_jv)\\&-3a^2\partial_ja\Delta^{-1}(a\Delta v\cdot\partial_iv)+a^3\Delta^{-1}r_{ij}(v,a)
\end{split}
\end{align}
Plug Equation (\ref{simplify1}) into Equation (\ref{pde2}) we have
\begin{align}\label{pde3}
 \partial_j(a^2\partial_iu\cdot v-a^2N_i(v,a))+\partial_i(a^2\partial_ju\cdot v&-a^2N_j(v,a))-2a^2\partial_{ij}u\cdot v =-\frac{1}{\lambda^2}\ud \psi_i^2-2a^2M_{ij}(v,a) ,
\end{align}
where for $i,j=1,...d$
\[
\left\{
\begin{array}{l}
N_i(v,a)=-a\Delta^{-1}(a\Delta v\cdot \partial_iv)-a\partial_ia|v|^2\\
M_{ij}(v,a)=\frac{1}{2}a\Delta^{-1}r_{ij}(v,a)-(a\partial_{ij}a+\partial_ia\partial_ja)|v|^2-\frac{3}{2}(\partial_ia\Delta^{-1}(a\Delta v\cdot\partial_jv))+\partial_ja\Delta^{-1}(a\Delta v\cdot\partial_iv)
\end{array}\right.
\]
Note that by definition and the regularity theory of elliptic operator, we know both $N_i(\cdot,a)$ and $M_{ij}(\cdot,a)$ are maps in $C^\infty(V_i)$. We will solve Equation (\ref{pde3}) through solving the following differential system:
\begin{equation}\label{pde4}
\left\{
\begin{array}{lll}
\partial_iu\cdot v&=&N_i(v,a)\\
\partial_{ij}u\cdot v&=&-\frac{1}{\lambda^2}\ud \psi_i^2-M_{ij}(v,a).
\end{array}
\right.
\end{equation}
Since by construction we know $u$ has linearly independent $\partial_iu$ and $\partial_{ij}u$, $i,j=1,...,d$, we can solve the under-determined linear system (\ref{pde4}) by
\begin{equation}\label{pde5}
v=E(u)F(v,h),
\end{equation}
where
\[
E(u)=\left[\left(\begin{array}{c}\partial_iu\\ \partial_{ij}u\end{array}\right)^T\left(\begin{array}{c}\partial_iu\\ \partial_{ij}u\end{array}\right)\right]^{-1}\left(\begin{array}{c}\partial_iu\\ \partial_{ij}u\end{array}\right)^T
\]
and
\[
F(v,\epsilon)=\left(N_i(v,a),-\frac{1}{\lambda^2}\ud \psi_i^2-M_{ij}(v,a)\right)^T=\left(N_i(v,a),-\epsilon^2\ud \psi_i^2-M_{ij}(v,a)\right)^T.
\]
Next we will apply contraction principle to show the existence of the solution $v$. Substitute $v=\mu v'$ for some $\mu\in\RR$ to be determined later. By the fact that $N_i(0,a)=0$ and $M_{ij}(0,a)=0$ we can rewrite Equation (\ref{pde5}) as
\[
w=\mu E(u)F(v',0)+\frac{1}{\mu}E(u)F(0,\epsilon).
\]
Set
\[
\Sigma=\left\{w\in C^{2,\alpha}(V_i,\RR^N); \|w\|_{2,\alpha}\leq 1\right\}
\]
and
\[
Tw=\mu E(u)F(v',0)+\frac{1}{\mu}E(u)F(0,\epsilon).
\]
By taking
\[
\mu=\left(\frac{\|E(u)F(0,\epsilon)\|_{2,\alpha}}{\|E(u)\|_{2,\alpha}}\right)^{1/2},
\]
we have
\[
\|Tw\|_{2,\alpha}\leq\mu\|E(u)\|_{2,\alpha}\|F(v',0)\|_{2,\alpha}+\frac{1}{\mu}\|E(u)F(0,\epsilon)\|_{2,\alpha}=C_1(\|E(u)\|_{2,\alpha}\|E(u)F(0,\epsilon)\|_{2,\alpha})^{1/2},
\]
where $C_1$ depends only on $\|a\|_{4,\alpha}$. Thus $T$ maps $\Sigma$ into $\Sigma$ if $\|E(u)\|_{2,\alpha}\|E(u)F(0,\epsilon)\|_{2,\alpha}\leq 1/C_1^2 $. This can be achieved by taking $\epsilon$ small enough, that is, by taking $\lambda$ big enough.

Similarly we have
\begin{align*}
 \|Tw_1-Tw_2\|_{2,\alpha}\leq &\,\mu\|E(u)\|_{2,\alpha}\|F(w_1,0)-F(w_2,0)\|_{2,\alpha}\\
\leq\,& C_2\|w_1-w_2\|_{2,\alpha}(\|E(u)\|_{2,\alpha}\|E(u)F(0,\epsilon)\|_{2,\alpha})^{1/2}.
\end{align*}
Then if $\|E(u)\|_{2,\alpha}\|E(u)F(0,\epsilon)\|_{2,\alpha}\leq \frac{1}{C_1^2+C_2^2}$ we show that $T$ is a contraction map. By the contraction mapping principle, we have a solution $v\in\Sigma$.

Further, since we have
\[
v=\mu^2 E(u)F(w,0)+E(u)F(0,\epsilon),
\]
by definition of $\mu$ we have
\[
\|v\|_{2,\alpha}\leq C\|E(u)F(0,\epsilon)\|_{2,\alpha},
\]
where $C$ is independent of $u$ and $v$. Thus by taking $\epsilon$ small enough, we can not only make $w=u+a^2v$ satisfy Equation (\ref{pde1}) but also make $w$ an embedding. Thus we are done with the patch $V_i$.

Now we take care $V_i$'s companion $z\circ V_i$. Fix charts around $x\in V_i$ and $z\circ x\in z\circ V_i$ so that $y\in V_i$ and $g\circ y\in z\circ V_i$ have the same coordinates for all $y\in V_i$. Working on these charts we have
\[
\partial_ju=\partial_j(\Phi\circ \tilde{u})=\partial_\ell\Phi\partial_j\tilde{u}^\ell
\]
and
\[
\partial_{ij}u=\partial_{ij}(\Phi\circ \tilde{u})=\partial_{k\ell}\Phi\partial_i \tilde{u}^k\partial_j \tilde{u}^\ell+\partial_\ell\Phi\partial_{ij}\tilde{u}^\ell.
\]
Note that since the first derivative of $\Phi$ is an even function while the second derivative of $\Phi$ is an odd function and $\tilde{u}(g\circ y)=-\tilde{u}(y)$ for all $y\in N_i$, we have
\[
E(u)(z\circ x)=-E(u)(x).
\]
Moreover, we have $N_i(v,a)=N_i(-v,a)$ and $M_{ij}(v,a)=M_{ij}(-v,a)$ for all $i,j=1,...,d$. Thus in $g\circ N_i$, we have $-v$ as the solution to Equation (\ref{pde4}) and $w-a^2v$ as the modified embedding. After finishing the perturbation of $V_i$ and $z\circ V_i$, the modified embedding is again symmetric.

Inductively we can perturb the embedding of $V_i$ for all $i=1,...,L$. Since there are only finite patches, by choosing $\epsilon$ small enough, we finish the proof.
\end{proof}
Note that we do not show the optimal dimension $p$ of the embedded Euclidean space but simply show the existence of the symmetric isometric embedding. How to take the symmetry into account in the optimal isometric embedding will be reported in the future work.
\begin{corollary}
Any smooth, closed non-orientable manifold $(\MM,g)$ has an orientable double covering embedded symmetrically inside $\RR^p$ for some $p\in\NN$.
\end{corollary}
\begin{proof}
It is well known that the orientable double covering of $\MM$ has isometric free $\ZZ_2$ action. By applying Theorem \ref{mainthm} we get the result.
\end{proof}

\bibliographystyle{imaiai}
\bibliography{spectral}

\end{document}